\documentclass[a4paper]{amsart}
\usepackage[all]{xy}
\usepackage{ifthen}
\usepackage{amssymb}

\SelectTips{eu}{}
\calclayout

\title[Coherent sheaves on weighted projective lines]{Introduction to
coherent sheaves on weighted projective lines}

\thanks{Version from September 18, 2010.}

\author[Xiao-Wu Chen]{Xiao-Wu Chen}
\address{Xiao-Wu Chen\\ Department of Mathematics\\ University of Science and
Technology of China\\ Hefei 230026, Anhui\\ PR China.}
\email{xwchen@ustc.edu.cn}

\author[Henning Krause]{Henning Krause}
\address{Henning Krause\\ Fakult\"{a}t f\"ur Mathematik\\
Universit\"at Bielefeld\\ D-33501 Bielefeld\\ Germany.}
\email{hkrause@math.uni-bielefeld.de}

\theoremstyle{plain}
\newtheorem{lem}{Lemma}[subsection]
\newtheorem{prop}[lem]{Proposition}
\newtheorem{cor}[lem]{Corollary}
\newtheorem{thm}[lem]{Theorem}

\theoremstyle{remark}

\theoremstyle{definition}
\newtheorem{rem}[lem]{Remark}
\newtheorem{exm}[lem]{Example}

\numberwithin{equation}{subsection}

\DeclareMathOperator{\add}{add}

\DeclareMathOperator{\Aut}{Aut}
\DeclareMathOperator{\coh}{coh}
\DeclareMathOperator{\Coker}{Coker}
\DeclareMathOperator{\End}{End}

\DeclareMathOperator{\disc}{disc}
\DeclareMathOperator{\Hom}{Hom}
\DeclareMathOperator{\Ext}{Ext}

\DeclareMathOperator{\Ker}{Ker}
\DeclareMathOperator{\rad}{rad}
\DeclareMathOperator{\Tor}{Tor}
\DeclareMathOperator{\PGL}{PGL}
\DeclareMathOperator{\Sq}{Sq}
\DeclareMathOperator{\rank}{rank}
\DeclareMathOperator{\card}{card}
\DeclareMathOperator{\id}{id}
\DeclareMathOperator{\Id}{Id}

\DeclareMathOperator{\proj}{proj}
\DeclareMathOperator{\Proj}{Proj}
\DeclareMathOperator{\pdim}{proj.dim}
\DeclareMathOperator{\idim}{inj.dim}

\DeclareMathOperator{\soc}{soc}
\DeclareMathOperator{\Spec}{Spec}
\DeclareMathOperator{\rep}{rep}
\DeclareMathOperator{\Supp}{supp}

\DeclareMathOperator{\gldim}{gl.dim}
\DeclareMathOperator{\tor}{tor}
\DeclareMathOperator{\ind}{ind}
\DeclareMathOperator{\vect}{vect}

\renewcommand{\mod}{\operatorname{mod}}
\renewcommand{\Im}{\operatorname{Im}}
\renewcommand{\top}{\operatorname{top}}
\renewcommand{\dim}{\operatorname{dim}}

\newcommand{\RHom}{\operatorname{{\bfR}Hom}}

\newcommand{\colim}[1]{\mathop{\mathrm{colim}}\limits_{#1}}

\newcommand{\lto}[1][{}]{\stackrel{#1}{\longrightarrow}}
\newcommand{\rto}[1][{}]{\stackrel{#1}{\longleftarrow}}
\newcommand{\xto}{\xrightarrow}
\newcommand{\smatrix}[1]{\left[\begin{smallmatrix}#1\end{smallmatrix}\right]}

\newcommand{\op}{\mathrm{op}}

\newcommand{\can}{\mathrm{can}}
\newcommand{\qis}{\mathrm{qis}}

\def\a{\alpha}
\def\b{\beta}
\def\e{\varepsilon}
\def\d{\delta}
\def\g{\gamma}
\def\p{\phi}
\def\r{\rho}
\def\s{\sigma}

\def\la{\lambda}

\def\Ga{\Gamma}
\def\La{\Lambda}
\def\Si{\Sigma}

\def\A{{\mathcal A}}
\def\B{{\mathcal B}}
\def\C{{\mathcal C}}

\def\F{{\mathcal F}}
\def\G{{\mathcal G}}

\def\L{{\mathcal L}}

\def\N{{\mathcal N}}
\def\Oc{{\mathcal O}}

\def\Y{{\mathcal Y}}

\def\T{{\mathcal T}}

\def\bbA{\mathbb A}

\def\bbP{\mathbb P}
\def\bbQ{\mathbb Q}
\def\bbX{\mathbb X}

\def\bbZ{\mathbb Z}

\def\bfC{\mathbf C}
\def\bfD{\mathbf D}

\def\bfK{\mathbf K}
\def\bfL{\mathbf L}
\def\bfp{\mathbf p}

\def\bfR{\mathbf R}

\def\bfX{\mathbf X}

\def\fra{\mathfrak a}
\def\frb{\mathfrak b}
\def\frm{\mathfrak m}
\def\frp{\mathfrak p}
\def\frq{\mathfrak q}

\newcommand{\bflambda}{\boldsymbol\lambda}

\newcommand{\ox}{\vec{x}}

\newcommand{\oc}{\vec{c}}

\makeindex

\begin{document}

\begin{abstract}
These notes provide a description of the abelian categories that arise
as categories of coherent sheaves on weighted projective lines. Two
different approaches are presented: one is based on a list of axioms and
the other yields a description in terms of expansions of abelian
categories.

A weighted projective line is obtained from a projective line by
inserting finitely many weights. So we describe the category of
coherent sheaves on a projective line in some detail, and the
insertion of weights amounts to adding simple objects. We call this
process `expansion' and treat it axiomatically. Thus most of these
notes are devoted to studying abelian categories, including a brief
discussion of tilting theory. We provide many details and have tried
to keep the exposition as self-contained as possible.
\end{abstract}

\maketitle

\setcounter{tocdepth}{1}
\tableofcontents

\section*{Introduction}\label{se:intro}

We begin with a brief description of weighted projective lines and
their categories of coherent sheaves.

Let $k$ be an algebraically closed field, let $\bbP^1_k$ be the
projective line over $k$, let $\bflambda=(\la_1,\dots,\la_n)$ be a
(possibly empty) collection of distinct closed points of $\bbP^1_k$,
and let $\bfp=(p_1,\dots,p_n)$ be a \index{weight sequence}
\emph{weight sequence}, that is, a sequence of positive integers.  The
triple $\bbX = (\bbP^1_k,\bflambda,\bfp)$ is called a \index{weighted
projective line} \emph{weighted projective line}.  Geigle and Lenzing
\cite{GL1987} have associated to each weighted projective line a
category $\coh\bbX$ of coherent sheaves on $\bbX$, which is the
quotient category of the category of finitely generated
$\bfL(\bfp)$-graded $S(\bfp,\bflambda)$-modules, modulo the Serre
subcategory of finite length modules. Here $\bfL(\bfp)$ is the rank 1
additive group
\[
\bfL(\bfp) = \langle \ox_1,\dots,\ox_n,\oc \mid p_1 \ox_1 =
\dots = p_n \ox_n = \oc \rangle,
\]
and
\[
S(\bfp,\bflambda) = k[u,v,x_1,\dots,x_n] / ( x_i^{p_i} +  \la_{i1} u
- \la_{i0} v ),
\]
with grading $\deg u = \deg v = \oc$ and $\deg x_i = \ox_i$, where
$\la_i = [\la_{i0}:\la_{i1}]$ in $\bbP^1_k$.  Geigle and Lenzing
showed that $\coh\bbX$ is a hereditary abelian category with finite
dimensional Hom and Ext spaces.  The free module $S(\bfp,\bflambda)$
yields a structure sheaf $\Oc$, and shifting the grading gives twists
$E(\ox)$ for any sheaf $E$ and $\ox\in\bfL(\bfp)$.

Every sheaf is the direct sum of a torsion-free sheaf and a finite
length sheaf. A torsion-free sheaf has a finite filtration by line
bundles, that is, sheaves of the form $\Oc(\ox)$. The finite length
sheaves are easily described as follows.  There are simple sheaves
$S_x$ ($x\in \bbP^1_k\smallsetminus \bflambda$) and $S_{ij}$ ($1\le
i\le n$, $1\le j\le p_i$) satisfying for any $r\in \mathbb{Z}$ that
$\Hom(\Oc(r\oc),S_{ij})\neq 0$ if and only if $j=1$, and the only
extensions between them are
\[
\Ext^1(S_x,S_x) =k, \quad \Ext^1(S_{ij},S_{ij'}) = k \quad (j'
\equiv j-1 \, (\mod \, p_i)).
\]
For each simple sheaf $S$ and $l>0$ there is a unique sheaf with
length $l$ and top $S$, which is \emph{uniserial}, meaning that it
has a unique composition series.  These are all the finite length
indecomposable sheaves.

Categories of the form $\coh\bbX$ for some weighted projective line
$\bbX$ play a special role in the study of abelian categories. This
follows from a theorem of Happel \cite{Ha} which we now explain.
Consider a connected hereditary abelian category $\A$ that is
$k$-linear with finite dimensional Hom and Ext spaces. Suppose in
addition that $\A$ admits a tilting object, that is some object $T$
with $\Ext^1_\A(T,T)=0$ such that $\Hom_\A(T,A)=0$ and
$\Ext^1_\A(T,A)=0$ imply $A=0$. Thus the functor
$\Hom_\A(T,-)\colon\A\to\mod\La$ into the category of modules over the
endomorphism algebra $\La=\End_\A(T)$ induces an equivalence
$$\bfD^b(\A)\lto[\sim]\bfD^b(\mod\La)$$ of derived categories. There
are two important classes of such hereditary abelian categories
admitting a tilting object: module categories over path algebras of
finite connected quivers without oriented cycles, and categories of
coherent sheaves on weighted projective lines. Happel's theorem then
states that there are no further classes. More precisely, an abelian
category $\A$ as above is, up to a derived equivalence, either of the
form $\mod k\Ga$ for some finite connected quiver $\Ga$ without
oriented cycles or of the form $\coh\bbX$ for some weighted projective
line $\bbX$.

The following treatment of coherent sheaves on weighted projective
lines is based on a list of axioms (extending the list in Happel's
theorem) which we postpone until \S\ref{se:cohX}. Before that we
discuss in some detail the necessary background material: abelian
categories, derived categories, tilting theory, expansions of abelian
categories, and coherent sheaves on $\bbP^1_k$.

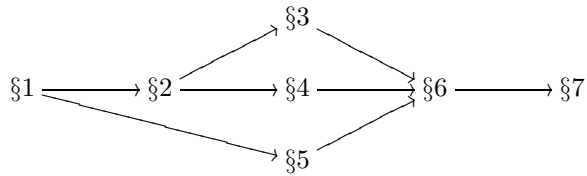
\begin{gather*}
\xymatrixrowsep{1pc} \xymatrixcolsep{3pc}
\xymatrix{
&&\S3\ar[rd]\\
\S1\ar[rrd]\ar[r]&\S2\ar[r]\ar[ru]&\S4\ar[r]&\S6\ar[r]&\S7\\
&&\S5\ar[ru]
}\\
\text{\centerline{\sc Leitfaden}}
\end{gather*}

\subsection*{Acknowledgements}
These notes are based on a seminar and a course held at the
University of Paderborn in the first half of 2009.  Both authors
wish to thank the participants for their interest and for
stimulating discussions related to the topic of these notes.  In
particular, we are grateful to Dirk Kussin and Helmut Lenzing for
their advice. It is a pleasure to thank Marco Angel
Bertani-{\O}kland  and David Ploog for their detailed comments.

Most of the material presented here is taken from the existing
literature. An exception is \S\ref{se:extensions}, where the concept
of an `expansion of abelian categories' is introduced. The previously
unpublished proof of Theorem~\ref{th:Gabriel} is due to Yu Ye.

\section{Abelian categories}

\subsection{Additive and abelian categories}
A category $\A$ is \emph{additive} \index{category!additive} if every finite family of objects
has a product, each morphism set $\Hom_\A(A,B)$ is an abelian group, and the
composition maps
$$\Hom_\A(A,B)\times\Hom_\A(B,C)\lto\Hom_\A(A,C)$$ are bilinear.
Given a finite number of objects $A_1,\dots,A_r$ of an additive
category $\A$, there exists a \emph{direct sum} \index{direct sum} $A_1\oplus\dots\oplus
A_r$, which is by definition an object $A$ together with morphisms
$\iota_i\colon A_i\to A$ and $\pi_i\colon A\to A_i$ for $1\leq i\leq
r$ such that $\sum_{i=1}^r\iota_i\pi_i=\id_A$,
$\pi_i\iota_i=\id_{A_i}$, and $\pi_j\iota_i=0$ for all $i\neq j$. Note that
the morphisms $\iota_i$ and $\pi_i$ induce isomorphisms
\[
\coprod_{i=1}^r A_i\cong\bigoplus_{i=1}^r A_i\cong \prod_{i=1}^r A_i.
\]
Given any object $A$ in $\A$, we denote by $\add A$ the full
subcategory of $\A$ consisting of all finite direct sums of copies of $A$ and
their direct summands.

A \emph{decomposition} \index{decomposition} $\A=\A_1\amalg\A_2$ of an
additive category $\A$ is a pair of full additive subcategories $\A_1$
and $\A_2$ such that each object in $\A$ is a direct sum of two
objects from $\A_1$ and $\A_2$, and
$\Hom_\A(A_1,A_2)=0=\Hom_\A(A_2,A_1)$ for all $A_1\in\A_1$ and
$A_2\in\A_2$.  An additive category $\A$ is \emph{connected}
\index{category!connected} if it admits no proper decomposition
$\A=\A_1\amalg\A_2$.

A functor $F\colon\A\to\B$ between additive categories is
\emph{additive} \index{functor!additive} if the induced map
$\Hom_\A(A,B)\to\Hom_\B(FA,FB)$ is linear for each pair of objects
$A,B$ in $\A$. The \emph{kernel} \index{kernel} $\Ker F$ of an
additive functor $F\colon\A\to\B$ is by definition the full
subcategory of $\A$ formed by all objects $A$ such that $FA=0$. The
\emph{essential image} \index{essential image} $\Im F$ of
$F\colon\A\to\B$ is the full subcategory of $\B$ formed by all objects
$B$ such that $B$ is isomorphic to $FA$ for some $A$ in $\A$.

An additive category $\A$ is \emph{abelian} \index{category!abelian}
if every morphism $\p\colon A\to B$ has a kernel and a cokernel, and
if the canonical factorization
$$\xymatrixrowsep{1.5pc} \xymatrixcolsep{1.5pc}
\xymatrix{\Ker\p\ar[r]^-{\p'}&A\ar[r]^-\p\ar[d]&
B\ar[r]^-{\p''}&\Coker\p\\
&\Coker\p'\ar[r]^-{\bar\p}&\Ker\p''\ar[u]}$$
of $\p$ induces  an isomorphism $\bar\p$.

Given an abelian category $\A$, a finite sequence of morphisms
$$A_1\lto[\p_1] A_2\lto[\p_2] \cdots \lto[\p_{n}] A_{n+1}$$ in $\A$ is
\index{sequence!exact}\emph{exact} if $\Im\p_i=\Ker\p_{i+1}$ for all
$1\le i< n$. An additive functor $F\colon\A\to\B$ between abelian categories
is \index{functor!exact}\emph{exact} if $F$ sends each exact sequence
in $\A$ to an exact sequence in $\B$.

\begin{exm}
(1) Let $\La$ be a right noetherian ring. The category $\mod\La$ of
finitely generated right modules over $\La$ is an abelian category.

(2) Let $k$ be a field and $\Ga$ a quiver. The category $\rep(\Ga,k)$ of
    finite dimensional $k$-linear representations of $\Ga$ is an abelian
    category.
\end{exm}

\subsection*{Conventions}
Throughout, all categories are supposed to be
\index{category!skeletally small} \emph{skeletally small}, unless
otherwise stated. This means that the isomorphism classes of objects
form a set. Subcategories are usually full subcategories and closed
under isomorphisms. Functors between additive categories are always
assumed to be additive. The composition of morphisms is written from
right to left, and modules over a ring are usually right modules.

\subsection{Serre subcategories and quotient categories}

Let $\mathcal{A}$ be an abelian category.  A non-empty full
subcategory $\C$ of $\A$ is called a \index{Serre subcategory}
\emph{Serre subcategory} provided that $\C$ is closed under taking
subobjects, quotients and extensions. This means that for every exact
sequence $0\to A'\to A\to A''\to 0$ in $\A$, the object $A$ belongs to
$\C$ if and only if $A'$ and $A''$ belong to $\C$.

\begin{exm}
The kernel of an exact functor $\A\to\B$ between abelian categories is
a Serre subcategory of $\A$.
\end{exm}

Given a Serre subcategory $\C$ of $\A$, the \index{quotient
category} \emph{quotient category} $\A/\C$ of $\A$ with respect to
$\C$ is defined as follows. The objects in $\A/\C$ are the objects
in $\A$. Given two objects $A,B$ in $\A$, there is for each pair of
subobjects $A'\subseteq A$ and $B'\subseteq B$ an induced map
$\Hom_\A(A,B)\to\Hom_\A(A',B/B')$.  The pairs $(A',B')$ such that
both $A/A'$ and $B'$ lie in $\C$ form a directed set, and one
obtains a direct system of abelian groups $\Hom_\A(A',B/B')$. We
define
$$\Hom_{\A/\C}(A, B)=\colim{(A', B')} \Hom_\A(A', B/B')$$ and the
composition of morphisms in $\A$ induces the composition in
$\A/\C$.\footnote{One needs to verify that $\A/\C$ is a category, in
particular that the composition of morphisms is associative. This
requires some work; see \cite{G,GZ}.}

The \index{quotient functor} \emph{quotient functor} $Q\colon\A\to
\A/\C$ is by definition the identity on objects. The functor takes a
morphism in $\Hom_{\A}(A, B)$ to its image under the canonical map
$\Hom_{\A}(A, B)\to\Hom_{\A/\C}(A, B)$.

\begin{lem}\label{le:mor}
Each morphism $A\to B$ in $\A/\C$ is of the form
\begin{equation}\label{eq:mor}
A\xto{(Q\iota)^{-1}} A'\xto{Q\p} B/B'\xto{(Q\pi)^{-1}}B
\end{equation}
for some pair $(A',B')$ of subobjects with $A/A'$ and $B'$ in $\C$ and
some morphism $\p\colon A'\to B/B'$ in $\A$, where $\iota\colon A'\to
A$ and $\pi\colon B\to B/B'$ denote the canonical morphisms in $\A$.
\end{lem}
\begin{proof}
For each morphism $A\to B$ in $\A/\C$, there is by definition a pair
$(A',B')$ of subobjects and a morphism $\p\colon A'\to B/B'$ in $\A$
such that the following diagram commutes.
$$\xymatrix{A\ar[r]&B\ar[d]^{Q\pi}\\
A'\ar[u]^{Q\iota}\ar[r]^-{Q\p}&B/B'}$$ Now observe that for each
object $C$ in $\A$ the inclusion $\iota\colon A'\to A$ induces a bijection
$\Hom_{\A/\C}(A,C)\to\Hom_{\A/\C}(A',C)$. Thus $Q\iota$ is invertible.
Analogously, one shows that $Q\pi$ is invertible.
\end{proof}

The following result summarizes the basic properties of a quotient
category and the corresponding quotient functor.

\begin{prop}\label{pr:quotient}
Let $\A$ be an abelian category and $\C$ a Serre subcategory.
\begin{enumerate}
\item The category $\mathcal{A}/\mathcal{C}$ is abelian and the
quotient functor $Q\colon\A\to \A/\C$ is exact with kernel $\Ker Q=\C$.
\item Let $F\colon\A\to\B$ be an exact functor between abelian
categories. If $\C\subseteq\Ker F$, then there is a unique functor
$\bar F\colon\A/\C\to\B$ such that $F=\bar FQ$. Moreover, the functor
$\bar F$ is exact.
\end{enumerate}
\end{prop}
\begin{proof}
(1) It follows from the construction that the morphism sets of the
quotient category $\A/\C$ are abelian groups. Also, the quotient
functor induces linear maps between the morphism sets and it
preserves finite direct sums. Thus the quotient category and the
quotient functor are both additive.

The quotient functor sends a morphism in $\A$ to the zero
morphism if and only if its image belongs to $\C$. Thus $\Ker
Q=\C$.

Let $\psi=(Q\pi)^{-1}Q\p(Q\iota)^{-1}$ be a morphism in $\A/\C$ as in
\eqref{eq:mor}. Denote by $\iota'\colon\Ker\p\to A'$ the kernel and by $\pi'\colon
B/B'\to\Coker\p$ the cokernel of $\p$ in $\A$. Then the kernel of
$\psi$ is $Q(\iota\iota')\colon \Ker\p\to A$, whereas the cokernel of
$\psi$ is $Q(\pi'\pi)\colon B\to\Coker\p$. It follows that the
category $\A/\C$ is abelian and that the quotient functor preserves
kernels and cokernels.

(2) The functor $\bar F\colon\A/\C\to\B$ takes an object $A$ to $FA$
and a morphism of the form $(Q\pi)^{-1}Q\p(Q\iota)^{-1}$ as in
\eqref{eq:mor} to $(F\pi)^{-1} F\p (F\iota)^{-1}$. Note that $F\iota$
and $F\pi$ are isomorphisms in $\B$, since $F$ is exact and
$\C\subseteq\Ker F$.

The functor $\bar F$ is additive and the description of (co)kernels in (1)
shows that $\bar F$ preserves (co)kernels. Thus $\bar F$ is exact.
\end{proof}

The quotient functor $\A\to\A/\C$ is the universal functor that
inverts the class $S(\C)$ of morphisms $\s$ in $\A$ with $\Ker\s$ and
$\Coker\s$ in $\C$. More precisely, for any class $S$ of morphisms in
$\A$, there exists a universal functor $P\colon \A\to\A[S^{-1}]$ such
that
\begin{enumerate}
\item the morphism $P\s$ is invertible for every $\s\in S$, and
\item every functor $F\colon\A\to \B$ such that $F\s$ is invertible
for each $\s\in S$ admits a unique functor $\bar F\colon\A[S^{-1}]\to
\B$ such that $F=\bar F P$.
\end{enumerate}
The category $\A[S^{-1}]$ is the \index{localization} \emph{localization} of $\A$ with
respect to $S$ and is unique up to a unique isomorphism; see
\cite[I.1]{GZ}.

\begin{lem}\label{le:Serre}
Let $\C$ be a Serre subcategory of $\A$. The quotient functor
$Q\colon\A\to \A/\C$ is the universal functor that inverts all
morphisms in $S(\C)$. Therefore $$\A[S(\C)^{-1}]=\A/\C.$$
\end{lem}
\begin{proof}
From Proposition~\ref{pr:quotient} it follows that $Q$ inverts all morphisms
in $S(\C)$. Now let $F\colon\A\to \B$ be a functor such that $F\s$ is
invertible for each $\s\in S(\C)$. Then for each pair $A,B$ of objects
in $\A$ and each pair of subobjects $A'\subseteq A$ and $B'\subseteq
B$ with $A/A'$ and $B'$ in $\C$, the map
$\Hom_\A(A,B)\to\Hom_\B(FA,FB)$ factors through the canonical map
$\Hom_\A(A,B)\to\Hom_\A(A',B/B')$. Thus there are induced maps
$\Hom_{\A/\C}(A,B)\to\Hom_\B(FA,FB)$ which induce a unique functor
$\bar F\colon A/\C\to \B$ such that $F=\bar F Q$. It follows that
$\A[S(\C)^{-1}]=\A/\C$.
\end{proof}

\begin{exm}
(1) Let $\La$ be a commutative noetherian ring and $\La_\mathfrak p$
the localization with respect to a prime ideal $\mathfrak p$. The
localization functor $T\colon\mod\La \to\mod\La_\mathfrak p$ sending
a $\La$-module $M$ to $M_\mathfrak p=M\otimes_\La\La_\mathfrak p$ is
exact and induces an equivalence $\mod\La/\Ker
T\xto{\sim}\mod\La_\mathfrak p$. Roughly speaking, restriction of
scalars along the morphism $\La\to\La_\mathfrak p$ yields a
quasi-inverse.

(2) Let $\La$ be a right noetherian ring and $e^2=e\in\La$ an
idempotent. The functor $T\colon\mod\La\to\mod e\La e$ sending a
$\La$-module $M$ to $Me=M\otimes_\La \La e$ is exact. The kernel
$\Ker T$ identifies with $\mod\La/\La e\La$ and $T$ induces an
equivalence $\mod\La/\Ker T\xto{\sim}\mod e\La e$.  The functor
$\Hom_{e\La e}(\La e,-)$ yields a quasi-inverse.

(3) Let $\La$ be a right artinian ring. Given a set $S_1,\dots, S_n$
of simple $\La$-modules, the $\La$-modules $M$ having a finite
filtration $0=M_0\subseteq M_1\subseteq\dots \subseteq M_r=M$ with
each factor $M_i/M_{i-1}$ isomorphic to one of the simples
$S_1,\dots, S_n$ form a Serre subcategory of $\mod\La$. Moreover,
each Serre subcategory of $\mod\La$ arises in this way and is
therefore of the form $\mod \La/\La e\La$ for some idempotent
$e\in\La$.
\end{exm}

\subsection{Properties of quotient categories}
We collect some further properties of abelian quotient categories.

\begin{lem}\label{le:small}
Let $\A$ be an abelian category that is not supposed to be
skeletally small, and let $\C$ be a Serre subcategory.  Then the
following are equivalent:
\begin{enumerate}
\item The category $\A$ is skeletally small.
\item The categories $\C$ and $\A/\C$ are skeletally small. In
addition, $\Ext^1_\A(A,C)$ and $\Ext^1_\A(C,A)$ are sets for all
$A\in\A$ and $C\in\C$.
\end{enumerate}
\end{lem}
\begin{proof}
One direction is clear. So suppose that $\C$ and $\A/\C$ are
skeletally small, and that extensions with objects in $\C$ form sets.
First observe that the morphisms in $\Hom_{\A/\C}(A,B)$ form a set for
each pair of objects $A,B$. Here one uses that the subobjects
$A'\subseteq A$ with $A'$ or $A/A'$ in $\C$ form, up to isomorphism, a
set, since $\C$ is skeletally small.  Next observe that for each
object $A$ in $\A$, there is only a set of isomorphism classes of
objects $B$ with $A\cong B$ in $\A/\C$. This follows from the fact
that each isomorphism $A\to B$ in $\A/\C$ is represented by a chain
$$A \leftarrowtail A'\twoheadrightarrow I\rightarrowtail
B/B'\twoheadleftarrow B$$ of epis and monos in $\A$ with kernel and
cokernel in $\C$; see Lemma~\ref{le:mor}. Here one uses that $\C$ is
skeletally small and that extensions with objects in $\C$ form sets.
From this it follows that the isomorphism classes of objects in $\A$
form a set, since the quotient $\A/\C$ has this property.
\end{proof}

The following example gives an abelian category $\A$ with a Serre
subcategory $\C$ such that $\C$ and $\A/\C$ are skeletally small but
$\A$ itself is not.

\begin{exm}
Let $k$ be a field and $\Ga$ a quiver with set of vertices $\{1,2\}$
and a proper class of arrows $1\to 2$.  Each arrow of $\Ga$
corresponds to a canonical element of $\Ext^1(S_1,S_2)$, where $S_i$
denotes the simple representation supported at the vertex $i$. These
extensions are linearly independent and yield pairwise non-isomorphic
two-dimensional representations.  The functor $T\colon\rep(\Ga,k)\to
\mod k$ sending a representation of $\Ga$ to the corresponding vector space at
vertex $1$ induces an equivalence $\rep(\Ga,k)/\Ker T\xto{\sim}\mod
k$, and $\Ker T$ is equivalent to $\mod k$.
\end{exm}

An abelian category is called \index{category!noetherian}
\emph{noetherian} if each of its objects is \emph{noetherian} (i.e.\
satisfies the ascending chain condition on subobjects).

\begin{lem}\label{le:noetherian}
Let $\A$ be an abelian category and $\C$ a Serre subcategory. Then
the following are equivalent:
\begin{enumerate}
\item The category $\A$ is noetherian.
\item The categories $\C$ and $\A/\C$ are noetherian, and each object
in $\A$ has a largest subobject that belongs to $\C$.
\end{enumerate}
\end{lem}
\begin{proof}
(1) $\Rightarrow$ (2): Suppose $\A$ is noetherian. Then $\C$ is
noetherian. Also, $\A/\C$ is noetherian because each ascending chain
of subobjects in $\A/\C$ can be represented by an ascending chain of
subobjects in $\A$; see Lemma~\ref{le:mor}. Noetherianess implies that
each non-empty set of subobjects has a maximal element. In particular,
each object has a subobject that is maximal among all subobjects
belonging to $\C$.

(2) $\Rightarrow$ (1): Let $A_1\subseteq A_2\subseteq \dots\subseteq
A$ be an ascending chain of subobjects in $\A$. Using that $\A/\C$ is
noetherian, there exists some integer $n$ such that $A_m/A_n$ belongs to $\C$
for all $m>n$. Let $\bar A$ be the maximal subobject of $A/A_n$
belonging to $\C$. Then the chain $A_{n+1}/A_n\subseteq
A_{n+2}/A_n\subseteq \dots \subseteq\bar A$ becomes stationary since
$\C$ is noetherian. It follows that the original chain of subobjects
of $A$ becomes stationary. Thus $A$ is noetherian.
\end{proof}

Next we give an example of an abelian category $\A$ with a Serre
subcategory $\C$ such that $\C$ and $\A/\C$ are noetherian but $\A$ itself is
not.

\begin{exm}
The ring $\La=\smatrix{\bbQ&\bbQ\\ 0&\bbZ}$ is well known to be left
but not right noetherian. Consider the abelian category $\mod\La$ of
finitely presented (right) $\La$-modules and the functor
$T\colon\mod\La\to\mod\bbQ$ sending a $\La$-module $M$ to $Me$ with
$e=\smatrix{1&0\\ 0&0}$. Then $\Ker T$ is equivalent to $\mod\bbZ$
and $T$ induces an equivalence $\mod\La/\Ker T\xto{\sim}\mod\bbQ$.
\end{exm}

The next lemma provides the analogue of a Noether isomorphism for
abelian categories. The proof is straightforward.

\begin{lem}\label{le:noeth}
\pushQED{\qed}
Let $\A$ be an abelian category and $\A_1,\A_2$ a pair of Serre
subcategories such that $\A_2\subseteq\A_1$. Then the following holds:
\begin{enumerate}
\item The inclusion $\A_1\to\A$ identifies $\A_1/\A_2$ with a Serre subcategory of $\A/\A_2$.
\item The quotient functor $\A\to\A/\A_2$ induces
an isomorphism \[\A/\A_1\lto[\sim](\A/\A_2)/(\A_1/\A_2).\qedhere\]
\end{enumerate}
\end{lem}

Recall that a non-zero object $S$ of an abelian category is
\index{object!simple}\emph{simple} if $S$ has no proper subobject
$0\neq U\subsetneq S$. The next lemma says that a quotient functor
preserves this property if the object is not annihilated.  The proof
is straightforward.

\begin{lem}\label{le:quotient_simple}
\pushQED{\qed}
Let $\A$ be an abelian category and $\C$ a Serre
subcategory. If $S$ is a simple object not belonging to $\C$, then $S$
is simple in $\A/\C$ and the quotient functor induces an isomorphism
$\End_\A(S)\xto{\sim}\End_{\A/\C}(S)$.
\qedhere
\end{lem}

\subsection{Perpendicular categories}

Let $\A$ be an abelian category.  In some cases, the quotient functor
$\A\to \A/\C$ with respect to a Serre subcategory $\C$ admits a right
adjoint. Then the perpendicular category $\C^\perp$ provides another
description of the quotient category $\A/\C$.

For any class $\C$ of objects in $\A$, its \index{perpendicular
category} \emph{perpendicular categories} are by definition the full
subcategories
\begin{align*}
\C^\perp&=\{A\in\A\mid\Hom_\A(C,A)=0=\Ext^1_\A(C,A)\text{ for all
}C\in\C\},\\
^\perp\C&=\{A\in\A\mid\Hom_\A(A,C)=0=\Ext^1_\A(A,C)\text{ for all
}C\in\C\}.
\end{align*}

\begin{lem}\label{le:perp}
Let $\A$ be an abelian category and $\C$ a Serre subcategory. Then the
following are equivalent for an object $B$ in $\A$:
\begin{enumerate}
\item The object $B$ belongs to $\C^\perp$.
\item The quotient functor induces a bijection
$\Hom_\A(A,B)\to\Hom_{\A/\C}(A,B)$ for every object $A$ in $\A$.
\item The map $\Hom_\A(\s,B)$ is bijective for every morphism $\s$ in
$\A$ with $\Ker\s$ and $\Coker\s$ in $\C$.
\end{enumerate}
\end{lem}

\begin{proof}
(1) $\Rightarrow$ (2): For each pair of subobjects $A'\subseteq A$ and
$B'\subseteq B$ such that both $A/A'$ and $B'$ lie in $\C$, the map
$\Hom_\A(A,B)\to\Hom_\A(A',B/B')$ is bijective. Thus the quotient
functor induces a bijection $\Hom_\A(A,B)\to\Hom_{\A/\C}(A,B)$.

(2) $\Rightarrow$ (3): The quotient functor sends a morphism $\s$ in
$\A$ with $\Ker\s$ and $\Coker\s$ in $\C$ to an isomorphism in
$\A/\C$. Thus $\Hom_{\A/\C}(\s,B)$ is bijective. Using the bijections
in (2) it follows that $\Hom_{\A}(\s,B)$ is bijective.

(3) $\Rightarrow$ (1): Let $C$ be an object in $\C$. Then $\s\colon
0\to C$ induces a bijection $\Hom_\A(\s,B)$. Thus
$\Hom_\A(C,B)=0$. Let $\xi\colon 0\to B\to E\to C\to 0$ an exact
sequence in $\A$. The morphism $\s\colon B\to E$ induces a bijection
$\Hom_\A(\s,B)$, and therefore $\xi$ splits.  Thus $\Ext^1_\A(C,B)=0$.
\end{proof}

We need the following elementary lemma about pairs of adjoint functors.

\begin{lem}\label{le:adjoint}
Let $F\colon\A\to\B$ and $G\colon\B\to\A$ be a pair of functors such
that $G$ is a right adjoint of $F$.
Denote by $S(F)$ the class of morphisms $\s$ in $\A$ such that $F\s$ is invertible.
Then the following are equivalent:
\begin{enumerate}
\item The functor $F$ induces an equivalence $\A[S(F)^{-1}]\xto{\sim}\B$.
\item The functor $G$ is fully faithful.
\item The adjunction morphism $F(GA)\to A$ is invertible for each $A\in\B$.
\end{enumerate}
\end{lem}

\begin{proof} See \cite[I.1.3]{GZ}.
\end{proof}

The next result provides a useful criterion for an exact functor to be
a quotient functor. Moreover, it describes the right adjoint of a quotient functor.

\begin{prop}\label{pr:perp}
Let $F\colon\A\to\B$ be an exact functor between abelian categories
and suppose that $F$ admits a right adjoint $G\colon\B\to\A$. Then the
following are equivalent:
\begin{enumerate}
\item The functor $F$ induces an equivalence $\A/\Ker F\xto{\sim}\B$.
\item The functor $F$ induces an equivalence $(\Ker F)^\perp\xto{\sim}\B$.
\item The functor $G$ induces an equivalence $\B\xto{\sim}(\Ker F)^\perp$.
\item The functor $G$ is fully faithful.
\end{enumerate}
Moreover, in that case $(\Ker F)^\perp=\Im G$ and  $\Ker F={^\perp}(\Im G)$.
\end{prop}
\begin{proof}
Let $S(F)$ denote the class of morphisms $\s$ in $\A$ such that $F\s$
is invertible. Then it follows from Lemma~\ref{le:Serre} that
$\A[S(F)^{-1}]=\A/\Ker F$.

(1) $\Rightarrow$ (2): The functor $F$ induces a full and faithful
functor $(\Ker F)^\perp\to\B$ by Lemma~\ref{le:perp}. For each
$B\in\B$, we have $F(G B)\cong B$ by Lemma~\ref{le:adjoint}, and $GB$
belongs to $(\Ker F)^\perp$ by Lemma~\ref{le:perp}. Thus $F$ induces
an equivalence $(\Ker F)^\perp\xto{\sim}\B$.

(2) $\Rightarrow$ (3): Note that $\Im G\subseteq (\Ker F)^\perp$. Thus
$G$ induces a functor $\B\to (\Ker F)^\perp$ which is a right adjoint
of the equivalence $(\Ker F)^\perp\to\B$.  Now one uses that an adjoint of an
equivalence is again an equivalence.

(3) $\Rightarrow$ (4): An equivalence is fully faithful.

(4) $\Rightarrow$ (1): Use Lemma~\ref{le:adjoint}.

Observe that (3) implies $(\Ker F)^\perp=\Im G$. In particular, $\Ker
F\subseteq{^\perp}(\Im G)$. The other inclusion follows from the
isomorphism $\Hom_\B(FA,B)\cong\Hom_\A(A,GB)$.
\end{proof}

The next result characterizes the fact that the quotient functor admits a right adjoint.

\begin{prop}\label{pr:section}
Let $\A$ be an abelian category and $\C$ a Serre subcategory. Then the
following are equivalent:
\begin{enumerate}
\item The quotient functor $\A\to\A/\C$ admits a right adjoint
$\A/\C\to\A$.
\item Every object $A$ in $\A$ fits into an exact sequence
\begin{equation}\label{eq:local}
0\lto A'\lto A\lto \bar A\lto A''\lto 0
\end{equation}
such that $A',A''\in\C$ and $\bar A\in\C^\perp$.
\item The quotient functor induces an equivalence $\C^\perp\xto{\sim}\A/\C$.
\end{enumerate}
In that case the functor $\A\to\C$ sending $A$ to $A'$ is a
right adjoint of the inclusion $\C\to\A$, and
the functor $\A\to\C^\perp$ sending $A$ to $\bar A$ is a left adjoint of
the inclusion $\C^\perp\to\A$.
\end{prop}
\begin{proof}
(1) $\Rightarrow$ (2): We apply Proposition~\ref{pr:perp}. Suppose
that the quotient functor $F\colon\A\to\A/\C$ admits a right adjoint
$G$. The functor $F$ inverts the adjunction morphism $\eta_A\colon
A\to G(FA)=\bar A$, since $FG\cong\Id_{\A/\C}$ by
Lemma~\ref{le:adjoint}. The exactness of $F$ then implies that
$A'=\Ker\eta_A$ and $A''=\Coker\eta_A$ belong to $\C$. The object
$\bar A$ belongs to $\Im G=\C^\perp$ by construction.

(2) $\Rightarrow$ (3): The quotient functor induces a fully faithful functor
$\C^\perp\to\A/\C$ by Lemma~\ref{le:perp}.  This functor is an
equivalence, because each object $A$ in $\A/\C$ is isomorphic to one
in its image via the isomorphism $A\xto{\sim}\bar A$.

(3) $\Rightarrow$ (1): Choose a quasi-inverse $G\colon
\A/\C\to\C^\perp$ of the equivalence $\C^\perp\to\A\xto{F}\A/\C$.  For
each $A$ in $\A$ and $B$ in $\A/\C$, there are bijections
$$\Hom_\A(A,GB)\xto{\sim}\Hom_{\A/\C}(FA,F(GB))\xto{\sim}\Hom_{\A/\C}(FA,
B).$$ The first map is bijective by Lemma~\ref{le:perp} and the second
is bijective because $FG\cong\Id_{\A/\C}$.  Thus $G$ is right
adjoint to the quotient functor $\A\to\A/\C$.

For any $C$ in $\C$, the induced map $\Hom_\A(C,A')\to\Hom_\A(C,A)$ is
bijective. Therefore sending $A$ to $A'$ provides a right adjoint of
the inclusion $\C\to\A$.  On the other hand, for any $B$ in
$\C^\perp$, the induced map $\Hom_\A(\bar A,B)\to\Hom_\A(A,B)$ is
bijective. Therefore sending $A$ to $\bar A$ provides a left adjoint
of the inclusion $\C^\perp\to\A$.
\end{proof}

\begin{rem}
The objects $A'$ and $A''$ occurring in \eqref{eq:local} represent
certain functors defined on $\C$. We have
$$\Hom_\A(-,A)|_\C\cong\Hom_\C(-,A')\quad\text{and}\quad\Ext^1_\A(-,A/A')|_\C\cong\Hom_\C(-,A''),$$
where  $A'$ is viewed as a subobject of $A$.
\end{rem}

\subsection{Global dimension}

Let $\A$ be an abelian category. For a pair of objects $A,B$ and
$n\geq 1$, let $\Ext^n_\A(A, B)$ denote the group of extensions in the
sense of Yoneda. Set $\Ext^0_\A(A, B)=\Hom_\A(A, B)$ and
$\Ext^{n}_\A(A, B)=0$ for $n<0$. Note that there are composition maps
$$\Ext^n_\A(A,B)\times\Ext^m_\A(B,C)\lto\Ext^{n+m}_\A(A,C)$$ for all
$n,m\in \mathbb{Z}$. The \index{dimension!projective}
\emph{projective dimension} of an object $A$ is by definition
$$\pdim A=\inf\{n\geq 0\mid \Ext_\A^{n+1}(A, -)=0\}.$$ Dually, one
defines the \index{dimension!injective} \emph{injective dimension} $\idim A$.

\begin{lem}\label{lem:hereditary}
Let $\A$ be an abelian category. For $A$ in $\A$ and $n>0$ the
following are equivalent:
\begin{enumerate}
\item $\Ext_\A^n(A, -)$ is right exact.
\item $\Ext_\A^{n+1}(A, -)=0$.
\item $\Ext_\A^{m}(A, -)=0$ for all $m>n$.
\item $\pdim A\leq n$.
\end{enumerate}
\end{lem}
\begin{proof}
It suffices to show that (1) and (2) are equivalent; the rest is straightforward. We use the long
exact sequence for $\Ext_\A^*(A,-)$.

(1) $\Rightarrow$ (2): Fix an element $\xi\in\Ext^{n+1}_\A(A,B)$ which
is represented by an exact sequence
$$0\lto B\lto E_{n+1}\lto E_n\lto \cdots\lto E_1\lto A\lto 0.$$ Let $C$ be the
image of $E_{n+1}\to E_{n}$ and write $\xi=\xi''\xi'$ as the composite
of extensions $\xi'\in\Ext^n_\A(A,C)$ and $\xi''\in\Ext^1_\A(C,B)$.
For the connecting morphism $\delta\colon
\Ext^n_\A(A,C)\to\Ext^{n+1}_\A(A,B)$ induced by $\xi''$, we have
$\delta=0$ since $\Ext_\A^n(A, -)$ is right exact. Thus
$\xi=\delta(\xi')=0$.

(2) $\Rightarrow$ (1): Clear.
\end{proof}

The \index{dimension!global} \emph{global dimension} of $\A$ is by
definition the smallest integer $n\geq 0$ such that
$\Ext_\A^{n+1}(-,-)=0$. As usual, the dimension is infinite if such a
number $n$ does not exist.  We denote this dimension by $\gldim\A$ and
observe that it is equal to $\sup\{\pdim A\mid A\in\A\}$ and
$\sup\{\idim A\mid A\in\A\}$.  The category $\A$ is called
\index{category!hereditary} \emph{hereditary} provided that
$\Ext_\A^{2}(-, -)=0$.

For a  right noetherian ring $\La$, the global dimension of the
module category $\mod \La$ is called the (right)
\index{dimension!global} \emph{global dimension} of $\La$ and
denoted by $\gldim\La$.

\begin{exm}
(1) Let $\La$ be the ring of integers $\bbZ$ or the polynomial ring
$k[x]$ over a field $k$. Then $\mod\La$ is hereditary. More generally,
$\mod\La$ is hereditary if $\La$ is a Dedekind domain.

(2) For a field $k$ and a quiver $\Ga$, the category of representations
$\rep(\Ga,k)$ is hereditary.

(3) Let $\A$ be a hereditary abelian category and $\C$ a Serre
subcategory.  Then $\C$ and $\A/\C$ are again hereditary.
\end{exm}

\subsection{Length categories}

Let $\A$ be an abelian category.  An object $A$ of $\A$ has
\index{object!finite length} \emph{finite length} if there exists a finite chain of subobjects
$$0=A_0\subseteq A_1\subseteq \dots \subseteq A_{n-1}\subseteq
A_n=A$$ such that each quotient $A_i/A_{i-1}$ is a simple object.
Such a chain is called a \index{composition series} \emph{composition
series} of $A$.  A composition series is not necessarily unique but
its length is an invariant of $A$ by the Jordan-H\"older theorem; it
is called the \index{length} \emph{length} of $A$ and is denoted by
$\ell(A)$.  Note that an object has finite length if and only if it is
both \index{object!artinian} \emph{artinian} (i.e.\ satisfies the
descending chain condition on subobjects) and
\index{object!noetherian} \emph{noetherian} (i.e.\ satisfies the
ascending chain condition on subobjects).

Every object of finite length decomposes essentially uniquely into a
finite direct sum of indecomposable objects with local endomorphism
rings. This follows from the Krull-Remak-Schmidt theorem.

The objects of finite length form a Serre subcategory of $\A$ which is
denoted by $\A_0$.  The abelian category $\A$ is called a
\index{length category} \emph{length category} if $\A=\A_0$.

Let $\A$ be a length category.  The \index{Ext-quiver}
\emph{Ext-quiver} or \index{Gabriel quiver} \emph{Gabriel quiver} of
$\A$ is a valued quiver $\Si=\Si(\A)$ which is defined as follows.
The set $\Si_0$ of vertices is a fixed set of representatives of the
isomorphism classes of simple objects in $\A$. For a simple object
$S$, let $\Delta(S)$ denote its endomorphism ring, which is a division
ring.  Observe that $\Ext_\A^1(S,T)$ carries a natural
$\Delta(T)$-$\Delta(S)$-bimodule structure for each pair $S,T$ in
$\Si_0$. There is an arrow $S\to T$ with valuation $\d_{S,T}=(s,t)$ in
$\Si$ if $\Ext_\A^1(S,T)\neq 0$ with
$s=\dim_{\Delta(S)}\Ext_\A^1(S,T)$ and
$t=\dim_{\Delta(T)^\op}\Ext_\A^1(S,T)$. We write $\d_{S,T}=(0,0)$ if
$\Ext_\A^1(S,T)=0$.

The following observation is easily proved.

\begin{lem}\label{le:comp}
\pushQED{\qed}
A length category is connected if and only if its Ext-quiver is
connected.\qedhere
\end{lem}

\begin{exm}
(1) Let $\La$ be a commutative noetherian local ring with maximal
    ideal $\mathfrak m$. Then $(\mod\La)_0$ equals the category of
    $\mathfrak m$-torsion modules and $\La/\mathfrak m$ is the unique
    simple $\La$-module.

(2) For a field $k$ and a quiver $\Ga$, the category of
representations $\rep(\Ga,k)$ is a length category.  Suppose that
$\Ga$ has no oriented cycle and no pair of parallel arrows. Then the
Ext-quiver of $\rep(\Ga,k)$ is isomorphic to $\Ga$, with valuation
$(1,1)$ for each arrow.
\end{exm}

\subsection{Uniserial categories}

Let $\mathcal{A}$ be a length category. An object $A$ is
\index{object!uniserial} \emph{uniserial} provided it has a unique
composition series. Note that any non-zero uniserial object is
indecomposable. Moreover, subobjects and quotient objects of uniserial
objects are uniserial.  The length category $\A$ is called
\index{category!uniserial} \emph{uniserial} provided that each
indecomposable object is uniserial.

The following result characterizes uniserial categories in terms of
their Ext-quivers.

\begin{thm}[Gabriel]\label{th:Gabriel}
A length category $\A$ is uniserial if and only if for each simple
object $S$, we have
\begin{equation}\label{eq:dim}
\sum_{S'\in \Si_0}\dim_{\Delta(S')}\Ext_\A^1(S', S)\leq 1 \text{ and
} \sum_{S'\in \Si_0}\dim_{\Delta(S')^\op}\Ext_\A^1(S, S')\leq 1.
\end{equation}
\end{thm}

The proof of this result requires some preparations and we begin with
some notation. Let $A$ be any object in $\A$. We denote by $\rad A$
the intersection of all its maximal subobjects and let $\top A=A/\rad
A$. Analogously, $\soc A$ denotes the sum of all simple subobjects of
$A$.

\begin{lem}\label{le:Gabriel}
Let $\A$ be a length category and suppose that \eqref{eq:dim} holds
for each simple object $S$.  Let $\xi\colon 0\rightarrow A \rightarrow
E\rightarrow S \rightarrow 0$ be a non-split extension such that $A$
is uniserial and $S$ is simple. Then we have the following:
\begin{enumerate}
\item Every epimorphism $A\to B\neq 0$ induces an isomorphism
$$\Ext^1_\A(S, A)\xto{\sim} \Ext^1_\A(S,B).$$
\item The object $E$ is uniserial.
\item Given any non-split extension $\xi'$ in $\Ext_\A^1(S,A)$, there exists an
isomorphism $\tau\colon A\rightarrow A$ such that $\xi'=\tau\xi$.
\end{enumerate}
\end{lem}
\begin{proof}
(1) It is sufficient to consider the case $B=\top A$. Moreover, it is
sufficient to show that the induced map $\Ext^1_\A(S, A)\to
\Ext^1_\A(S,\top A)$ is a monomorphism, since
$\dim_{\Delta(S)}\Ext^1_\A(S,\top A)\leq 1$ by \eqref{eq:dim}. We use
the long exact sequence which is obtained from the short exact
sequence $0\to A'\to A\to\top A\to 0$ by applying $\Hom_\A(S,-)$.

If $\Ext_\A^1(S, A')=0$, then the induced map $\Ext_\A^1(S, A)
\rightarrow \Ext_\A^1(S, \top A)$ is a monomorphism. Now assume that
$\Ext_\A^1(S, A')\neq 0$. Using induction on the length, we have that
$\Ext_\A^1(S, A')\xto{\sim} \Ext_\A^1(S, \top A')\neq 0$.  Next
observe that $\Ext_\A^1(\top A, \top A')\neq 0$ since $A$ is
uniserial. Thus \eqref{eq:dim} implies $S\cong \top A$, and a
dimension argument shows that the connecting morphism $\Hom_\A(S, \top A)
\rightarrow \Ext_\A^1(S, A')$ is an isomorphism. Thus from the long
exact sequence we infer that the natural map $\Ext_\A^1(S,
A)\rightarrow \Ext_\A^1(S, \top A)$ is a monomorphism.

(2) It suffices to show that every proper subobject $U\subseteq E$ is
contained in $A$. Otherwise we have an induced extension $0\rightarrow
A\cap U \rightarrow U\rightarrow S \rightarrow 0$. Thus the inclusion
$A\cap U\to A$ induces a non-zero map $\Ext_\A^1(S, A\cap U)
\rightarrow \Ext_\A^1(S, A)$.  Composing this with the isomorphism
$\Ext_\A^1(S, A)\xto{\sim} \Ext_\A^1(S, A/{A\cap U})$ from (1) gives a
non-zero map $\Ext_\A^1(S, A\cap U) \rightarrow {\rm Ext}_\A^1(S,
A/{A\cap U})$ which is induced by the composite $A\cap U\rightarrow A
\rightarrow A/{A\cap U}$. This is impossible.

(3) We use induction on the length of $A$. The case $\ell(A)=1$
follows from the equality $\dim_{\Delta(A)^\op}\Ext_\A^1(S,A)=1$. If
$\ell(A)>1$, choose a maximal subobject $A'\subseteq A$ and let $\bar
A=A/A'$. It follows from (1) that the canonical morphism $\pi\colon
A\to\bar A$ induces an isomorphism $\Ext^1_\A(S, A)\xto{\sim}
\Ext^1_\A(S,\bar A)$ taking $\xi$ to $\pi\xi$. There is an isomorphism
$\bar\tau\colon\bar A\to\bar A$ such that $\pi\xi'=\bar \tau(\pi\xi)$
since $\ell(\bar A)=1$. We claim that $\bar\tau$ extends to an
isomorphism $\tau\colon A\to A$ satisfying $\pi\tau=\bar\tau
\pi$. This implies $\pi\xi'= \bar\tau\pi\xi=\pi\tau\xi$, and therefore
$\xi'=\tau\xi$. Thus it remains to construct $\tau$. To this end
consider the non-split extension $\mu\colon 0\to A'\to A\to\bar A\to
0$. The induction hypothesis yields an isomorphism $\tau'\colon A'\to
A'$ such that $\tau'(\mu\bar\tau)=\mu$ since $\ell(A')<\ell(A)$. This
gives the isomorphism $\tau$ satisfying $\pi\tau=\bar\tau \pi$.
\end{proof}

\begin{lem}\label{le:uniserial}
Let $\A$ be a length category and suppose that \eqref{eq:dim} holds
for each simple object $S$.  For two uniserial objects $A$ and $B$ the
following are equivalent:
\begin{enumerate}
\item $A\cong B$.
\item $\top A\cong\top B$ and $\ell(A)=\ell(B)$.
\item $\soc A\cong\soc B$ and $\ell(A)=\ell(B)$.
\end{enumerate}
\end{lem}
\begin{proof}
The condition \eqref{eq:dim} is self-dual. Thus it suffices to show
the equivalence (1) $\Leftrightarrow$ (2). This equivalence follows
from Lemma~\ref{le:Gabriel} using induction on the length $\ell(A)$.
\end{proof}

\begin{proof}[Proof of Theorem~\ref{th:Gabriel}]
Suppose first that $\A$ is uniserial. Choose a simple object $S$ and
assume that $$\sum_{S'\in
\Si_0}\dim_{\Delta(S')^\op}\Ext_\A^1(S,S')\geq 2.$$ Then there exists
an extension $\xi\colon 0 \rightarrow S'\oplus S'' \rightarrow E
\rightarrow S \rightarrow 0$ with $S',S''\in\Si_0$ such that for each
non-zero morphism $\theta\colon S'\oplus S''\rightarrow T$ with $T\in
\Si_0$, the induced extension $\theta\xi$ does not split. It is not
difficult to check that $E$ is indecomposable and has at least two
different composition series. Thus $E$ is not uniserial which is a
contradiction.

Now assume that \eqref{eq:dim} holds for each simple object $S$ and
fix an indecomposable object $A$. We show by induction on $\ell(A)$
that $A$ is uniserial. The case $\ell(A)=1$ is clear. Thus we choose
an exact sequence $\xi\colon 0\rightarrow A' \rightarrow A \rightarrow
S\rightarrow 0$ with $S$ simple and fix a decomposition
$A'=\bigoplus_{i=1}^lA_i$ into indecomposable objects. Note that each
$A_i$ is uniserial by our hypothesis. If $l=1$, then $A$ is uniserial
by Lemma~\ref{le:Gabriel}. Now assume $l>1$. Denote by $\xi_i$ the
pushout of $\xi$ along the projection $A'\rightarrow A_i$. Note that
$\xi_i\neq 0$; otherwise $A_i$ is isomorphic to a direct summand of
$A$.  Therefore $\Ext_\A^1(S, A_i)\neq 0$ for all $i$, and
Lemma~\ref{le:Gabriel} implies $\top A_i\cong\top A_1$ for all $i$.
Assume that $\ell(A_1)\geq \ell(A_2)$. Then we have an epimorphism
$\pi\colon A_1\rightarrow A_2$ by Lemma~\ref{le:uniserial} which
induces an isomorphism $\Ext_\A^1(S, A_1)\xto{\sim} \Ext_\A^1(S, A_2)$
by Lemma~\ref{le:Gabriel}.  Moreover, there exists an isomorphism
$\tau\colon A_2\to A_2$ such that $\pi\xi_1=\tau\xi_2$. Consider the
morphism $\p\colon A'\to A_2$ with $\p_1=\pi$, $\p_2=-\tau$ and
$\p_i=0$ for $2<i\leq l$. We have $\p\xi=0$ by construction, and
therefore $A_2$ is isomorphic to a direct summand of $A$. This is a
contradiction. Thus $A$ is uniserial.
\end{proof}

Let $\A$ be a uniserial category. Choose a complete set of
representatives of the isomorphism classes of indecomposable objects
in $\A$ and denote it by $\ind\A$. An object in $\ind\A$ with top $S$
and length $n$ is denoted by $S^{[n]}$. Analogously, we write
$S_{[n]}$ for an object in $\ind\A$ with socle $S$ and length $n$.

For each simple object $S$, we have a chain of monomorphisms
$$S=S_{[1]}\rightarrowtail S_{[2]}\rightarrowtail
S_{[3]}\rightarrowtail \cdots$$ which is either finite or
infinite. Dually, there is a chain of epimorphisms
$$\cdots \twoheadrightarrow S^{[3]}\twoheadrightarrow
S^{[2]}\twoheadrightarrow S^{[1]}=S.$$

A morphism $\p\colon A\to B$ in an additive category is called
\index{morphism!irreducible} \emph{irreducible} if $\p$ is neither a
split monomorphism nor a split epimorphism and if for any
factorisation $\p=\p''\p'$ the morphism $\p'$ is a split monomorphism
or $\p''$ is a split epimorphism.

\begin{lem}\label{le:irr}
Let $\A$ be a uniserial category. For a morphism $\p\colon A\to B$
between indecomposable objects, the following are equivalent:
\begin{enumerate}
\item The morphism $\p$ is irreducible.
\item The object $\Ker\p\oplus\Coker\p$ is simple.
\item There exists a simple object $S$ and an integer $n$ such
that  $\p$ is, up to isomorphism, of the form
$S_{[n]}\rightarrowtail S_{[n+1]}$ or $S^{[n+1]}\twoheadrightarrow
S^{[n]}$.
\end{enumerate}
\end{lem}
\begin{proof}
(1) $\Rightarrow$ (2): An irreducible morphism is either a
monomorphism or an epimorphism. It suffices to discuss the case that
$\p$ is an epimorphism; the other case is dual. If $\ell(\Ker\p)>1$
and $S\subseteq\Ker \p$ is a simple subobject, then $\p$ can be
written as composite $A\to A/S\to B$ of two proper epimorphisms. This
is a contradiction, and therefore $\Ker\p\oplus\Coker\p$ is simple.

(2) $\Rightarrow$ (3): Clear.

(3) $\Rightarrow$ (1): It suffices to consider the morphism $\p\colon
S^{[n+1]}\twoheadrightarrow S^{[n]}$; the dual argument works for
$S_{[n]}\rightarrowtail S_{[n+1]}$. Let $S^{[n+1]}\xto{\a}
X\xto{\b}S^{[n]}$ be a factorization and fix a decomposition
$X=\bigoplus_iX_i$ into indecomposable objects. Then
$\b_{i_0}\a_{i_0}$ is an epimorphism for at least one index $i_0$. It
follows from Lemma~\ref{le:uniserial} that $X_{i_0}=S^{[m]}$ for some
$m\geq n$. If $m=n$, then $\b_{i_0}$ is an isomorphism, and therefore
$\b$ is a split epimorphism. Otherwise, we obtain a factorization
$S^{[n+1]}\xto{\a_{i_0}}
X_{i_0}\xto{\b'_{i_0}}S^{[n+1]}\twoheadrightarrow S^{[n]}$ of the
epimorphism $\b_{i_0}\a_{i_0}$. It follows that $\b'_{i_0}\a_{i_0}$ is
an epimorphism and hence an isomorphism. Thus $\a$ is a split
monomorphism.
\end{proof}

\begin{rem}\label{re:inj}
Let $\A$ be a uniserial category and $S$ a simple object. Suppose
there is a bound $n$ such that $\ell(A)\leq n$ for each indecomposable
object with $\soc A\cong S$. Then each indecomposable object $A$ of
length $n$ with $\soc A\cong S$ is injective, since
$\Ext_\A^1(T,A)=0$ for every simple object $T$ in $\A$ by
Lemma~\ref{le:Gabriel}.
\end{rem}

\begin{exm}
(1) Let $\La$ be a Dedekind domain. The finitely generated torsion
modules over $\La$ form a uniserial category. We denote this category
by $\mod_0\La$ because it coincides with the category $(\mod\La)_0$ of
finite length objects of $\mod\La$.  Let $\Spec \La$ denote the set of
prime ideals.  The functor which takes a $\La$-module $M$ to the
family of localizations $(M_\frp)_{\frp\in\Spec\La}$ induces an
equivalence
$$\mod_0\La\lto[\sim]\coprod_{0\neq\frp\in\Spec\La}\mod_0(\La_\frp).$$

(2) Let $k$ be a field and $P\in k[x]$ an irreducible polynomial. For
each $n>0$, the finitely generated $k[x]/(P^n)$-modules form a
uniserial category with a unique simple object.

(3) Let $k$ be a field and $\Ga$ a quiver having no oriented cycle. The
category of representations $\rep(\Ga,k)$ is uniserial if and only if
for each vertex $x$ of $\Ga$, there is at most one arrow starting at $x$
and at most one arrow ending at $x$.
\end{exm}

\subsection{Serre duality}

Let $k$ be a commutative ring.  A category $\A$ is
\index{category!$k$-linear} \emph{$k$-linear} if each morphism set is
a $k$-module and the composition maps are $k$-bilinear.  A functor
between $k$-linear categories is \index{functor!$k$-linear}
\emph{$k$-linear} provided that the induced maps between the morphism
sets are $k$-linear.  A $k$-linear category is
\index{category!Hom-finite} \emph{Hom-finite} if each morphism set is
a $k$-module of finite length.  Suppose now that $\A$ is a $k$-linear
abelian category.  Then the extension groups are naturally modules
over $k$, and $\A$ is called \index{category!Ext-finite}
\emph{Ext-finite} if $\Ext^n_\A(A,B)$ is of finite length over $k$ for
all $A,B$ in $\A$ and $n\geq 0$.

Fix a field $k$ and a Hom-finite $k$-linear abelian category $\A$.
The category $\A$ is said to satisfy \index{Serre duality} \emph{Serre
  duality}\footnote{This is the appropriate notion of Serre duality
  for hereditary abelian categories. Higher dimensional analogues
  involving $D\Ext^n_\A(A,-)$ appear in algebraic geometry; see also
  \S3.4.} if there exists an equivalence $\tau\colon \A \xto{\sim} \A$
with functorial $k$-linear isomorphisms
$$D\Ext^1_\A(A, B)\xto{\sim} \Hom_\A(B, \tau A)$$ for all $A, B$ in
$\A$, where $D=\Hom_k(-, k)$ denotes the standard $k$-duality. The
functor $\tau$ is called \index{Serre functor} \emph{Serre
functor} or \index{Auslander-Reiten translation}
\emph{Auslander-Reiten translation}. Note that a Serre functor is
$k$-linear and essentially unique provided it exists; this follows
from Yoneda's lemma.

Recall that $\A_0$ denotes the full subcategory consisting of all
finite length objects in $\A$. Denote by $\A_+$ the full subcategory
consisting of all objects $A$ in $\A$ satisfying $\Hom_\A(A_0,A)=0$ for
all $A_0$ in $\A_0$.

\begin{prop}\label{pr:Serreduality}
Let $\A$ be a Hom-finite $k$-linear abelian category and suppose $\A$
admits a Serre functor $\tau$. Then the following holds:
\begin{enumerate}
\item  The category $\A$ is hereditary.
\item The category $\A$ has no non-zero projective or injective
objects.
\item A noetherian object $A$ has a unique maximal subobject $A_0$ of
finite length. Moreover, $A_0$ is a direct summand of $A$ and $A/A_0$
belongs to $\A_+$.
\item For each indecomposable object $A$ in $\A$, there is an almost
split sequence\footnote{For the notion of an almost split sequence, we
refer to \cite{ARS}.} $0\rightarrow \tau A \rightarrow E \rightarrow
A \rightarrow 0$.
\item For each object $A$ in $\A$, we have $A^\perp={^\perp \tau A}$.
\end{enumerate}
\end{prop}

\begin{proof}
(1) For each object $A$, the functor $\Ext^1_\A(A, -)$ is right
exact. Thus the category $\A$ is hereditary by
Lemma~\ref{lem:hereditary}.

(2) Let $A$ be a projective object. Then $\Hom_\A(-, \tau A)\cong
D\Ext^1_\A(A, -)=0$. Thus $\tau A=0$ and therefore $A=0$. The dual
argument works for injective objects.

(3) Choose a maximal subobject $A_0$ of finite length. Then $A/A_0$
    belongs to $\A_+$, and every finite length subobject of $A$ is
    contained in $A_0$. In particular, $A_0$ is unique. We have
    $\Ext^1_\A(A/A_0,A_0)=0$ by Serre duality, and therefore $A_0$ is
    a direct summand of $A$.

(4) Let $A$ be an indecomposable object. The endomorphism ring
$\End_\A(A)$ is local and we denote by $\mathfrak m$ its maximal
ideal.  Choose any non-zero $k$-linear map $\omega\colon
\End_\A(A)\rightarrow k$ such that $\omega$ vanishes on $\mathfrak
m$. The map $\omega$ corresponds via Serre duality to a non-split
short exact sequence $\xi\colon 0\rightarrow \tau A \xto{} E \xto{}
A \rightarrow 0$. We claim that $\xi$ is an almost split sequence.
For this one needs to show that each morphism $\a\colon A'\to A$
factors through the morphism $E\to A$, provided that $\a$ is not a
split epimorphism. Thus one needs to show that $\xi\a=0$. The
element $\xi\a$ corresponds via Serre duality to $\omega\a$ which
sends $\p\in\Hom_\A(A,A')$ to $\omega(\a\p)$. Thus $\xi\a=0$, since
$\a\p$ belongs to $\mathfrak m$.

(5) This is clear from the definitions.
\end{proof}


Next observe that a Serre functor $\tau$ on $\A$ restricts to a Serre
functor on the subcategory $\A_0$ of finite length objects. The following
result describes the structure of a length category with Serre
duality. Let us recall the shape of the relevant diagrams.
\begin{align*} \tilde\bbA_{n}&\colon\;\;\; \xymatrix{1
\ar@{-}[r]&2\ar@{-}[r]&3\ar@{-}[r]&{\dots}\ar@{-}[r]&n\ar@{-}[r]&n+1\ar@{-}
@/^1pc/[lllll]}\\[3ex] \bbA_\infty^\infty&\colon\;\;\; \xymatrix{\dots
\ar@{-}[r]&\scriptstyle{\bullet}\ar@{-}[r]&
\scriptstyle{\bullet}\ar@{-}[r]&\scriptstyle{\bullet}\ar@{-}[r]&
\scriptstyle{\bullet}\ar@{-}[r]&{\dots}}
\end{align*}

\begin{prop}\label{pr:serre_finlen}
Let $\A$ be a Hom-finite $k$-linear length category and suppose $\A$
admits a Serre functor $\tau$. Then $\A$ is uniserial. The category
$\A$ admits a unique decomposition $\A=\coprod_{i\in I} \A_i$ into
connected uniserial categories with Serre duality, where the index set
equals the set of $\tau$-orbits of simple objects in $\A$.  The
Ext-quiver of each $\A_i$ is either of type $\bbA_\infty^\infty$ (with
linear orientation) or of type $\tilde\bbA_{n}$ (with cyclic
orientation).
\end{prop}

\begin{proof}
We apply the criterion of Theorem~\ref{th:Gabriel} to show that
$\A$ is uniserial. To this end fix a simple object $S$. Then
$\Ext_\A^1(S, S')\cong D\Hom_\A(S', \tau S) \neq 0$ for some
$S'\in\Si_0$ if and only if $S'\cong\tau S$. Moreover,
$\dim_{\Delta(S)}\Ext_\A^1(S, \tau S)=1$. Thus the category $\A$ is
uniserial.

The structure of the Ext-quiver of $\A$ follows from the condition
\eqref{eq:dim}. The Serre functor acts on $\Si_0$ and the set of
$\tau$-orbits $I=\Si_0/\tau$ is the index set of the decomposition
$\A=\coprod_{i\in I} \A_i$ into connected components; see
Lemma~\ref{le:comp}. The Ext-quiver of $\A_i$ is of type
$\bbA_\infty^\infty$ if the corresponding $\tau$-orbit is
infinite. Otherwise, the Ext-quiver of $\A_i$ is of type $\tilde
\bbA_n$ where $n+1$ equals the cardinality of the $\tau$-orbit.
\end{proof}

Let $\A$ be a Hom-finite $k$-linear length category and suppose $\A$
admits a Serre functor. Then a complete set of representatives of the
isomorphism classes of indecomposable objects of $\A$ is given by
$\{S^{[n]}\mid S\in\Si_0,\, n\ge 1\}$ and also by $\{S_{[n]}\mid
S\in\Si_0,\, n\ge 1\}$; see Remark~\ref{re:inj}.

\begin{exm}\label{ex:cyclic}
Let $k$ be a field and $\Ga$ a quiver of extended Dynkin type
$\tilde\bbA_n$ with cyclic orientation. Denote by $\A=\rep_0(\Ga,k)$
the full subcategory of $\rep(\Ga,k)$ consisting of all nilpotent
representations. Then $\A$ satisfies Serre duality and the Ext-quiver
of $\A$ is isomorphic to $\Ga$, with valuation $(1,1)$ for each arrow.
Note that the Serre functor on $\A$ has order $n+1$ and every simple
object has endomorphism algebra $k$. In fact, a connected Hom-finite
$k$-linear length category with Serre duality satisfying these
properties is equivalent to $\rep_0(\Ga,k)$.
\end{exm}

\section{Derived categories}

To each abelian category is associated its derived category.  This
section provides a brief introduction. We present the definition and
discuss two cases where one has a convenient description: If the
abelian category is hereditary, then each complex is isomorphic to its
cohomology. On the other hand, if there are enough projective objects,
then one can compute morphisms in the derived category by passing to
the homotopy category of projective objects.

\subsection{Categories of complexes}

Let $\A$ be an additive category.  A \index{cochain complex}
\emph{cochain complex} in $\A$ is a sequence of morphisms
$$\cdots \lto X^{n-1}\lto[d^{n-1} ]X^n\lto[d^n]X^{n+1}\lto\cdots$$
such that $d^{n} d^{n-1}=0$ for all $n\in\bbZ$. We denote by
$\bfC(\A)$ the category of cochain complexes, where a morphism
$\p\colon X\to Y$ between cochain complexes consists of morphisms
$\p^n\colon X^n\to Y^n$ with $d_Y^{n}\p^n=\p^{n+1}d_X^n$ for all
$n\in\bbZ$.

A \index{chain complex} \emph{chain complex} in $\A$ is a sequence of
morphisms $$\cdots \lto X_{n+1}\lto[d_{n+1}
]X_n\lto[d_n]X_{n-1}\lto\cdots$$ such that $d_{n} d_{n+1}=0$ for all
$n\in\bbZ$. Any chain complex may be viewed as a cochain complex by
changing its indices, and vice versa. Thus we often confuse both
concepts and simply use the term \index{complex} \emph{complex}.

A morphism $\p\colon X\to Y$ between complexes is
\index{morphism!null-homotopic} \emph{null-homotopic} if there are
morphisms $\r^n\colon X^n\to Y^{n-1}$ such that $\p^n=d_Y^{n-1}
\r^{n}+\r^{n+1} d_X^n$ for all $n\in\bbZ$. The null-homotopic
morphisms form an \index{ideal} \emph{ideal} $\N$ in $\bfC(\A)$, that
is, for each pair $X,Y$ of complexes a subgroup
$$\N(X,Y)\subseteq\Hom_{\bfC(\A)}(X,Y)$$ such that any composition
$\psi\p$ of morphisms in $\bfC(\A)$ belongs to $\N$ if $\p$ or $\psi$
belongs to $\N$.  The \index{homotopy category}\emph{homotopy
category} $\bfK(\A)$ is the quotient of $\bfC(\A)$ with respect to
this ideal. Thus
$$\Hom_{\bfK(\A)}(X,Y)=\Hom_{\bfC(\A)}(X,Y)/\N(X,Y)$$ for every pair
of complexes $X,Y$.

Now let $\A$ be an abelian category. The \index{cohomology}
\emph{cohomology} of a complex $X$ in degree $n$ is by definition
$H^nX=\Ker d^n/\Im d^{n-1}$, and each morphism $\p\colon X\to Y$ of
complexes induces a morphism $H^n\p\colon H^nX\to H^nY$. The morphism
$\p$ is a \index{quasi-isomorphism} \emph{quasi-isomorphism} if
$H^n\p$ is an isomorphism for all $n\in\bbZ$.  Note that two morphisms
$\p,\psi\colon X\to Y$ induce the same morphism $H^n\p=H^n\psi$, if
$\p-\psi$ is null-homotopic.

The \index{derived category} \emph{derived category} $\bfD(\A)$ of
$\A$ is obtained from $\bfK(\A)$ by formally inverting all
quasi-isomorphisms. To be precise, one defines
$$\bfD(\A)=\bfK(\A)[\qis^{-1}]$$ as the localization of $\bfK(\A)$
with respect to the class of all quasi-isomorphisms.  The full
subcategory consisting of objects that are isomorphic to a complex
$X$ such that $X^n=0$ for almost all $n\in\bbZ$ is denoted by
$\bfD^b(\A)$.

An object $A$ in $\A$ is identified with the complex
$$\cdots\lto 0\lto A\lto 0\lto\cdots$$ concentrated in degree zero,
and this complex is also denoted by $A$.  Given any complex $X$ in $\A$ and
$p\in\bbZ$, we denote by $X[p]$ the shifted complex with
$$X[p]^n=X^{n+p}\quad\textrm{and}\quad d^n_{X[p]}=(-1)^pd^{n+p}_X.$$
This operation induces an isomorphism $\bfD(\A)\xto{\sim}\bfD(\A)$ and
is called \index{shift} \emph{shift}.

The derived category $\bfD(\A)$ is an additive category with some
additional structure: it is a triangulated category in the sense of
Verdier \cite{V}. For instance, any exact sequence $0\to A\to B\to C\to
0$ in $\A$ induces an exact triangle $A\to B\to C\to A[1]$ in
$\bfD(\A)$.

Given two abelian categories $\A$ and $\A'$, a functor $F\colon
\bfD(\A)\to\bfD(\A')$ is by definition a \index{derived
equivalence}\emph{derived equivalence} if it is an \emph{equivalence
of triangulated categories}, that is, $F$ is an equivalence, there is
a functorial isomorphism $(FX)[1]\cong F(X[1])$ for each $X$ in
$\bfD(\A)$, and $F$ preserves exact triangles.

The following statement justifies the study of derived categories.

\begin{prop}\label{pr:ext}
Let $A,B$ be objects in $\A$. Then
$$\Ext^n_\A(A,B)\cong\Hom_{\bfD(\A)}(A,B[n])\quad\text{for all}\quad
n\in \mathbb{Z}.$$
\end{prop}
\begin{proof}
For the case that $\A$ has enough injectives or enough projectives,
see \cite[Corollary~10.7.5]{W}. For the general case, see \cite[III.3]{V}.
\end{proof}

\subsection{Hereditary abelian categories}

Let $\A$ be a hereditary abelian category, that is, $\Ext_\A^2(-,-)$
vanishes. In this case, there is an explicit description of all
objects and morphisms in $\bfD^b(\A)$ via the ones in $\A$. Every
complex $X$ is completely determined by its cohomology because there
is an isomorphism between $X$ and the following complex with trivial
differential.
$$\cdots \lto H^{n-1}X\lto[0]H^nX\lto[0]H^{n+1}X\lto\cdots$$ To
construct this isomorphism, note that the vanishing of
$\Ext^2_\A(H^nX,-)$ implies the existence of a commutative diagram
$$
\xymatrix{0\ar[r]&X^{n-1}\ar[r]\ar[d]&E^n\ar[r]\ar[d]&
H^nX\ar[r]\ar@{=}[d]&0\\ 0\ar[r]&\Im d^{n-1}\ar[r]&\Ker
d^n\ar[r]&H^nX\ar[r]&0}$$ with exact rows. We obtain the following
commutative diagram.
$$\xymatrixrowsep{1.5pc} \xymatrixcolsep{1.5pc} \xymatrix{
\cdots\ar[r]&0\ar[r]&0\ar[r]&H^nX\ar[r]&0\ar[r]&\cdots\\
\cdots\ar[r]&0\ar[u]\ar[r]\ar[d]&X^{n-1}\ar[u]\ar[r]\ar@{=}[d]&E^n\ar[u]\ar[r]\ar[d]&
0\ar[u]\ar[r]\ar[d]&\cdots\\
\cdots\ar[r]&X^{n-2}\ar[r]&X^{n-1}\ar[r]&X^{n}\ar[r]&X^{n+1}\ar[r]&\cdots
}$$ The vertical morphisms yield two morphisms in $\bfD^b(\A)$. The
upper one is a quasi-isomorphism, and the lower one induces a
cohomology isomorphism in degree $n$. This yields for each $n\in\bbZ$
a morphism $(H^nX)[-n]\to X$ in $\bfD^b(\A)$ and therefore the
following description of $X$.

\begin{lem}\label{le:der_hereditary}
Let $\A$ be a hereditary abelian category and $X$  a complex in $\A$. In
$\bfD^b(\A)$ there is a (non-canonical) isomorphism
\[\coprod_{n\in\bbZ}(H^nX)[-n]\lto[\sim] X.\]
\end{lem}
\begin{proof}
The morphism is a quasi-isomorphism by construction.
\end{proof}

For an abelian category $\A$ one defines its \index{repetitive
category}\emph{repetitive category} $\bigsqcup_{n\in\bbZ}\A[n]$ as the
additive closure of the union of disjoint copies $\A[n]$ of $\A$ with
morphisms
$$\Hom(A,B)=\Ext_\A^{q-p}(A,B)\quad \text{for}\quad A\in\A[p],\,
B\in\A[q]$$ and composition given by the Yoneda product of
extensions. It follows from Proposition~\ref{pr:ext} that the family
of functors $\A[n]\to\bfD^b(\A)$ ($n\in\bbZ$) sending an object $A$ to
$A[n]$ induces a fully faithful functor
$$\bigsqcup_{n\in\bbZ}\A[n]\lto\bfD^b(\A).$$

\begin{cor}\label{co:der_hereditary}
\pushQED{\qed} The canonical functor $\bigsqcup_{n\in\bbZ}\A[n]\to\bfD^b(\A)$ is
an equivalence for any hereditary abelian category $\A$.  \qedhere
\end{cor}

\subsection{Abelian categories with enough projectives}

We describe the derived category of an abelian category $\A$ in terms
of its projective objects. The crucial observation is the following.

\begin{lem}\label{le:der_proj}
Let $X,Y$ be a pair of complexes in $\A$. Suppose that each $X^n$ is
projective and $X^n=0$ for $n\gg 0$. Then the map
$\Hom_{\bfK(\A)}(X,Y)\to\Hom_{\bfD(\A)}(X,Y)$ is bijective.
\end{lem}
\begin{proof}
See for example \cite[Corollary~10.4.7]{W}.
\end{proof}

Let $\Proj\A$ denote the full subcategory of $\A$ consisting of all
objects that are projective. Denote by $\bfK^{-,b}(\Proj\A)$ the full
subcategory of complexes $X$ in $\bfK(\Proj\A)$ such that $X^n=0$ for
$n\gg 0$ and $H^nX=0$ for almost all $n\in\bbZ$.  One says that $\A$
has \index{category!with enough projectives}\emph{enough projectives}
if each object in $\A$ is the quotient of some projective object.

\begin{prop}\label{pr:der_equiv_proj}
Let $\A$ be an abelian category having enough projectives.  Then the
canonical functor $\bfK(\A)\to\bfD(\A)$ induces an equivalence
$$\bfK^{-,b}(\Proj\A)\lto[\sim]\bfD^b(\A).$$
\end{prop}
\begin{proof}
The functor $F\colon\bfK^{-,b}(\Proj\A)\to\bfD(\A)$ is by definition
the identity on objects, and $F$ is fully faithful by
Lemma~\ref{le:der_proj}.  It is clear that each object in the image of
$F$ is isomorphic to one in $\bfD^b(\A)$.  To show that each complex
$X$ in $\bfD^b(\A)$ is isomorphic to one in the image of $F$, we use
induction on
$$\ell(X)=\card\{n\in\bbZ\mid X^n\neq 0\}.$$ Note that each bounded
complex $X\neq 0$ fits into an exact triangle $X'\to X''\to X\to
X'[1]$ such that $\ell(X')=1$ and $\ell(X'')=\ell(X)-1$.  If
$\ell(X)=1$, say $X^n\neq0$, then $X\cong F(P[-n])$ where $P$ denotes
a projective resolution of $X^n$. Such a resolution exists since $\A$
has enough projectives. If $\ell(X)>1$, then the induction hypothesis
implies that the morphism $X'\to X''$ is up to isomorphism of the form
$F\p$ for some morphism $\p\colon P'\to P''$ in $\bfK^{-,b}(\Proj\A)$.
Completing the morphism $\p$ to an exact triangle $P'\to P''\to P\to
P'[1]$ shows that $X$ belongs to $\Im F$ since $X\cong FP$.
\end{proof}

\section{Tilting theory}

Tilting provides a method to relate a category of coherent sheaves to
a category of linear representations.  For instance, a result of
Beilinson \cite{B} establishes for the category $\coh\bbP^n_k$ of
coherent sheaves on the projective $n$-space over a field $k$ an
equivalence of derived categories
$$\RHom(T,-)\colon\bfD^b(\coh\bbP^n_k)\lto[\sim]\bfD^b(\mod\End(T))$$
via a tilting object $T$ in $\coh\bbP^n_k$.\footnote{Except for $n=1$,
the object $T=\Oc(0)\oplus\dots\oplus\Oc(n)$ is not a tilting object
in the strict sense of these notes; see Proposition~\ref{pr:coh_tilt}.}

In this section let $k$ be a field and $\A$ a $k$-linear abelian
category that is Ext-finite. We show that each tilting object $T$ in $\A$
provides an equivalence of derived categories
$$\RHom_\A(T,-)\colon\bfD^b(\A)\lto[\sim]\bfD^b(\mod\End_\A(T))$$
as in the example above. The principal reference for this result is
\cite{HRS}, even though the proof given here is somewhat more direct,
avoiding the formalism of torsion pairs and t-structures.

\subsection{Tilting objects}

Fix an object $T$ in $\A$. The object $T$ is called \index{tilting
object} \emph{tilting object}, provided that
\begin{enumerate}
\item $\pdim T \leq 1$,
\item $\Ext^1_\A(T,T)=0$, and
\item $\Hom_\A(T,A)=0=\Ext^1_\A(T,A)$ implies $A=0$ for each object $A$ in $\A$.
\end{enumerate}

A morphism $T'\to A$ in $\A$ is called \index{right approximation}
\emph{right $T$-approximation} of $A$ if it induces an epimorphism
$\Hom_\A(T,T')\to\Hom_\A(T,A)$ and $T'$ belongs to $\add T$. An exact
sequence $0\to A\to B\to T'\to 0$ is called \index{universal
  extension} \emph{universal $T$-extension} of $A$ if it induces an
epimorphism $\Hom_\A(T,T')\to\Ext^1_\A(T,A)$ and $T'$ belongs to $\add
T$.  Such approximations and extensions exist for all $A$ in $\A$,
since $\A$ is Ext-finite. Finally, set
$$\T(T)=\{A\in\A\mid\Ext^1_\A(T,A)=0\}.$$

\begin{lem}\label{le:tilt_obj}
Let $T\in\A$ be a tilting object. Then the following holds:
\begin{enumerate}
\item Let $\pi\colon T'\to A$ be a right $T$-approximation. Then
$\Ker\pi$ is in $\T(T)$, $\Hom_\A(T,\Coker\pi)=0$, and
$\Ext^1_\A(T,A)\xto{\sim}\Ext^1_\A(T,\Coker\pi)$.
\item Let $0\to A\to B\to T'\to 0$ be a universal $T$-extension. Then $B\in\T(T)$.
\item The objects in $\T(T)$ are precisely the factor objects of objects in $\add T$.
\end{enumerate}
\end{lem}
\begin{proof}
(1) Write the sequence $0\to A'\to T'\xto{\pi} A\to A''\to 0$ as composite
of two exact sequences $0\to A'\to T'\to \bar A\to 0$ and
$0\to \bar A\to A\to A''\to 0$. Then apply $\Hom_\A(T,-)$ to both
sequences.

(2) Apply $\Hom_\A(T,-)$ to the sequence $0\to A\to B\to T'\to 0$.

(3) Clearly, each factor of an object in $\add T$ belongs to
    $\T(T)$. For the other implication one uses (1).
\end{proof}

\subsection{A derived equivalence}

Let $T$ be an object in $\A$ and $\La=\End_\A(T)$.  We consider the
functor $$\Hom_\A(T,-)\colon\A\lto\mod\La.$$ This functor induces an
equivalence $\add T\xto{\sim}\proj\La$ and admits a left adjoint
$$-\otimes_\La T\colon \mod\La\lto\A.$$ Given a $\La$-module $M$ with
projective presentation $P_1\to P_0\to M\to 0$, the cokernel of the
corresponding morphism $T_1\to T_0$ in $\add T$ is by definition
$M\otimes_\La T$.  For $i>0$, denote by
$$\Tor^\La_i(-,T)\colon\mod\La\lto\A$$ the $i$-th left derived functor of
$-\otimes_\La T$ and set
$$\Y(T)=\{M\in\mod\La\mid \Tor^\La_1(M,T)=0\}.$$

\begin{lem}\label{le:tilt_tor}
Let $T\in\A$ be a tilting object. Then $\Tor^\La_i(-,T)=0$ for $i>1$.
\end{lem}
\begin{proof}
Let $M\in\mod\La$ and choose a projective resolution
$$\cdots \lto P_2\lto[\delta_2] P_1\lto[\delta_1] P_0\lto[\delta_0]
M\lto 0.$$ Apply $-\otimes_\La T$ and set
$Z_i=\Ker(\delta_i\otimes_\La T)$. The induced morphism
$\bar\delta_{i+1}\colon P_{i+1}\otimes_\La T\to Z_i$ is a right
$T$-approximation for $i>0$, and therefore $Z_i$ belongs to $\T(T)$
for $i>1$ by Lemma~\ref{le:tilt_obj}. Thus $\bar\delta_{i+1}$ is an
epimorphism for $i>1$, and this implies $\Tor^\La_i(-,T)=0$.
\end{proof}

\begin{lem}\label{le:tilt_eq}
Let $T\in\A$ be a tilting object. Then $\Hom_A(T,-)$
and $-\otimes_\La T$ restrict to equivalences between $\T(T)$ and $\Y(T)$.
\end{lem}
\begin{proof}
Fix objects $A\in\T(T)$ and $M\in\Y(T)$. We need to show that the
adjunction morphisms
$$\Hom_\A(T,A)\otimes_\La T\lto[\theta_A]
A\quad\text{and}\quad M\lto[\eta_M]\Hom_\A(T,M\otimes_\La T)$$
are invertible.

Choose an exact sequence
\begin{equation*}
\xi\colon\;\cdots \lto T_2\lto[\delta_2] T_1\lto[\delta_1]
T_0\lto[\delta_0] A\lto 0
\end{equation*}
such that the induced morphism $T_i\to\Im\delta_i$ is a right
$T$-approximation of $\Im\delta_i$ for each $i\geq 0$.  Such a
sequence exists by Lemma~\ref{le:tilt_obj}.

The functor $\Hom_\A(T,-)$ sends the sequence $\xi$ to a projective
resolution of the $\La$-module $\Hom_\A(T,A)$. Applying then
$-\otimes_\La T$ gives back $\xi$, that is, the adjunction morphism
$\theta_A$ is invertible. Moreover, $\Hom_\A(T,A)$ belongs to $\Y(T)$.

Now choose an exact sequence $\pi\colon 0\to M'\to P\to M\to 0$ such
that $P$ is projective. Note that $M'$ belongs to $\Y(T)$ by
Lemma~\ref{le:tilt_tor}. The sequence $\pi\otimes_\La T$ is exact
since $M\in\Y(T)$, and the sequence $\Hom_\A(T,\pi\otimes_\La T)$ is
exact since $M'\otimes_\La T$ belongs to $\T(T)$. Thus there is the
following commutative diagram with exact rows.
\[\xymatrix@C=1pc@R=2pc{
0 \ar[r]& M' \ar[r] \ar[d]^{\eta_{M'}} & P \ar[r] \ar[d]^{\eta_P}
& M \ar[r] \ar[d]^{\eta_M} & 0 \\ 0 \ar[r] &\Hom_\A(T, M'\otimes_\La
T) \ar[r] & \Hom_\A(T, P\otimes_\La T) \ar[r] & \Hom_\A(T,
M\otimes_\La T) \ar[r] & 0 }\] The morphism $\eta_P$ is an
isomorphism and it follows that $\eta_M$ is an epimorphism for all
$M$ in $\Y(T)$. In particular, $\eta_{M'}$ is an epimorphism. Using
the snake lemma, it follows that $\eta_M$ is an isomorphism.
\end{proof}

Let $A$ be an object in $\A$.  An \index{$\add T$-resolution}
\emph{$\add T$-resolution} of $A$ is by definition a complex
$$Q\colon\;\cdots \lto Q_2\lto Q_1\lto Q_0\lto 0\lto 0\lto \cdots$$ together
with a quasi-isomorphism $Q\to A$ such that each $Q_n$ belongs to
$\add T$.

\begin{lem}\label{le:tilt_ex}
Let $T\in\A$ be a tilting object and $Q\to A$ an $\add T$-resolution
of an object $A\in \A$. Then
$$H^n\Hom_\A(Q,B)\cong\Ext_\A^n(A,B)$$ for all $B\in\T(T)$ and $n\geq
0$.
\end{lem}
\begin{proof}
Use induction on $n$ and dimension shifting.
\end{proof}

\begin{lem}\label{le:tilt_ext}
Let $T\in\A$ be a tilting object. Then the functor $-\otimes_\La T$ induces
an isomorphism
$$\Ext^n_\La(M,N)\lto[\sim]\Ext_\A^n(M\otimes_\La T,N\otimes_\La T)$$ for all
$M,N$ in $\Y(T)$ and $n\geq 0$.
\end{lem}
\begin{proof}
Choose a projective resolution $P\to M$ of $M$. Note that
$N\cong\Hom_\A(T,N\otimes_\La T)$ by Lemma~\ref{le:tilt_eq}, since $N$
belongs to $\Y(T)$.  Then we obtain the following sequence of isomorphisms.
\begin{align*}
\Ext^n_\La(M,N)&\cong H^n\Hom_\La(P,N)\\ &\cong H^n\Hom_\La(P,\Hom_\A(T,N\otimes_\La T))\\
&\cong H^n \Hom_\A(P\otimes_\La T,N\otimes_\La T)\\ &\cong\Ext_\A^n(M\otimes_\La T,N\otimes_\La T)
\end{align*}
The last isomorphism follows from Lemma~\ref{le:tilt_ex}, since
$P\otimes_\La T\to M\otimes_\La T$ is an $\add T$-resolution.
\end{proof}

For a tilting object $T$ in $\A$, let us define the derived functor
$$-\otimes^\bfL_\La
T\colon\bfD^b(\mod\La)\lto[\sim]\bfK^{-,b}(\proj\La)\lto\bfD^b(\A)$$
by taking a complex $P$ of projective $\La$-modules with bounded
cohomology to $P\otimes_\La T$; see
Proposition~\ref{pr:der_equiv_proj}. The cohomology of $P\otimes_\La
T$ is bounded, since $\Tor_i^\La(-,T)=0$ for $i>1$ by
Lemma~\ref{le:tilt_tor}.

\begin{thm}[Happel-Reiten-Smal{\o}]\label{th:tilt}
Let $\A$ be a $k$-linear abelian category that is Ext-finite.  Let $T$
be a tilting object in $\A$ and $\La=\End_\A(T)$. Then the functor
$$-\otimes^\bfL_\La T\colon\bfD^b(\mod\La)\lto\bfD^b(\A)$$ is an
equivalence of triangulated categories and its right adjoint
$\RHom_\A(T,-)$ is a quasi-inverse.
\end{thm}

We do not give the formal definition of the derived functor
$\RHom_\A(T,-)$; all we use is the fact that it is a right adjoint of
$-\otimes^\bfL_\La T$.

\begin{proof}
Set $F_T=-\otimes_\La^\bfL T$. We identify objects in $\mod\La$ and
$\A$ with complexes that are concentrated in degree zero.  For
instance, $F_T M=M\otimes_\La T$ for each $M$ in $\Y(T)$.

We need to show that for each pair of
complexes $X,Y$ in $\mod\La$, the induced map
$$\p_{X,Y}\colon \Hom_{\bfD^b(\mod\La)}(X,Y)\lto \Hom_{\bfD^b(\A)}(F_T
X, F_T Y)$$ is bijective.  Set
$$\ell(X)=\card\{n\in\bbZ\mid X_n\neq 0\}$$ and note that each bounded
complex $X\neq 0$ fits into an exact triangle $X'\to X\to X''\to
X'[1]$ such that $\ell(X')=\ell(X)-1$ and $\ell(X'')=1$.

Using the five lemma and induction on $\ell(X)+\ell(Y)$, one shows
that $\p_{X,Y}$ is bijective. The case $\ell(X)=\ell(Y)=1$ follows
from Lemma~\ref{le:tilt_ext}.  To be precise, one uses that each
$\La$-module $M$ fits into an exact sequence $0\to M'\to P\to M\to 0$
with $M',P$ in $\Y(T)$, which yields an exact triangle $M'\to P\to
M\to M'[1]$ in $\bfD^b(\mod\La)$.

Next we show that each object in $\bfD^b(\A)$ is isomorphic to one in
the image of $F_T$. In fact, it suffices to show that each object in
$\A$ belongs to the essential image $\Im F_T$, since $\Im F_T$ is a
triangulated subcategory and $\bfD^b(\A)$ is generated (as a
triangulated category) by the objects from $\A$.

It follows from Lemma~\ref{le:tilt_obj} that each object $A$ in $\A$
fits into an exact triangle $A\to B\to C\to A[1]$ with $B,C$ in
$\T(T)$.  On the other hand, each $A$ in $\T(T)$ belongs to $\Im F_T$,
since $A\cong F_T(\Hom_\A(T,A))$ by Lemma~\ref{le:tilt_eq}.
\end{proof}

\begin{exm}
(1) Let $T,T'$ be two objects in $\A$ with $\add T=\add T'$. Then $T$
    is a tilting object if and only if $T'$ is a tilting object.

(2) Let $k$ be a field and $\La$ a finite dimensional $k$-algebra. Then
any free $\La$-module of finite rank is a tilting object in $\mod
\La$.

(3) Let $k$ be a field and $\La=k\Ga$ the path algebra of a finite quiver
$\Ga$ without oriented cycles. For each vertex $i\in\Ga_0$ let $e_i$ denote
the corresponding idempotent. Fix a vertex $i_0$ which is not a sink
and consider the following short exact sequence
$$0\lto e_{i_0}\La\lto \bigoplus_{\a\colon i_0\to i}e_i\La\lto
T_{i_0}\lto 0$$ where the direct sum runs over all arrows starting at
$i_0$ and each morphism $e_{i_0}\La\to e_i\La$ is given by
multiplication with the corresponding arrow $\a\colon i_0\to i$. Set
$T_i=e_i\La$ for each vertex $i\neq i_0$. Then
$T=\bigoplus_{i\in\Ga_0}T_i$ is a tilting object of $\mod\La$.
\end{exm}

\subsection{Grothendieck groups}

Let $\A$ be an abelian category. Denote by $F(\A)$ the free abelian
group generated by the isomorphism classes of objects in $\A$. Let
$F_0(\A)$ be the subgroup generated by $[X]-[Y]+[Z]$ for all exact
sequences $0\to X\to Y\to Z\to 0$ in $\A$. The \index{Grothendieck
group} \emph{Grothendieck group} $K_0(\A)$ of $\A$ is by definition
the factor group $F(\A)/F_0(\A)$.

\begin{lem}\label{le:gr_gr1}
Let $\A$ be a length category. Then $K_0(\A)$ is a free
abelian group and the isomorphism classes of simple objects in $\A$
form a basis.
\end{lem}
\begin{proof}
Let $X$ be an object in $\A$ and $0=X_0\subseteq
X_1\subseteq\dots\subseteq X_n=X$ a composition series. Then
$[X]=[X_1/X_0]+\cdots +[X_n/X_{n-1}]$ in $K_0(\A)$. The
Jordan-H\"older theorem implies the uniqueness of this expression.
\end{proof}

Let $\T$ be a triangulated category. Denote by $F(\T)$ the free
abelian group generated by the isomorphism classes of objects in
$\T$. Let $F_0(\T)$ be the subgroup generated by $[X]-[Y]+[Z]$ for all
exact triangles $X\to Y\to Z\to X[1]$ in $\T$. The \index{Grothendieck
group} \emph{Grothendieck group} $K_0(\T)$ of $\T$ is by definition
the factor group $F(\T)/F_0(\T)$.

\begin{lem}\label{le:gr_gr2}
Let $\A$ be an abelian category. The inclusion $\A\to\bfD^b(\A)$ induces
an isomorphism $K_0(\A)\xto{\sim} K_0(\bfD^b(\A))$.
\end{lem}
\begin{proof}
Each exact sequence $0\to X\to Y\to Z\to 0$ in $\A$ induces an exact
triangle $X\to Y\to Z\to X[1]$ in $\bfD^b(\A)$.  This gives a morphism
$K_0(\A)\to K_0(\bfD^b(\A))$.  The inverse map sends the class $[X]$
of a complex $X$ to $\sum_{n\in\bbZ}(-1)^n[H^nX]$.
\end{proof}

\subsection{Serre duality}

Let $k$ be a field and $\La$ a finite dimensional $k$-algebra. We
denote by $D=\Hom_k(-,k)$ the standard $k$-duality.

The \index{Nakayama functor} \emph{Nakayama functor}
$\nu=D\Hom_\La(-,\La)\colon\mod\La\to\mod\La$ identifies the category
of projective $\La$-modules with the category of injective
$\La$-modules. Note that
\begin{equation*}
D\Hom_\La(P,-)\cong\Hom_\La(-,\nu P)
\end{equation*}
for every finitely generated projective $\La$-module $P$, since both
functors are left exact and agree on $\La$. This isomorphism induces
for every bounded complex $X$ of finitely generated projective
$\La$-modules a sequence of natural isomorphisms
\begin{equation}\label{eq:nu}
\begin{split}
D\Hom_{\bfD^b(\mod\La)}(X,-)&\cong D\Hom_{\bfK^b(\mod\La)}(X,-)\\
&\cong\Hom_{\bfK^b(\mod\La)}(-,\nu X)\\
&\cong\Hom_{\bfD^b(\mod\La)}(-,\nu X),
\end{split}
\end{equation}
where the first and the last isomorphism follow from Lemma~\ref{le:der_proj}.

A Hom-finite $k$-linear triangulated category $\T$ is said to satisfy
\index{Serre duality} \emph{Serre duality} if there exists an
equivalence $\tau\colon \T \xto{\sim} \T$ of triangulated categories
with functorial $k$-linear isomorphisms
$$D\Hom_\T(X,Y)\xto{\sim} \Hom_\T(Y, \tau X)$$ for all $X,Y$ in
$\T$. The functor $\tau$ is called a \index{Serre functor} \emph{Serre
functor}.  Note that a Serre functor is $k$-linear and essentially
unique provided it exists; this follows from Yoneda's lemma.

\begin{prop}\label{pr:serre_alg}
Let $\La$ be a finite dimensional $k$-algebra. Then $\bfD^b(\mod\La)$
satisfies Serre duality if and only if $\La$ has finite global
dimension.
\end{prop}
\begin{proof}
If the global dimension of $\La$ is finite, then every bounded complex
in $\mod\La$ is quasi-isomorphic to a bounded complex of finitely
generated projective $\La$-modules. Thus Serre duality for
$\bfD^b(\mod\La)$ follows from the isomorphism \eqref{eq:nu}.  The
converse follows immediately from Lemmas~\ref{le:asymp} and
\ref{le:gldim} below.
\end{proof}

\begin{lem}\label{le:asymp}
Let $\A$ be an abelian category and  $X,Y\in\bfD^b(\A)$. Then the
following holds:
\begin{enumerate}
\item $\Hom_{\bfD^b(\A)}(X,Y[n])=0$ for $n\ll 0$.
\item $\Hom_{\bfD^b(\A)}(X,Y[n])=0$ for $n\gg 0$, if $\A$ has finite
global dimension.
\end{enumerate}
\end{lem}
\begin{proof}
Use induction on $\ell(X)+\ell(Y)$, where $\ell(Z)=\card\{n\in\bbZ\mid
Z^n\neq 0\}$ for any complex $Z$. The case $\ell(X)=1=\ell(Y)$ is
clear, since $\Hom_{\bfD^b(\A)}(A,B[n])\cong\Ext^n_\A(A,B)$ for all
objects $A,B$ in $\A$; see Proposition~\ref{pr:ext}.
\end{proof}

\begin{lem}\label{le:gldim}
Let $\La$ be a finite dimensional algebra and $S_1,\dots,S_r$ a set
of representatives of the isomorphism classes of simple
$\La$-modules. Then $\gldim\La\leq n$ if and only if
$\Ext^{n+1}_\La(S_i,S_j)=0$ for all $i,j$.
\end{lem}
\begin{proof}
Use that each $\La$-module $M$ has a finite filtration $0=M_0\subseteq
M_1\subseteq\dots\subseteq M_p=M$ such that $M_i/M_{i-1}$ is
semisimple for all $i$.
\end{proof}

\begin{lem}\label{le:tilt_gldim}
Let $\A$ be an Ext-finite $k$-linear abelian category and suppose
there exists a tilting object $T$. If $\A$ has finite global
dimension, then $\End_\A(T)$ has finite global dimension.
\end{lem}
\begin{proof}
Let $\La=\End_\La(T)$. The functor $\RHom_\A(T,-)$ provides an
equivalence $\bfD^b(\A)\xto{\sim}\bfD^b(\mod\La)$ of triangulated
categories by Theorem~\ref{th:tilt}.  Thus we have $\Ext_\La^n(S,T)=0$ for
$n\gg 0$ and each pair $S,T$ of simple $\La$-modules by
Lemma~\ref{le:asymp}. It follows from Lemma~\ref{le:gldim} that the
global dimension of $\La$ is finite.
\end{proof}

\begin{prop}\label{pr:serre}
Let $\A$ be a $k$-linear abelian category that is Ext-finite and admits
a tilting object. Then the following are equivalent:
\begin{enumerate}
\item The category $\A$ is hereditary and has  no non-zero projective
object.
\item The category $\A$ satisfies Serre duality.
\end{enumerate}
\end{prop}
\begin{proof}
(1) $\Rightarrow$ (2): Let $T$ be the tilting object and
$\La=\End_\A(T)$. Then $\La$ has finite global dimension by
Lemma~\ref{le:tilt_gldim}, and therefore $\bfD^b(\mod\La)$ has Serre
duality by Proposition~\ref{pr:serre_alg}.  There is an equivalence
$\bfD^b(\A)\xto{\sim}\bfD^b(\mod\La)$ by Theorem~\ref{th:tilt}, and
this yields a Serre functor $\nu\colon\bfD^b(\A)\to\bfD^b(\A)$.

Now let $A,B$ be objects in $\A$ and view them as complexes
concentrated in degree zero.  Then
$$D\Ext^1_\A(A,B)\cong\Hom_{\bfD^b(\A)}(B,\nu A[-1]),$$ and it remains
to show that $H^i(\nu A[-1])$ for all $i\neq 0$. Any complex $X$ in
$\A$ is quasi-isomorphic to $\coprod_{i\in\bbZ} (H^iX)[-i]$ since $\A$
is hereditary; see Lemma~\ref{le:der_hereditary}. Assume that $A$ is
indecomposable. Then $\nu A[-1]\cong\bar A[d]$ for some $d\in\bbZ$ and
some object $\bar A$ in $\A$. We claim that $d=0$. First observe that
$\Hom_{\bfD^b(\A)}(B,\bar A[d])\neq 0$ for some object $B$, since $A$
is non-projective. Thus $d=0$ or $d=1$. The case $d=1$ is impossible
since $\Ext^2_\A(A,-)=0$. Thus $\nu A[-1]$ is concentrated in degree
zero.

(2) $\Rightarrow$ (1): See Proposition~\ref{pr:Serreduality}.
\end{proof}

\subsection{The Euler form}

Let $k$ be a field and $\A$ a $k$-linear abelian category. Suppose
that $\A$ is Ext-finite and of finite global dimension. The
\index{Euler form} \emph{Euler form} associated to $\A$ is by
definition the bilinear form $K_0(\A)\times K_0(\A)\to\bbZ$ with
$$\langle [A],[B]\rangle=\sum_{n\geq 0}(-1)^n\dim_k\Ext_\A^n(A,B).$$

Suppose that $K_0(\A)$ is a free abelian group of finite rank and fix
a basis $b_1,\dots,b_r$. The \index{Euler form!discriminant of}
\emph{discriminant} of the Euler form is then by definition the
determinant of the matrix $(\langle b_i,b_j\rangle)_{i,j}$ and we
denote it by $\disc\langle-,-\rangle$.  Note that this value does not
depend on the choice of the basis since the matrix is defined over
$\bbZ$.

\begin{lem}\label{le:euler_fin}
Let $\La$ be a finite dimensional $k$-algebra of finite global
dimension. Then the Euler form associated to $\mod\La$ is non-degenerate.
\end{lem}
\begin{proof}
Set $\A=\mod\La$. Let $S_1,\dots, S_r$ be representatives of
the isomorphism classes of simple $\La$-modules and choose a
projective cover $P_i\to S_i$ for each $i$. Then $[P_1],\dots,[P_r]$
form a basis of $K_0(\A)$, since each $S_i$ has a finite projective
resolution. Let $x=\sum_i\a_i[P_i]$ be a non-zero element of $K_0(\A)$
and pick an index $j$ such that $\a_{j}\neq 0$. Then $\langle
x,[S_{j}]\rangle=\a_{j}\dim_k\Hom_\La(P_{j},S_{j})\neq 0$.  Thus
$\langle -,-\rangle$ is non-degenerate.
\end{proof}

Let $\T$ be a $k$-linear triangulated category. Suppose that $\T$ is
Hom-finite and that $\Hom_\T(X,Y[n])=0$ for each pair of objects $X,Y$
and $|n|\gg 0$. The \index{Euler form} \emph{Euler form} associated to
$\T$ is by definition the bilinear form $K_0(\T)\times K_0(\T)\to\bbZ$
with $$\langle
[X],[Y]\rangle=\sum_{n\in\bbZ}(-1)^n\dim_k\Hom_\T(X,Y[n]).$$

It is clear from these definitions that the isomorphism $\p\colon
K_0(\A)\to K_0(\bfD^b(\A))$ from Lemma~\ref{le:gr_gr2} is an
\index{isometry}\emph{isometry}, that is, $$\langle
\p(x),\p(y)\rangle=\langle x,y\rangle\quad\text{for all}\quad x,y\in K_0(\A).$$

Now let $\A$ and $\B$ be $k$-linear abelian categories that are
Ext-finite and of finite global dimension. A $k$-linear
equivalence $F\colon\bfD^b(\A)\xto{\sim}\bfD^b(\B)$ of triangulated
categories induces an isometry $K_0(\A)\xto{\sim}K_0(\B)$ which is
defined by the commutativity of the following diagram.
$$\xymatrix{ K_0(\A)\ar[rr]^\sim\ar[d]_{\p_\A}&&K_0(\B)\ar[d]^{\p_\B}\\
K_0(\bfD^b(\A))\ar[rr]^{K_0(F)}&&K_0(\bfD^b(\B))}$$
This has the following consequence.

\begin{prop}\label{pr:euler}
Let $\A$ be a $k$-linear abelian category that is Ext-finite and of
finite global dimension.  Suppose that $\A$ has a tilting
object. Then the Grothendieck group $K_0(\A)$ is free of finite rank
and the Euler form associated to $\A$ is non-degenerate.
\end{prop}
\begin{proof}
We identify $K_0(\A)$ with $K_0(\mod \La)$, where $\La=\End_\A(T)$
for a tilting object $T$ in $\A$. Then the Grothendieck group
$K_0(\A)$ is free of finite rank by Lemma~\ref{le:gr_gr1}, and the
Euler form is non-degenerate by Lemma~\ref{le:euler_fin}.
\end{proof}

The following examples show that the existence of a tilting object is
an essential assumption for the Euler form to be non-degenerate.

\begin{exm}\label{ex:degen}
Let $\A$ be a $k$-linear length category that is Hom-finite and
satisfies Serre duality. Suppose that the Grothendieck group has
finite rank, and let $S_1,\dots,S_n$ be a representative set of the
simple objects. Then $\langle x,[S_i]\rangle=0$ for $x=\sum_j[S_j]$
and all $i$. Thus the Euler form is degenerate.
\end{exm}

\begin{exm}
Let $E$ be a smooth elliptic curve over some algebraically closed
field. Then the category of coherent sheaves on $E$ is hereditary and
satisfies Serre duality, but the corresponding Euler form is
degenerate. In fact, any two simple sheaves $S,T$ satisfy
$\langle-,[S]\rangle=\langle-,[T]\rangle$, but $[S]\neq[T]$ if $S$ and
$T$ are concentrated in different points of $E$.  Indeed, $[S]\neq
[T]$ follows from \cite[Chap.~II, Exercise~6.11]{H} and the fact that the
set of closed points is naturally identified with a subset of the
Picard group \cite[Chap.~IV, Example~1.3.7]{H}.
\end{exm}

Next we collect some further properties of the Grothendieck group and
its Euler form.

\begin{lem}\label{le:euler_tilt}
Let $\A$ be a $k$-linear abelian category that is hereditary,
Ext-finite, and has a non-degenerate Euler form. Suppose also
that $[A]\neq 0$  for each non-zero object $A$ in $\A$. Then an object
$T$ is a tilting object if and only if $\Ext^1_\A(T,T)=0$ and the
classes of the indecomposable direct summands of $T$ generate
$K_0(\A)$.
\end{lem}
\begin{proof}
Suppose first that $T$ is a tilting object with $\La=\End_\A(T)$.  The
isomorphism $K_0(\A)\xto{\sim} K_0(\mod\La)$ identifies (the classes
of) the indecomposable direct summands of $T$ with the indecomposable
projective $\La$-modules. Now one uses that the indecomposable
projective $\La$-modules generate $K_0(\mod\La)$; see the proof of
Lemma~\ref{le:euler_fin}

Conversely, suppose that $\Ext^1_\A(T,T)=0$ and that the
indecomposable direct summands of $T$ generate $K_0(\A)$.  Then there
exists for any non-zero object $A$ in $\A$ some indecomposable direct
summand $T'$ of $T$ such that $\langle[T'],[A]\rangle\neq 0$. Thus
$\Ext_\A^*(T,A)\neq 0$, and it follows that $T$ is a tilting object.
\end{proof}

A full subcategory $\B$ of an abelian category $\A$ is called
\index{subcategory!exact abelian} \emph{exact abelian} if $\B$ is an
abelian category and the inclusion functor is exact.

\begin{lem}\label{le:groth_onepoint}
Let $\A$ be an abelian category and $\B$ an exact abelian subcategory
such that the inclusion admits an exact left adjoint. Let
$\C={^\perp\B}$.  Then $K_0(\A)= K_0(\B)\oplus K'_0(\C)$, where
$K'_0(\C)$ denotes the image of the canonical map $K_0(\C)\to K_0(\A)$.
\end{lem}

Let $i\colon\B\to\A$ be the inclusion and $i_\la$ its left
adjoint. Observe that $\C=\Ker i_\la$ is a Serre subcategory of $\A$ by
Proposition~\ref{pr:perp}. Thus the inclusion $\C\to\A$ induces a
linear map $K_0(\C)\to K_0(\A)$.

\begin{proof}
We have $i_\la i\cong\Id_\B$ and therefore $K_0(i)$ identifies
$K_0(\B)$ with a direct summand of $K_0(\A)$.  The kernel of $K_0(i_\la)$ equals
$K'_0(\C)$, since there is an exact sequence $0\to A'\to A\to ii_\la
A\to A''\to 0$ for each object $A$ in $\A$ with $A',A''$ in $\C$; see
Proposition~\ref{pr:section}.
\end{proof}

The following lemma describes more specifically the term $ K'_0(\C)$
in the decomposition $K_0(\A)\cong K_0(\B)\oplus K'_0(\C)$.

\begin{lem}\label{le:euler_onepoint}
Let $\A$ be a $k$-linear abelian category that is Ext-finite and of
finite global dimension. Suppose that $K_0(\A)$ is free of
finite rank. Let $\B$ be an exact abelian subcategory such
that the inclusion admits an exact left adjoint and $^\perp\B$ is
equivalent to $\mod\Delta$ for some division ring $\Delta$.  Then
$K_0(\A)\cong K_0(\B)\oplus\bbZ$ and
$$\disc\langle-,-\rangle_\A=\dim_k\Delta\cdot\disc\langle-,-\rangle_\B.$$
\end{lem}
\begin{proof}
Let $i\colon\B\to\A$ be the inclusion and denote by $i_\la$ its left
adjoint. Then $^\perp\B=\Ker i_\la=\add S$ for some simple object $S$
with $\Delta\cong\End_\A(S)$; see Proposition~\ref{pr:perp}.
Lemma~\ref{le:groth_onepoint} implies that $K_0(\A)= K_0(\B)\oplus\bbZ
[S]$. Observe that $n[S]\neq 0$ for $n\neq 0$ in $\bbZ$, since
$\langle n[S],[S]\rangle\neq 0$.  Thus $K_0(\A)\cong
K_0(\B)\oplus\bbZ$. The formula for $\disc\langle-,-\rangle_\A$ follows
since $\langle[S],[B]\rangle=0$ for every object $B$ in $\B$.
\end{proof}

\section{Expansions of abelian categories}\label{se:extensions}

In this section we introduce the concept of expansion and contraction
for abelian categories.\footnote{The authors are indebted to Claus
Michael Ringel for suggesting the terms `expansion' and
`contraction'.}  Roughly speaking, an expansion is a fully faithful
and exact functor $\B\to\A$ between abelian categories that admits an
exact left adjoint and an exact right adjoint. In addition one
requires the existence of simple objects $S_\la$ and $S_\rho$ in $\A$
such that $S_\la^\perp=\B={^\perp S_\rho}$, where $\B$ is viewed as a
full subcategory of $\A$. In fact, these simple objects are related by
an exact sequence $0\to S_\rho\to S\to S_\la\to 0$ in $\A$ such that
$S$ is a simple object in $\B$. In terms of the Ext-quivers of $\A$
and $\B$, the expansion $\B\to\A$ turns the vertex $S$ into an arrow
$S_\la\to S_\rho$. On the other hand, $\B$ is a contraction of $\A$ in
the sense that the left adjoint of $\B\to\A$ identifies $S_\rho$ with
$S$, whereas the right adjoint identifies $S_\la$ with $S$.

In the following we use the term `expansion' but there are interesting
situations where `contraction' yields a more appropriate point of
view. So one should think of a process having two directions that are
opposite to each other.

Further material about expansions can be found in \cite{CK}.

\subsection{Left and right expansions}

Let $\A$ be an abelian category.  Recall that a full subcategory
$\B$ of $\A$ is called exact abelian if $\B$ is an abelian category
and the inclusion functor is exact.

Now let $i\colon \B\to\A$ be a fully faithful and exact functor
between abelian categories. It is convenient to identify $\B$ with the
essential image of $i$, which means that $\B$ is an exact abelian
subcategory of $\A$. We call the functor $i$ a \index{left expansion}
\emph{left expansion} if the following conditions are satisfied:
\begin{enumerate}
\item The functor $\B\to\A$ admits an exact left adjoint.
\item The category $^\perp\B$ is equivalent to $\mod\Delta$ for some division ring $\Delta$.
\item $\Ext^2_\A(A,B)=0$ for all $A,B\in{^\perp\B}$.
\end{enumerate}
The functor $\B\to\A$ is called \index{right expansion} \emph{right
expansion} if the dual conditions are satisfied.

\begin{lem}\label{le:leftmax}
Let $i\colon\B\to\A$ be a left expansion of abelian categories. Denote
by $i_\la$ its left adjoint and set $\C=\Ker i_\la$.
\begin{enumerate}
\item The category $\C$ is a Serre subcategory of $\A$ satisfying
$\C={^\perp\B}$ and $\C^\perp=\B$.
\item The composite $\B\xto{i}\A\xto{\can}\A/\C$ is
an equivalence and the left adjoint $i_\la$
induces a quasi-inverse $\A/\C\xto{\sim}\B$.
\item $\Ext^n_\B(i_\la A,B)\cong\Ext^n_\A(A,iB)$ for all $A\in\A$, $B\in\B$, and $n\geq 0$.
\end{enumerate}
\end{lem}
\begin{proof}
Part (1) and (2) follow from Proposition~\ref{pr:perp}. It remains to
prove (3). The case $n=0$ is clear. For $n\geq 1$, the isomorphism
sends a class $[\xi]$ in $\Ext^n_\B(i_\la A,B)$ to $[(i\xi).\eta_A]$
in $\Ext_\A^n(A,iB)$, where $\eta_A\colon A\rightarrow ii_\la(A)$ is
the unit of the adjoint pair and $(i\xi).\eta_A$ denotes the pullback
of $i\xi$ along $\eta_A$.
\end{proof}

An object $S$ in $\A$ is called \index{object!localizable}
\emph{localizable} if the following conditions are satisfied:
\begin{enumerate}
\item The object $S$ is simple.
\item $\Hom_\A(S,A)$ and $\Ext^1_\A(S,A)$ are of finite length over
$\End_\A(S)$ for all $A\in\A$.
\item $\Ext^1_\A(S,S)=0$ and $\Ext^2_\A(S,A)=0$ for all $A\in\A$.
\end{enumerate}
An object $S$ is \index{object!colocalizable} \emph{colocalizable} if
the dual conditions are satisfied.

\begin{lem}\label{le:localizable}
Let $\A$ be an abelian category and $\B$ an exact abelian
subcategory. Then the following are equivalent:
\begin{enumerate}
\item The inclusion $\B\to\A$ is a left expansion.
\item There exists a localizable object $S\in\A$ such that $S^\perp=\B$.
\end{enumerate}
\end{lem}
\begin{proof}
(1) $\Rightarrow$ (2): Let $S$ be an indecomposable object in
$^\perp\B$.  Then $S$ is a simple object and $\Ext_\A^1(S,S)=0$ since
$^\perp\B=\add S$ is semisimple.  For each object $A$ in $\A$, we use the
natural exact sequence \eqref{eq:local}
$$0\lto A'\lto A\lto[\eta_A] \bar A\lto A''\lto 0$$ with $A',A''\in
{^\perp\B}$ and $\bar A\in\B$. This sequence induces the following
isomorphisms.
\begin{gather*}\Hom_\A(S,A')\lto[\sim]\Hom_\A(S,A)\\
\Ext_\A^1(S,A)\lto[\sim]\Ext_\A^1(S,\Im\eta_A)\rto[\sim]\Hom_\A(S,A'')
\end{gather*}
Here we use the condition $\Ext^2_\A(S,S)=0$.  It follows that
$\Hom_\A(S,A)$ and $\Ext^1_\A(S,A)$ are of finite length over
$\End_\A(S)$. Now observe that the functor sending $A$ to
$\Hom_\A(S,A'')$ is right exact. Thus $\Ext^2_\A(S,A)=0$ by
Lemma~\ref{lem:hereditary}. Finally, $S^\perp=\B$ follows from
Proposition~\ref{pr:perp}.

(2) $\Rightarrow$ (1): A left adjoint $i_\la$ of the inclusion
$\B\to\A$ is constructed as follows.  Fix an object $A$ in $\A$.
There exists an exact sequence $0\to A\to B\to S^n\to 0$ for some
$n\ge 0$ such that $\Ext_\A^1(S,B)=0$ since $\Ext^1_\A(S,A)$ is of
finite length over $\End_\A(S)$. Now choose a morphism $S^m\to B$ such
that the induced map $\Hom_\A(S,S^m)\to\Hom_\A(S,B)$ is surjective and
let $\bar A$ be its cokernel. It is easily checked that the composite
$A\to B\to \bar A$ is the universal morphism into $S^\perp$. Thus we
define $i_\la A=\bar A$.

Next observe that the kernel and cokernel of the adjunction morphism
$A\to i_\la A$ belong to $\C=\add S$ for each object $A$ in $\A$.
Moreover, $\C$ is a Serre subcategory of $\A$ since $S$ is simple
and $\Ext^1_\A(S,S)=0$. Thus we can apply
Proposition~\ref{pr:section} and infer that the composite
$\A\xto{i_\la}\C^\perp\xto{\sim}\A/\C$ is the quotient functor with
respect to $\C$. Therefore the left adjoint $i_\la$ is exact. We
have $^\perp\B=\C$ by Proposition~\ref{pr:perp}, and $\Hom_\A(S,-)$
induces an equivalence $\C\xto{\sim}\mod\End_\A(S)$.  Thus the
inclusion $\B\to\A$ is a left expansion.
\end{proof}

\subsection{Expansions of abelian categories}

A fully faithful and exact functor $\B\to\A$ between abelian
categories is by definition an \index{expansion} \emph{expansion} of
abelian categories if the functor is a left and a right expansion.

Let $i\colon \B\to\A$ be an expansion of abelian categories.
Then we identify $\B$ with the essential image of $i$. We denote by
$i_\la$ the left adjoint of $i$ and by $i_\rho$ the right adjoint of
$i$. We choose an indecomposable object $S_\la$ in $^\perp\B$ and an
indecomposable object $S_\rho$ in $\B^\perp$. Thus $^\perp\B=\add
S_\la$ and $\B^\perp=\add S_\rho$. Finally, set $\bar
S=i_\la(S_\rho)$.

An expansion $i\colon\B\to\A$ is called \index{expansion!split}
\emph{split} if $\B^\perp={^\perp\B}$. If the expansion is non-split,
then the exact sequences \eqref{eq:local} for $S_\la$ and $S_\rho$ are
of the form
\begin{equation}\label{eq:locseq}
0\to S_\rho\to ii_\la(S_\rho)\to S_\la^l\to 0\quad\text{and}\quad
0\to S_\rho^r\to ii_\rho(S_\la)\to S_\la\to 0
\end{equation}
for some integers $l,r\geq 1$. In Lemma~\ref{le:extsimple}, we see that $l=1=r$.

\begin{lem}
Let $\B\to\A$ be an expansion of abelian categories. Then the
following are equivalent:
\begin{enumerate}
\item The expansion $\B\to\A$ is split.
\item $\A=\B\amalg\C$ for some Serre subcategory $\C$ of $\A$.
\item $\B$ is a Serre subcategory of $\A$.
\end{enumerate}
\end{lem}
\begin{proof}
(1) $\Rightarrow$ (2): Take $\C={^\perp\B}=\B^\perp$.

(2) $\Rightarrow$ (3): An object $A\in\A$ belongs to $\B$ if and only
if $\Hom_\A(A,B)=0$ for all $B\in\C$. Thus $\B$ is closed under taking
quotients and extensions. The dual argument shows that $\B$ is closed
under taking subobjects.

(3) $\Rightarrow$ (1): If the expansion is non-split, then the
sequences in \eqref{eq:locseq} show that $\B$ is not a Serre
subcategory.
\end{proof}

\begin{lem}\label{le:extsimple}
Let $i\colon \B\to\A$ be a non-split expansion of abelian
categories.
\begin{enumerate}
\item The object $\bar S=i_\la(S_\rho)$ is a simple object in $\B$ and
isomorphic to $i_\rho(S_\la)$.
\item The functor $i_\la$ induces an equivalence
 $\B^\perp\xto{\sim}\add\bar S$.
\item The functor $i_\rho$ induces an equivalence
$^\perp\B\xto{\sim}\add\bar S$.
\end{enumerate}
\end{lem}
\begin{proof}
(1) Let $\p\colon i_\la(S_\rho)\to A$ be a non-zero morphism in
$\B$. Adjunction takes this to a monomorphism $S_\rho\to A$ in
$\A$ since $S_\rho$ is simple. Applying $i_\la$ gives back a morphism which is isomorphic to
$\p$. This is a monomorphism since $i_\la$ is exact. Thus
$i_\la(S_\rho)$ is simple.

Now apply $i_\rho$ to the first and $i_\la$ to the second sequence
in \eqref{eq:locseq}. Note that by adjunction $i_\la i\cong {\rm
Id}_\B\cong i_\rho i$. Then we have
$$i_\la(S_\rho)\cong i_\rho(S_\la)^l \mbox{ and } i_\rho(S_\la)\cong i_\la(S_\rho)^r.$$
This implies $l=1=r$ and therefore $i_\la(S_\rho)\cong i_\rho(S_\la)$.

(2) We have a sequence of isomorphisms
$$\Hom_\A(S_\rho,S_\rho)\lto[\sim]\Hom_\A(S_\rho,ii_\la(S_\rho))\lto[\sim]
\Hom_\B(i_\la(S_\rho),i_\la(S_\rho))$$ which takes a morphism $\p$ to
$i_\la\p$. Thus $i_\la$ induces an equivalence $\add
S_\rho\xto{\sim}\add i_\la(S_\rho)$.

(3) Follows from (2) by duality.
\end{proof}

An expansion $\B\to\A$ of abelian categories determines a division
ring $\Delta$ such that $^\perp\B$ and $\B^\perp$ are equivalent to
$\mod\Delta$; we call $\Delta$ the \index{expansion!division ring of}
\emph{associated division ring}.

Fix an expansion $i\colon\B\to\A$ with associated division ring
$\Delta$. We identify the perpendicular categories of $\B$ with
$\mod\Delta$ via the equivalences
${^\perp\B}\xto{\sim}\mod\Delta\xleftarrow{\sim}\B^\perp$. There are
inclusions $j\colon{^\perp\B}\to\A$ and $k \colon{\B^\perp}\to\A$ with
adjoints $j_\rho$ and $k_\la$. These functors yield the following
diagram.
$$\xymatrixrowsep{3pc} \xymatrixcolsep{3pc}\xymatrix{
\B\,\ar[rr]|-{i}&&\,\A\,
\ar@<1.0ex>[rr]|-{j_\rho}\ar@<-1.0ex>[rr]|-{k_\la}
\ar@<1.2ex>[ll]^-{i_\rho}\ar@<-1.2ex>[ll]_-{i_\la}&&
\,\mod\Delta\ar@<2.2ex>[ll]^-{k}\ar@<-2.2ex>[ll]_-{j}
}$$ Note that this diagram induces a recollement of triangulated
categories \cite{BBD}:
$$\xymatrixrowsep{3pc} \xymatrixcolsep{3pc}\xymatrix{
  \bfD^b(\B)\,\ar[rr]|-{\bfD^b(i)}&&\,\bfD^b(\A)\, \ar[rr]|-{}
  \ar@<1.2ex>[ll]^-{\bfD^b(i_\rho)}\ar@<-1.2ex>[ll]_-{\bfD^b(i_\la)}&&
  \,\bfD^b(\mod\Delta)\ar@<1.2ex>[ll]^-{\bfD^b(k)}\ar@<-1.2ex>[ll]_-{\bfD^b(j)}
}$$ Indeed, the labeled functors are part of a recollement, and
therefore the right adjoint of $\bfD^b(j)$ is isomorphic to
the left adjoint of $\bfD^b(k)$, both of which are isomorphic
to the quotient functor of $\bfD^b(\A)$ with respect to the
triangulated subcategory $\bfD^b(\B)$.

\subsection{Simple objects}
Let $i\colon\B\to\A$ be an expansion. The left adjoint $i_\la$ induces
a bijection between the isomorphism classes of simple objects of $\A$
that are different from $S_\la$, and the isomorphism classes of simple
objects of $\B$.  On the other hand, all simple objects of $\A$
correspond to simple objects of $\B$ via $i$.  All this is made
precise in the next lemma.

\begin{lem}\label{le:simple}
Let $i\colon\B\to\A$ be an expansion of abelian categories.
\begin{enumerate}
\item If $S$ is a simple object in $\B$ and $S\not\cong \bar S$, then
$iS$ is simple in $\A$ and $i_\la iS\cong S$.
\item There is an exact sequence $0\to S_\rho\to i\bar S\to S_\la\to
0$ in $\A$, provided the expansion $\B\to\A$ is non-split.
\item If $S$ is a simple object in $\A$ and $S\not\cong S_\la$, then
$i_\la S$ is simple in $\B$. Moreover, $S\cong i i_\la S$ if
$S\not\cong S_\rho$.
\end{enumerate}
\end{lem}
\begin{proof}
(1) Let $0\neq U\subseteq i S$ be a subobject. Then $\Hom_\B(i_\la
    U,S)\cong\Hom_\A(U,iS)\neq 0$ shows that $U\not\in\Ker i_\la$.
    Thus $i_\la U=S$, and therefore $iS/U$ belongs to $\Ker i_\la=\add
    S_\la$. On the other hand, $\Hom_\A(iS,S_\la)\cong\Hom_\B(S,\bar
    S)=0$. Thus $iS/U=0$, and it follows that $iS$ is simple. Finally
    observe that $i_\la iA\cong A$ for every object $A$ in $\B$.

(2) Take the exact sequence in \eqref{eq:locseq}.

(3) This is a general fact: A quotient functor $\A\to\A/\C$ takes each
    simple object of $\A$ not belonging to $\C$ to a simple object of
    $\A/\C$; see Lemma \ref{le:quotient_simple}.  Here, we take $\C=\Ker i_\la$ and identify $i_\la$ with
    the corresponding quotient functor.

If $S\not\cong S_\rho$, then $i_\la S\not\cong \bar S$ and therefore
$ii_\la S$ is simple by (1). Thus the canonical morphism $S\to ii_\la
S$ is an isomorphism.
\end{proof}

The Ext-groups of most simple objects in $\A$ can be computed from
appropriate Ext-groups in $\B$. This follows from an adjunction
formula; see Lemma~\ref{le:leftmax}. The remaining cases are treated
in the following lemma.

\begin{lem}\label{le:ext_simple}
Let $i\colon\B\to\A$ be a non-split expansion of abelian categories.
\begin{enumerate}
\item $\Hom_\A(S_\la,S_\la)\cong\Ext_\A^1(S_\la,S_\rho)\cong\Hom_\A(S_\rho,S_\rho)$.
\item $\Ext_\B^n(\bar S,\bar S)\cong\Ext_\A^n(S_\rho,S_\la)$ for $n\geq 1$.
\end{enumerate}
\end{lem}
\begin{proof}
(1) Applying $\Hom_\A(S_\la,-)$ to the first sequence in
\eqref{eq:locseq} yields the isomorphism
$\Hom_\A(S_\la,S_\la)\cong\Ext_\A^1(S_\la,S_\rho)$. The other
isomorphism is dual.

(2) We have
$$\Ext_\B^n(i_\la(S_\rho),i_\la(S_\rho))\cong\Ext_\A^n(S_\rho,ii_\la(S_\rho))\cong\Ext_\A^n(S_\rho,S_\la),$$
where the first isomorphism follows from Lemma~\ref{le:leftmax} and
the second from the first sequence in \eqref{eq:locseq}.
\end{proof}

\begin{prop}\label{pr:onepoint_length}
Let $i\colon\B\to\A$ be an expansion of abelian
categories.
\begin{enumerate}
\item The functor $i$ and its adjoints $i_\la$ and $i_\rho$
send  finite length objects to finite length objects.
\item The restriction $\B_0\to\A_0$ is an expansion of
abelian categories.
\item The induced functor $\B/\B_0\to\A/\A_0$ is an equivalence.
\end{enumerate}
\end{prop}
\begin{proof}
(1) follows from Lemma~\ref{le:simple} and (2) is an immediate
    consequence.

(3) Let $\C=\Ker i_\la$. The functor $i_\la$ induces an equivalence
    $\A/\C\xto{\sim}\B$. Moreover, $\C\subseteq\A_0$ and $i_\la$
    identifies $\A_0/\C$ with $\B_0$ by (1).  It follows from
    Lemma~\ref{le:noeth} that $i_\la$ induces an equivalence
    $\A/\A_0\xto{\sim}\B/\B_0$. This is a quasi-inverse for the
    functor $\B/\B_0\to\A/\A_0$ induced by $i$.
\end{proof}

The Ext-quiver $\Si(\A)$ of $\A$ can be computed explicitly from the
Ext-quiver $\Si(\B)$, and vice versa. The following statement makes this precise.

\begin{prop}
Let $i\colon\B\to\A$ be a non-split expansion of abelian
categories. The functor induces a bijection
$$\Si_0(\B)\smallsetminus\{\bar
S\}\lto[\sim]\Si_0(\A)\smallsetminus\{S_\la,S_\rho\},$$ and for each
pair  $U,V\in\Si_0(\B)\smallsetminus\{\bar S\}$ the
following identities:
\[
\d_{iU,iV}=\d_{U,V},\quad \d_{iU,S_\la}=\d_{U,\bar S},\quad
\d_{S_\rho,iV}=\d_{\bar S,V}, \quad \d_{S_\rho,S_\la}=\d_{\bar S,\bar
  S},\quad \d_{S_\la,S_\rho}=(1,1).
\]
\end{prop}
\begin{proof}
Combine Lemmas~\ref{le:leftmax}, \ref{le:simple}, and
\ref{le:ext_simple}.
\end{proof}
The following diagram shows how $\Si(\B)$ and $\Si(\A)$ are
related.\footnote{The expansion of the vertex $\bar S$ into an arrow
  linking $S_\la$ with $S_\rho$ justifies the term `expansion of
  abelian categories'.}  \setlength{\unitlength}{1.5pt}
\[
\begin{picture}(120,25)
\put(-55,8){$\Si(\B)$}
\put(-05,8){$\bar S$}
\put(-30,02){\circle*{1}}
\put(-30,10){\circle*{1}}
\put(-30,18){\circle*{1}}
\put(25,02){\circle*{1}}
\put(25,10){\circle*{1}}
\put(25,18){\circle*{1}}
\put(-22,21){\vector(2,-1){15}}
\put(-22,-01){\vector(2,1){15}}
\put(-22,10){\vector(1,0){15}}
\put(02,11){\vector(2,1){15}}
\put(02,09){\vector(2,-1){15}}
\put(55,8){$\Si(\A)$}
\put(105,8){$S_\la$}
\put(137,8){$S_\rho$}
\put(80,02){\circle*{1}}
\put(80,10){\circle*{1}}
\put(80,18){\circle*{1}}
\put(170,02){\circle*{1}}
\put(170,10){\circle*{1}}
\put(170,18){\circle*{1}}
\put(88,21){\vector(2,-1){15}}
\put(88,-01){\vector(2,1){15}}
\put(88,10){\vector(1,0){15}}
\put(115,10){\vector(1,0){20}}
\put(147,11){\vector(2,1){15}}
\put(147,09){\vector(2,-1){15}}
\put(118,13){$\scriptstyle{(1,1)}$}
\end{picture}
\]

\begin{exm}
(1) Let $k$ be a field and $\Ga_n$ a quiver of extended Dynkin type
$\tilde\bbA_n$ with cyclic orientation. Consider the category
$\A=\rep_0(\Ga_n,k)$ of all finite dimensional nilpotent
representations. Fix a simple object $S$. This object is
(co)localizable and $S^\perp={^\perp(\tau S)}$. Thus the inclusion
$S^\perp\to\rep_0(\Ga_n,k)$ is an expansion, and $S^\perp$ is
equivalent to $\rep_0(\Ga_{n-1},k)$.

The following diagram depicts the shape of the \index{Auslander-Reiten
quiver}\emph{Auslander-Reiten quiver} of $\rep_0(\Ga_n,k)$.  Thus the
vertices represent the indecomposable objects, and there is an arrow
between two indecomposable objects if and only if there exists an
irreducible morphism; see Lemma~\ref{le:irr}.
\setlength{\unitlength}{1.5pt}
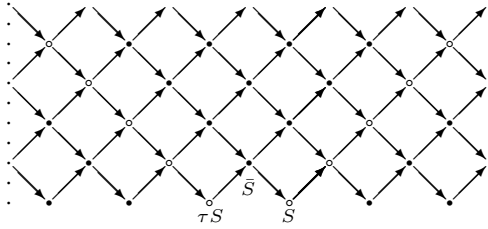
\[
\begin{picture}(-20,60)
\put(-12,07){$\scriptstyle{\bar S}$}
\put(-02,00){$\scriptstyle{S}$}
\put(-23,00){$\scriptstyle{\tau S}$}


\put(-70, 05){\circle*{.2}}

\put(-70, 10){\circle*{.2}}

\put(-70, 15){\circle*{.2}}

\put(-70, 20){\circle*{.2}}

\put(-70, 25){\circle*{.2}}

\put(-70, 30){\circle*{.2}}

\put(-70, 35){\circle*{.2}}

\put(-70, 40){\circle*{.2}}

\put(-70, 45){\circle*{.2}}

\put(-70, 50){\circle*{.2}}

\put(-70, 55){\circle*{.2}}

\put(50, 05){\circle*{.2}}

\put(50, 10){\circle*{.2}}

\put(50, 15){\circle*{.2}}

\put(50, 20){\circle*{.2}}

\put(50, 25){\circle*{.2}}

\put(50, 30){\circle*{.2}}

\put(50, 35){\circle*{.2}}

\put(50, 40){\circle*{.2}}

\put(50, 45){\circle*{.2}}

\put(50, 50){\circle*{.2}}

\put(50, 55){\circle*{.2}}


\put(-60,05){\circle*{1.2}}
\put(-40,05){\circle*{1.2}}
\put(-20,05){\circle{1.2}}
\put(00,05){\circle{1.2}}
\put(20,05){\circle*{1.2}}
\put(40,05){\circle*{1.2}}

\put(-60,25){\circle*{1.2}}
\put(-40,25){\circle{1.2}}
\put(-20,25){\circle*{1.2}}
\put(00,25){\circle*{1.2}}
\put(20,25){\circle{1.2}}
\put(40,25){\circle*{1.2}}

\put(-60,45){\circle{1.2}}
\put(-40,45){\circle*{1.2}}
\put(-20,45){\circle*{1.2}}
\put(00,45){\circle*{1.2}}
\put(20,45){\circle*{1.2}}
\put(40,45){\circle{1.2}}

\put(-50,15){\circle*{1.2}}
\put(-30,15){\circle{1.2}}
\put(-10,15){\circle*{1.2}}
\put(10,15){\circle{1.2}}
\put(30,15){\circle*{1.2}}

\put(-50,35){\circle{1.2}}
\put(-30,35){\circle*{1.2}}
\put(-10,35){\circle*{1.2}}
\put(10,35){\circle*{1.2}}
\put(30,35){\circle{1.2}}

\put(01,06){\vector(1,1){8}}
\put(01,26){\vector(1,1){8}}
\put(01,46){\vector(1,1){8}}
\put(01,06){\vector(1,1){8}}

\put(-59,24){\vector(1,-1){8}}
\put(-59,44){\vector(1,-1){8}}

\put(-39,24){\vector(1,-1){8}}
\put(-39,44){\vector(1,-1){8}}

\put(-19,24){\vector(1,-1){8}}
\put(-19,44){\vector(1,-1){8}}

\put(01,24){\vector(1,-1){8}}
\put(01,44){\vector(1,-1){8}}

\put(21,24){\vector(1,-1){8}}
\put(21,44){\vector(1,-1){8}}

\put(41,24){\vector(1,-1){8}}
\put(41,44){\vector(1,-1){8}}

\put(-69,14){\vector(1,-1){8}}
\put(-69,34){\vector(1,-1){8}}
\put(-69,54){\vector(1,-1){8}}

\put(-49,14){\vector(1,-1){8}}
\put(-49,34){\vector(1,-1){8}}
\put(-49,54){\vector(1,-1){8}}

\put(-29,14){\vector(1,-1){8}}
\put(-29,34){\vector(1,-1){8}}
\put(-29,54){\vector(1,-1){8}}

\put(-09,14){\vector(1,-1){8}}
\put(-09,34){\vector(1,-1){8}}
\put(-09,54){\vector(1,-1){8}}

\put(11,14){\vector(1,-1){8}}
\put(11,34){\vector(1,-1){8}}
\put(11,54){\vector(1,-1){8}}

\put(31,14){\vector(1,-1){8}}
\put(31,34){\vector(1,-1){8}}
\put(31,54){\vector(1,-1){8}}

\put(-69,16){\vector(1,1){8}}
\put(-69,36){\vector(1,1){8}}

\put(-49,16){\vector(1,1){8}}
\put(-49,36){\vector(1,1){8}}

\put(-29,16){\vector(1,1){8}}
\put(-29,36){\vector(1,1){8}}

\put(-09,16){\vector(1,1){8}}
\put(-09,36){\vector(1,1){8}}

\put(11,16){\vector(1,1){8}}
\put(11,36){\vector(1,1){8}}

\put(31,16){\vector(1,1){8}}
\put(31,36){\vector(1,1){8}}

\put(-59,06){\vector(1,1){8}}
\put(-59,26){\vector(1,1){8}}
\put(-59,46){\vector(1,1){8}}

\put(-39,06){\vector(1,1){8}}
\put(-39,26){\vector(1,1){8}}
\put(-39,46){\vector(1,1){8}}

\put(-19,06){\vector(1,1){8}}
\put(-19,26){\vector(1,1){8}}
\put(-19,46){\vector(1,1){8}}

\put(01,06){\vector(1,1){8}}
\put(01,26){\vector(1,1){8}}
\put(01,46){\vector(1,1){8}}

\put(21,06){\vector(1,1){8}}
\put(21,26){\vector(1,1){8}}
\put(21,46){\vector(1,1){8}}

\put(41,06){\vector(1,1){8}}
\put(41,26){\vector(1,1){8}}
\put(41,46){\vector(1,1){8}}
\end{picture}
\]
Note that the two dotted lines are identified and that there are $n+1$
simple objects which sit at the bottom ($n=5$ in this example). The
Serre functor $\tau$ induces an automorphism of order $n+1$, that is,
$\tau^{n+1}=\Id_\A$. The indecomposables not belonging to $S^\perp$
are represented by circles. Thus $S$ and $\tau S$ disappear, while
$\bar S$ becomes a simple object in $S^\perp$.

Suppose that $n>1$. Then $\bar S$ is a (co)localizable object of
$S^\perp$ and this gives another expansion $\bar S^\perp\to
S^\perp$. Iterating this formation of perpendicular categories
yields a chain $\A^0\to\A^1\to\cdots\to\A^{n+1}=\A$ of expansions
such that $\A^i$ is equivalent to $\rep_0(\Ga_{i},k)$ for each $i$.
The category $\A^0$ has a unique simple object $S_0$ which the
inclusion $\A^0\to\A$ sends to $S^{[n+1]}$. The induced map
$K_0(\A^0)\to K_0(\A)$ sends the class $[S_0]$ to
$\sum_{i=0}^n[\tau^i S]$.\footnote{This expansion of the class
$[S_0]$ in $K_0(\A)$ again explains the term `expansion of abelian
categories'.}

(2) Let $k$ be a field and consider a finite quiver $\Ga$ without
oriented cycles having two vertices $a,b$ that are joined by an
arrow $\xi\colon a\to b$ which is the unique arrow starting at $a$
and the unique arrow terminating at $b$.
\setlength{\unitlength}{1.5pt}
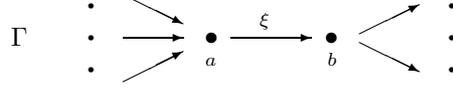
\[
\begin{picture}(80,25)
\put(-20,8){$\Ga$}
\put(28.5,03){$\scriptstyle{a}$}
\put(59,03){$\scriptstyle{b}$}
\put(30,10){\circle*{2.5}}
\put(60,10){\circle*{2.5}}
\put(00,02){\circle*{1}}
\put(00,10){\circle*{1}}
\put(00,18){\circle*{1}}
\put(90,02){\circle*{1}}
\put(90,10){\circle*{1}}
\put(90,18){\circle*{1}}
\put(08,21){\vector(2,-1){15}}
\put(08,-01){\vector(2,1){15}}
\put(08,10){\vector(1,0){15}}
\put(35,10){\vector(1,0){20}}
\put(67,11){\vector(2,1){15}}
\put(67,09){\vector(2,-1){15}}
\put(42,13){$\scriptstyle{\xi}$}
\end{picture}
\]
We obtain a new quiver $\Ga'$ by identifying $a$ and $b$ and removing
$\xi$. Let $\A=\rep(\Ga,k)$ and consider the full subcategory $\B$ of
representations such that $\xi$ is represented by an isomorphism. Note
that $\B$ is equivalent to $\rep(\Ga',k)$.

The simple representations $S_a$ and $S_b$ are (co)localizing objects
of $\A$ and they are joined by an almost split sequence $0\to S_b\to E\to
S_a\to 0$. The Auslander-Reiten formulae
$$D\Ext_\A^1(-,S_b)\cong\Hom_\A(S_a,-)\quad\text{and}\quad
D\Ext_\A^1(S_a,-)\cong\Hom_\A(-,S_b)$$ imply that
$S_a^\perp=\B={^\perp S_b}$. Thus the inclusion functor $\B\to\A$ is
an expansion.
\end{exm}

\subsection{An Auslander-Reiten formula}
Given a non-split expansion $\B\to\A$, the corresponding
simple objects $S_\la$ and $S_\rho$ in $\A$ are related by an
Auslander-Reiten formula.

\begin{prop}\label{pr:ARformula}
Let $\B\to\A$ be a non-split expansion of abelian
categories and $\Delta$ its associated division ring. Then
$$D\Ext_\A^1(-,S_\rho)\cong\Hom_\A(S_\la,-)\quad\text{and}\quad
D\Ext_\A^1(S_\la,-)\cong\Hom_\A(-,S_\rho),$$ where
$D=\Hom_\Delta(-,\Delta)$ denotes the standard duality. In particular, any
non-split extension $0\to S_\rho\to E\to S_\la\to$ is an almost split
sequence.
\end{prop}
\begin{proof}
Recall that ${^\perp\B}=\add S_\la$ and $\B^\perp=\add S_\rho$.  Fix
an object $A$ in $\A$ and consider the corresponding exact sequence
\eqref{eq:local}
$$0\lto A'\lto A\lto \bar A\lto A''\lto 0.$$ with $A',A''$ in
${^\perp\B}$ and $\bar A$ in $\B$.  The morphism $A'\to A$ induces the
first and the third isomorphism in the sequence below, while the second
isomorphism follows from the isomorphism
$\Hom_\A(S_\la,S_\la)\cong\Ext_\A^1(S_\la,S_\rho)$ in
Lemma~\ref{le:ext_simple}.
$$D\Ext_\A^1(A,S_\rho)\cong
D\Ext_\A^1(A',S_\rho)\cong\Hom_\A(S_\la,A')\cong\Hom_\A(S_\la,A).$$
The isomorphism $D\Ext_\A^1(S_\la,-)\cong\Hom_\A(-,S_\rho)$ follows
from the first by duality.

The argument given in the proof of Proposition~\ref{pr:Serreduality}
shows that any non-split extension $0\to S_\rho\to E\to S_\la\to$ is
an almost split sequence.
\end{proof}

\subsection{Decompositions}
Expansions of abelian categories respect decompositions of
abelian categories. The following lemma is a precise formulation of this fact.

\begin{lem}\label{le:decomp}
Let $i\colon\B\to\A$ be a non-split expansion of abelian
categories.
\begin{enumerate}
\item If $\A=\A_1\amalg\A_2$ is a decomposition, then
there exists a decomposition $i=\smatrix{i_1&0\\0&i_2}\colon
\B_1\amalg\B_2\to\A_1\amalg\A_2$ such that one of $i_1$ and $i_2$ is a
non-split expansion and the other is an equivalence.
\item If $\B=\B_1\amalg\B_2$ is a decomposition, then
there exists a decomposition $i=\smatrix{i_1&0\\0&i_2}\colon
\B_1\amalg\B_2\to\A_1\amalg\A_2$ such that one of $i_1$ and $i_2$ is a
non-split expansion and the other is an equivalence.
\item If $i'\colon\B'\to\A'$ is an equivalence of abelian categories,
 then $\smatrix{i&0\\0&i'}\colon \B\amalg\B'\to\A\amalg\A'$
 is a non-split expansion.
\end{enumerate}
Therefore $\A$ is connected if and only if $\B$ is connected.
\end{lem}
\begin{proof}
We identify $\B$ with the essential image of $i$ and choose a simple
object $S$ in $\A$ with $S^\perp=\B$.

(1) The decomposition of $\A$ restricts to a decomposition of $\B$ by
taking $\B_\a=\B\cap\A_\a$ for $\a=1,2$. Now suppose without loss of
generality that $S$ belongs to $\A_1$. Then $i_1$ is a non-split
expansion and $i_2$ is an equivalence.

(2) Let $\bar S=i_\rho S$ and suppose without loss of generality
that $\bar S$ belongs to $\B_1$.  Then $\B_2\subseteq {^\perp S}$ and
this yields a decomposition $\A=\A_1\amalg\A_2$ if we set $\A_2=\B_2$
and $\A_1=\A_2^\perp={^\perp\A_2}$. It follows that $i_1$ is a non-split
expansion and $i_2$ is an equivalence.

(3) Clear.
\end{proof}

\subsection{Dimensions}
We compute the global dimension for an expansion of abelian
categories.
\begin{lem}\label{le:cohdim}
Let $i\colon\B\to\A$ be a non-split expansion of abelian
categories.
\begin{enumerate}
\item $\gldim\A=\max\{1,\gldim\B\}$.
\item $\A$ has non-zero projective objects if and only if $\B$ has
non-zero projective objects.
\end{enumerate}
\end{lem}
\begin{proof}
(1) Use the adjunction formula for $\Ext^n(-,-)$ from
Lemma~\ref{le:leftmax} together with the fact that $\pdim A\leq 1$ for all
$A$ in $^\perp\B$.  Note also that $\Ext_\A^1(-,-)\neq 0$ by
Lemma~\ref{le:ext_simple}.

(2) We use the general fact that a functor between abelian categories
preserves projectivity if it admits an exact right adjoint.  Thus $i$
and its left adjoint $i_\la$ preserve projectivity. Note that there
are no non-zero projectives in the kernel of $i_\la$ by
Lemma~\ref{le:ext_simple}.
\end{proof}

Next we discuss  Ext-finiteness for expansions of abelian categories.

\begin{lem}\label{le:extfin}
Let $k$ be a commutative ring and $i\colon\B\to\A$ a non-split
expansion of $k$-linear abelian categories. Then $\A$ is Ext-finite if
and only if $\B$ is Ext-finite.
\end{lem}
\begin{proof}
We use the adjunction formula for $\Ext^n(-,-)$ from
Lemma~\ref{le:leftmax}. Note that these isomorphisms are $k$-linear
since we assume the functor $i$ to be $k$-linear. It is clear that
$\B$ is Ext-finite if $\A$ is Ext-finite. To prove the converse, fix a
simple object $S$ in $\A$ such that $S^\perp=\B$, and an arbitrary
object $C$ in $\A$. Then $\End_\A(S)\cong\End_\B(\bar S)$ for some
simple object $\bar S$ in $\B$; see Lemma~\ref{le:extsimple}. Thus
$\End_\A(S)$ is of finite length over $k$, and it follows that
$\Ext_\A^n(S,C)$ is of finite length over $k$ for all $n\ge 0$ since
the object $S$ is localizable; see Lemma~\ref{le:localizable}. On the
other hand, $\Ext_\A^n(B,C)$ is of finite length over $k$ for all $B$
in $\B$ by the adjunction formula for $\Ext^n(-,-)$ from
Lemma~\ref{le:leftmax}.  Now choose $A$ in $\A$ and apply
$\Ext^n_\A(-,C)$ to the natural exact sequence \eqref{eq:local}
$$0\lto A'\lto A\lto[\eta_A] \bar A\lto A''\lto 0$$ with $A',A''$ in
${^\perp\B}=\add S$ and $\bar A$ in $\B$. It follows that
$\Ext_\A^n(A,C)$ is of finite length over $k$ for all $n\ge 0$.
\end{proof}

\section{Coherent sheaves on the projective line}\label{se:coh}

In this section we discuss the category of coherent sheaves on the
projective line $\bbP^1_k$ over an arbitrary base field $k$. This is
a hereditary abelian category with finite dimensional Hom and Ext
spaces. Moreover, the category satisfies Serre duality and admits a
tilting object. Various structural properties can be derived from
these basic facts.

The projective line $\bbP^1_k$ is covered by two copies of the affine
line $\bbA^1_k$. Using this fact, we identify sheaves on $\bbP^1_k$
with `triples', that is, pairs of modules over the polynomial ring
$k[x]$ which are glued together by a glueing morphism.  In this way,
basic properties of sheaves are easily derived from properties of
modules over a polynomial ring in one variable.

Every coherent sheaf is the direct sum of a torsion-free sheaf and a
finite length sheaf. An indecomposable torsion-free sheaf is a line
bundle and an indecomposable finite length sheaf is uniserial. The
simple sheaves are parametrized by the closed points of $\bbP^1_k$,
and for each simple sheaf $S$ and $r>0$ there is a unique sheaf with
length $r$ and top $S$.  This yields a complete classification of all
indecomposable sheaves.

A classical theorem of Serre identifies the category of coherent
sheaves on $\bbP^1_k$ with the quotient category of the category of
finitely generated $\bbZ$-graded $k[x_0,x_1]$-modules modulo the Serre
subcategory of finite length modules. We use the concept of
dehomogenization to pass from graded $k[x_0,x_1]$-modules to modules
over $k[x]$. In geometric terms, this reflects the passage from the
projective to an affine line.

\subsection{Coherent sheaves on $\bbA^1_k$}

Let $k$ be a field and $\bbA^1_k$ the affine line over $k$.  The
polynomial ring $k[x]$ is the ring of regular functions and the
category of coherent sheaves $\coh\bbA^1_k$ is equivalent to the
module category $\mod k[x]$ via the global section functor.

Let $\Spec k[x]$ denote the set of prime ideals of $k[x]$. Note that
$k[x]$ is a principal ideal domain. Thus irreducible polynomials
correspond to non-zero prime ideals by taking a polynomial $P$ to the
ideal $(P)$ generated by $P$. A \index{point!closed} \emph{closed
point} of $\bbA^1_k$ is by definition a non-zero prime ideal $\frp$
and the \index{point!generic} \emph{generic point} is the zero ideal.

The following result describes the category $\mod_0 k[x]$ of torsion
modules and  the corresponding quotient category.

\begin{prop}\label{pr:aff}
Let $k[x]$ be the polynomial ring over a field $k$.
\begin{enumerate}
\item The functor which sends a $k[x]$-module $M$ to its family of
localizations $(M_\frp)_{\frp\in\Spec k[x]}$ induces an equivalence
$$\mod_0 k[x]\lto[\sim]\coprod_{0\neq\frp\in\Spec k[x]}\mod_0(k[x]_\frp).$$
\item The localization functor $\mod k[x]\to\mod k(x)$ induces an
equivalence
$$\frac{\mod k[x]}{\mod_0 k[x]}\lto[\sim]\mod k(x).$$
\end{enumerate}
\end{prop}
\begin{proof}
(1) The assertion follows from standard properties of finitely
generated modules over principal ideal domains. Note that the
quasi-inverse functor takes a family of modules
$(N_\frp)_{\frp\in\Spec k[x]}$ to $\bigoplus_\frp N_\frp$.

(2) Set $\A=\mod k[x]/\mod_0 k[x]$. The kernel of the localization
functor $T=-\otimes_{k[x]}k(x)$ is the category $\mod_0 k[x]$. Thus
$T$ induces a faithful functor $\bar T\colon\A\to\mod k(x)$.  The
functor is dense, since a $k(x)$-module of rank $r$ is isomorphic to
$\bar T(k[x]^r)$.  To show that $\bar T$ is full, it suffices to show
that $\bar T$ induces a surjective map
$$f\colon\Hom_\A(k[x],k[x])\lto\Hom_{k(x)}(k(x),k(x))\cong k(x).$$
Given any non-zero polynomial $P\in k[x]$, let $\mu_P\colon k[x]\to
k[x]$ denote the multiplication by $P$. The kernel and cokernel of
$\mu_P$ belong to $\mod_0 k[x]$, and therefore $\mu_P$ becomes
invertible in $\A$. Thus $f(\mu_P^{-1})=P^{-1}$. It follows that $f$
is surjective.
\end{proof}

\subsection{Coherent sheaves on $\bbP^1_k$}

Let $k$ be a field and $\bbP^1_k$ the projective line over $k$.  We
view $\bbP^1_k$ as a scheme and begin with a description of the
underlying set of points.

Let $k[x_0,x_1]$ be the polynomial ring in two variables with the
usual $\bbZ$-grading by total degree.  Denote by $\Proj k[x_0,x_1]$
the set of homogeneous prime ideals of $k[x_0,x_1]$ that are different
from the unique maximal ideal consisting of positive degree
elements. Note that $k[x_0,x_1]$ is a two-dimensional graded factorial
domain. Thus homogeneous irreducible polynomials correspond to
non-zero homogeneous prime ideals by taking a polynomial $P$ to the
ideal $(P)$ generated by $P$. A \index{point!closed} \emph{closed
point} of $\bbP^1_k$ is by definition an element $\frp\neq 0$ in
$\Proj k[x_0,x_1]$, and the \index{point!generic} \emph{generic point}
is the zero ideal. Using homogeneous coordinates, a
\index{point!rational}\emph{rational point} of $\bbP^1_k$ is a pair
$[\la_0:\la_1]$ of elements of $k$ subject to the relation
$[\la_0:\la_1]=[\alpha \la_0:\alpha \la_1]$ for all $\alpha\in k$,
$\alpha\neq 0$. We identify each rational point $[\la_0:\la_1]$ with
the prime ideal $(\la_1 x_0-\la_0 x_1)$ of $k[x_0,x_1]$.

Using the identification $y=x_1/x_0$, we cover $\bbP^1_k$ by two
copies $U'=\Spec k[y]$ and $U''=\Spec k[y^{-1}]$ of the affine line,
with $U'\cap U''=\Spec k[y,y^{-1}]$. More precisely, the
morphism $k[x_0,x_1]\to k[y]$ which sends a polynomial $P$
to $P(1,y)$ induces a bijection
\begin{equation*}
\Proj k[x_0,x_1]\smallsetminus\{(x_0)\}\lto[\sim]\Spec k[y];
\end{equation*}
see \S\ref{se:dehom} for details. Analogously, the morphism
$k[x_0,x_1]\to k[y^{-1}]$ which sends a polynomial $P$ to
$P(y^{-1},1)$ induces a bijection
\begin{equation*}
\Proj k[x_0,x_1]\smallsetminus\{(x_1)\}\lto[\sim]\Spec k[y^{-1}].
\end{equation*}

Based on the covering $\bbP^1_k=U'\cup U''$, the category
$\coh\bbP^1_k$ of coherent sheaves admits a description in
terms of the following pullback of abelian categories
$$\xymatrix{\coh\bbP^1_k\ar[r]\ar[d]&\coh U'\ar[d]\\\coh
U''\ar[r]&\coh U'\cap U'' }$$ where each functor is given by
restricting a sheaf to the appropriate open subset; see
\cite[Chap.~VI, Prop.~2]{G}. More concretely, this pullback diagram
has, up to  equivalence, the form
$$\xymatrix{\A\ar[r]\ar[d]&\mod k[y]\ar[d]\\ \mod k[y^{-1}]\ar[r]&\mod
k[y,y^{-1}]}$$ where the category $\A$ is defined as follows. The
objects of $\A$ are triples $(M',M'',\mu)$, where $M'$
is a finitely generated $k[y]$-module, $M''$ is a finitely generated
$k[y^{-1}]$-module, and $\mu\colon M_y'\xto{\sim} M_{y^{-1}}''$ is an
isomorphism of $k[y,y^{-1}]$-modules. Here, we use for any $R$-module
$M$ the notation $M_x$ to denote the localization with respect to an
element $x\in R$.  A morphism from $(M',M'',\mu)$ to $(N',N'',\nu)$ in
$\A$ is a pair $(\p',\p'')$ of morphisms, where $\p'\colon M'\to N'$
is $k[y]$-linear and $\p''\colon M''\to N''$ is $k[y^{-1}]$-linear
such that $\nu\p'_y=\p''_{y^{-1}}\mu$.

Given a sheaf $\F$ on $\bbP^1_k$, we denote for any open subset
$U\subseteq\bbP^1_k$ by $\Ga(U,\F)$ the \index{sheaf!sections of}
\emph{sections} over $U$.
\begin{lem}
The functor which sends a coherent sheaf $\F$ on $\bbP^1_k$ to the
triple $(\Ga(U',\F),\Ga(U'',\F),\id_{\Ga(U'\cap U'',\F)})$ gives an
equivalence $\coh\bbP^1_k\xto{\sim}\A$.
\end{lem}
\begin{proof}
The description of a sheaf $\F$ on $\bbP^1_k=U'\cup U''$ in
terms of its restrictions $\F|_{U'}$, $\F|_{U''}$, and $\F|_{U'\cap
U''}$ is classical; see \cite[Chap.~VI, Prop.~2]{G}. Thus it remains
to observe that taking global sections identifies $\coh U'=\mod k[y]$,
$\coh U''=\mod k[y^{-1}]$, and $\coh U'\cap U''=\mod k[y,y^{-1}]$.
\end{proof}

From now on we identify the categories $\coh\bbP^1_k$ and $\A$ via
the above equivalence.

\subsection{Serre's theorem and Serre duality}

Let $\mod^\bbZ k[x_0,x_1]$ denote the category of finitely generated
$\bbZ$-graded $k[x_0,x_1]$-modules and let $\mod_0^\bbZ k[x_0,x_1]$ be
the Serre subcategory consisting of all finite length objects.

There is a functor
\begin{equation}\label{eq:serre}
\mod^\bbZ k[x_0,x_1]\lto\coh\bbP^1_k
\end{equation}
that takes each graded $k[x_0,x_1]$-module $M$ to the triple
$((M_{x_0})_0, (M_{x_1})_0,\s_M)$, where $y$ acts on the degree zero
part of $M_{x_0}$ via the identification $y=x_1/x_0$, the variable
$y^{-1}$ acts on the degree zero part of $M_{x_1}$ via the
identification $y^{-1}=x_0/x_1$, and the isomorphism $\s_M$ equals the
obvious identification
$[(M_{x_0})_0]_{x_1/x_0}=[(M_{x_1})_0]_{x_0/x_1}$.

\begin{prop}[Serre]\label{pr:Serre}
The functor \eqref{eq:serre} induces an equivalence
\begin{equation*}
\frac{\mod^\bbZ k[x_0,x_1]}{\mod_0^\bbZ
k[x_0,x_1]}\lto[\sim]\coh\bbP^1_k.
\end{equation*}
\end{prop}
\begin{proof}
We refer to \cite{S} for the proof. It is clear from the definition
that the functor \eqref{eq:serre} is exact having kernel ${\mod_0^\bbZ
  k[x_0,x_1]}$.  This fact yields the induced functor which is
faithful by construction.
\end{proof}

For any $n\in\bbZ$ and $\F=(M',M'',\mu)$ in $\coh\bbP^1_k$, denote by
$\F(n)$ the \index{sheaf!twisted} \emph{twisted sheaf}
$(M',M'',\mu^{(n)})$, where $\mu^{(n)}$ is the map $\mu$ followed by
multiplication with $y^{-n}$.  Given a module $M$ in $\mod^\bbZ
k[x_0,x_1]$, the \index{module!twisted} \emph{twisted module} $M(n)$
is obtained by shifting the grading, that is, $M(n)_i=M_{i+n}$ for
$i\in\bbZ$. Note that the functor \eqref{eq:serre} is compatible with
the twist functors defined on $\mod^\bbZ k[x_0,x_1]$ and
$\coh\bbP^1_k$.

\begin{prop}[Serre]
The category $\coh\bbP^1_k$ is a Hom-finite $k$-linear abelian
category satisfying Serre duality. More precisely, there is a
functorial $k$-linear isomorphism
$$D\Ext^1(\F,\G)\cong\Hom(\G,\F(-2))\quad\text{for all}\quad \F,\G\in\coh\bbP^1_k.$$
\end{prop}
\begin{proof}
See \cite[III.7]{H}.
\end{proof}

\subsection{Locally free and torsion sheaves}

A sheaf $(M',M'',\mu)$ in $\coh\bbP^1_k$ is called
\index{sheaf!locally free} \emph{locally free} or \index{vector
bundle} \emph{vector bundle} if $M'$ and $M''$ are free modules over
$k[y]$ and $k[y^{-1}]$ respectively. We denote the full subcategory of
vector bundles in $\coh \bbP^1_k$ by $\vect \bbP^1_k$.

The \index{structure sheaf} \emph{structure sheaf} is the sheaf $\Oc=(k[y],k[y^{-1}],
\id_{k[y,y^{-1}]})$. For any pair $m,n\in\bbZ$, we have a natural
bijection
\begin{equation}\label{eq:bij_pol}
k[x_0,x_1]_{n-m}\lto[\sim]\Hom(\Oc(m),\Oc(n)).
\end{equation}
The map sends a homogeneous polynomial $P$ of degree $n-m$ to the
morphism $(\p',\p'')$, where $\p'\colon k[y]\to k[y]$ is
multiplication by $P(1,y)$ and $\p''\colon k[y^{-1}]\to k[y^{-1}]$ is
multiplication by $P(y^{-1},1)$. This bijection is a special case of
the next result.

Let $R=k[x_0,x_1]$ and denote by $\proj^\bbZ R$ the category of
finitely generated projective $\bbZ$-graded $R$-modules. Note that $R$
is a graded local ring since the homogeneous elements of positive
degree form the unique maximal homogeneous ideal. Thus finitely
generated projective $R$-modules are up to isomorphism of the form
$R(n_1)\oplus\dots\oplus R(n_r)$.

\begin{prop}[Grothendieck]\label{pr:Grothendieck}
The functor \eqref{eq:serre} induces an equivalence
$$\proj^\bbZ k[x_0,x_1]\lto[\sim]\vect \bbP^1_k.$$ In particular,
each locally free coherent sheaf on $\bbP^1_k$ is isomorphic to a sheaf of the form
$\Oc(n_1)\oplus \dots\oplus\Oc(n_r)$.
\end{prop}
\begin{proof}
We need to show that the functor \eqref{eq:serre} is fully faithful
when it is restricted to $\proj^\bbZ R$, where $R= k[x_0,x_1]$. Every
finitely generated projective $R$-module is up to isomorphism of the
form $R(n_1)\oplus\dots\oplus R(n_r)$. Thus it suffices to show that
\eqref{eq:serre} induces a bijection
$\Hom_R(R(m),R(n))\to\Hom(\Oc(m),\Oc(n))$ for each pair $m,n\in\bbZ$.
But this is clear since the map coincides with the bijection \eqref{eq:bij_pol}.

It remains to show that the functor is dense, that is, each locally
free coherent sheaf is isomorphic to one of the form $\Oc(n_1)\oplus
\dots\oplus\Oc(n_r)$. For this we refer to \cite{Gr}. An elementary
proof is based on an argument due to Birkhoff \cite{Bi,EET}. A vector bundle
$\F=(M',M'',\mu)$ is basically determined by an invertible matrix over
$k[y,y^{-1}]$, which represents the isomorphism $\mu$. Now one uses
the fact that such a matrix can be transformed into a diagonal matrix
with entries $(y^{-n_1},\dots,y^{-n_r})$ by multiplying it with an
invertible matrix over $k[y]$ from the right and an invertible matrix
over $k[y^{-1}]$ from the left. This yields an isomorphism
$\F\xto{\sim}\Oc(n_1)\oplus \dots\oplus\Oc(n_r)$.
\end{proof}

A sheaf $(M',M'',\mu)$ in $\coh\bbP^1_k$ is called
\index{sheaf!torsion} \emph{torsion} if $M'$ and $M''$ are torsion
modules over $k[y]$ and $k[y^{-1}]$ respectively. We denote the full
subcategory of torsion sheaves in $\coh \bbP^1_k$ by $\coh_0
\bbP^1_k$. Note that each sheaf $\F=(M',M'',\mu)$ admits a unique
maximal subobject $\tor\F$ in $\coh\bbP^1_k$ that is torsion. One
obtains $\tor\F$ by taking the torsion parts of $M'$ and $M''$
respectively. Clearly, $\F/\tor\F$ is locally free.

\begin{prop}
Each coherent sheaf $\F$ on $\bbP^1_k$ admits an essentially unique
decomposition $\F=\F'\oplus\F''$ such that $\F'$ is torsion and $\F''$
is locally free. The torsion sheaves are precisely the objects of
finite length in $\coh\bbP^1_k$.
\end{prop}

\begin{proof}
We use elementary facts about finitely generated modules over $k[y]$
and $k[y^{-1}]$ respectively.  Take $\F'=\tor\F$ and
$\F''=\F/\tor\F$. Serre duality implies
$\Ext^1(\F'',\F')=0$, since there are no non-zero morphisms
from torsion to locally free sheaves. Thus the canonical exact
sequence $0\to\F'\to\F\to\F''\to 0$ splits.

Let $\F=(M',M'',\mu)$ be a coherent sheaf. If $\F$ is torsion, then it
is a finite length object in $\coh\bbP^1_k$, since the corresponding
torsion modules $M'$ and $M''$ have finite length. On the other hand,
the structure sheaf has infinite length, since it admits factor
objects of arbitrary length. Thus each non-zero locally free sheaf is
of infinite length, by Proposition~\ref{pr:Grothendieck}.
\end{proof}

\subsection{Dehomogenization}\label{se:dehom}

Let $R$ be a $\bbZ$-graded commutative ring and $x\in R$ a
non-nilpotent element of degree $1$. The \index{dehomogenization}
\emph{dehomogenization} of $R$ with respect to $x$ is the ring
$R/(x-1)$.

\begin{lem}\label{le:dehom}
The canonical morphism $\pi\colon R\to R/(x-1)$ induces an isomorphism
$(R_x)_0\xto{\sim}R/(x-1)$. Moreover, $\pi$ induces a bijection
between the set of homogeneous (prime) ideals of $R$ modulo which the
element $x$ is regular and the set of all (prime) ideals of $R/(x-1)$.
\end{lem}
\begin{proof}
The morphism $\pi$ induces a morphism $S=R_x\to R/(x-1)$ and its
kernel is the ideal generated by $x-1$. Now observe that $S_0\cong
S/(x-1)$, since $x$ is in $S$ a unit of degree $1$.

The bijective correspondence between ideals of $R$ and $R/(x-1)$ sends
an ideal $\fra\subseteq R$ to $\pi(\fra)$, and its inverse sends
an ideal $\frb\subseteq (R_x)_0=R/(x-1)$ to $\frb R_x\cap R$.
\end{proof}

Let us consider the dehomogenization of the polynomial ring
$k[x_0,x_1]$ with respect to the variable $x_0$. Then
Lemma~\ref{le:dehom} implies an isomorphism $$k[x_0,x_1]/(x_0-1)\cong
(k[x_0,x_1]_{x_0})_0=k[y]$$ via the identification $y=x_1/x_0$.
Denote by $\pi\colon k[x_0,x_1]\to k[y]$ the canonical morphism which
sends a polynomial $P$ to $P(1,y)$.  The morphism $\pi$ induces a bijection
\begin{equation}\label{eq:dehom_bij}
\Proj k[x_0,x_1]\smallsetminus\{(x_0)\}\lto[\sim]\Spec k[y]
\end{equation}
and for any prime ideal $\frp\neq (x_0)$ in $\Proj k[x_0,x_1]$ an
isomorphism
\begin{equation}\label{eq:dehom}
(k[x_0,x_1]_\frp)_0\lto[\sim]k[y]_{\pi(\frp)}.
\end{equation}
Here, $k[x_0,x_1]_\frp$ denotes the homogeneous localization with
respect to $\frp$ which inverts all homogeneous elements not lying in $\frp$.

The following lemma describes the dehomogenization for graded modules
over $k[x_0,x_1]$ with respect to $x_0$. This functor induces an
equivalence if one passes to the localization with respect to a prime
ideal $\frp\neq (x_0)$.

\begin{lem}\label{le:dehom_mod}
Let $\frp\neq (x_0)$ be  in $\Proj k[x_0,x_1]$.  The
functor sending a graded $k[x_0,x_1]$-module $M$ to $FM=(M_{x_0})_0$
induces the following commutative diagram of exact functors
$$\xymatrix{\mod^\bbZ k[x_0,x_1]\ar[rr]^-F\ar[d]&&\mod k[y]\ar[d]\\
\mod^\bbZ k[x_0,x_1]_\frp\ar[r]^-{F'_\frp}&\mod
(k[x_0,x_1]_\frp)_0\ar[r]^-{F''_\frp}&\mod k[y]_{\pi(\frp)} }$$ where
the vertical functors are the localization functors with respect to
$\frp$ and $\pi(\frp)$ respectively. Moreover, the functors $F'_\frp$
and $F''_\frp$ are equivalences.
\end{lem}
\begin{proof}
The composite $k[x_0,x_1]\to k[y]\to k[y]_{\pi(\frp)}$ induces a
morphism $k[x_0,x_1]_\frp\to k[y]_{\pi(\frp)}$ and therefore $F$
composed with localization at $\pi(\frp)$ induces a functor
$$F_\frp\colon\mod^\bbZ k[x_0,x_1]_\frp\lto\mod k[y]_{\pi(\frp)}.$$
This functor can be written as composite $F''_\frp F'_\frp$. The first
functor $F'_\frp$ takes a graded $k[x_0,x_1]_\frp$-module to its
degree zero part; it is an equivalence since $k[x_0,x_1]_\frp$ is
strongly graded. The second functor $F''_\frp$ is an equivalence
thanks to the isomorphism \eqref{eq:dehom}.
\end{proof}

\begin{rem}
There are analogous results for the dehomogenization of $k[x_0,x_1]$
with respect to $x_1$ which is denoted by $k[y^{-1}]$.
\end{rem}

Next we describe the category $\coh_0\bbP^1_k$ of torsion sheaves
more explicitly. Note that $\coh_0\bbP^1_k$ is uniserial because it
is a length category with Serre duality; see
Proposition~\ref{pr:serre_finlen}. The category decomposes into
connected abelian categories, and each component has a unique simple
object since it is equivalent to the category of finite length
modules over a local ring.

\begin{prop}\label{pr:coh_P}
\begin{enumerate}
\item The functor \eqref{eq:serre} induces an equivalence
$$\coprod_{0\neq\frp\in\Proj
k[x_0,x_1]}\mod^\bbZ_0 k[x_0,x_1]_\frp\lto[\sim]\coh_0\bbP^1_k$$
\item The functor taking a sheaf $(M',M'',\mu)$ to $M'\otimes_{k[y]}
k(y)\cong M''\otimes_{k[y^{-1}]} k(y^{-1})$ induces an equivalence
$$\frac{\coh \bbP^1_k}{\coh_0 \bbP^1_k}\lto[\sim]\mod k(y).$$
\end{enumerate}
\end{prop}
\begin{proof}
(1) We denote the functor \eqref{eq:serre} by $F$. Dehomogenization
with respect to $x_0$ and $x_1$, respectively, induces the functors
$F_0$, $F_1$, and $F_{0,1}$, which make the following diagram
commutative.
$$
\xymatrix@R=1pc@C=-0.2pc{
\underset{\frp}{\coprod}\mod^\bbZ_0 k[x_0,x_1]_\frp
\ar[rrr]\ar[ddd]\ar[rd]^-{F}&&& \underset{\frp\neq
(x_0)}{\coprod}\mod^\bbZ_0 k[x_0,x_1]_\frp \ar@{-}[d]\ar[rd]^-{F_0}\\
&\coh_0\bbP^1_k\ar[rrr]\ar[ddd]&&{}\ar[dd]&\mod_0 k[y]\ar[ddd]\\ \\
\underset{\frp\neq (x_1)}{\coprod}\mod^\bbZ_0 k[x_0,x_1]_\frp
\ar@{-}[r]\ar[rd]^-{F_1}&{}\ar[rr]&& \underset{(x_0)\neq\frp\neq
(x_1)}{\coprod}\mod^\bbZ_0 k[x_0,x_1]_\frp \ar[rd]^-{F_{0,1}}\\
&\mod_0 k[y^{-1}]\ar[rrr]&&&\mod_0 k[y,y^{-1}]}$$ Here, $\frp$ runs
through all non-zero prime ideals in $\Proj k[x_0,x_1]$.  Note that
the front square and the back square are pullback diagrams of abelian
categories. The functors $F_0$, $F_1$, and $F_{0,1}$ are
equivalences. This follows from Proposition~\ref{pr:aff} and
Lemma~\ref{le:dehom_mod} in combination with the bijection
\eqref{eq:dehom_bij}. Thus $F$ is an equivalence.

(2) The functor $\coh\bbP^1_k\to\mod k(y)$ is exact and its kernel
is the category $\coh_0\bbP^1_k$ of torsion sheaves. Thus there is
an induced functor $\frac{\coh \bbP^1_k}{\coh_0 \bbP^1_k}\to\mod
k(y)$ which is faithful. The structure sheaf is the unique
indecomposable object, and the argument given in the proof of
Proposition~\ref{pr:aff} shows that the functor is an equivalence.
\end{proof}

\subsection{Support}\label{se:support}

Given a point $\frp\in\Proj k[x_0,x_1]$ of $\bbP^1_k$, the
\index{point!local ring at} \emph{associated local ring}
$\Oc_{\bbP^1_k,\frp}$ is by definition $(k[x_0,x_1]_\frp)_0$. Denote
by $k(\frp)$ the \index{point!residue field of} \emph{residue field} at $\frp$
which is by definition the residue field of the local ring
$\Oc_{\bbP^1_k,\frp}$.

 Note that dehomogenization induces
isomorphisms
$$\Oc_{\bbP^1_k,\frp}\cong\begin{cases}
k[y]_{\frp'}&\text{if }\frp\neq (x_0),\\
k[y^{-1}]_{\frp''}&\text{if }\frp\neq (x_1),
\end{cases}$$
where
$$\frp'=\{P(1,y)\in k[y]\mid P\in
\frp\}\quad\text{and}\quad\frp''=\{P(y^{-1},1)\in k[y^{-1}]\mid P\in
\frp\}.$$ The \index{stalk} \emph{stalk} of a sheaf $\F=(M',M'',\mu)$ at $\frp$ is
the $\Oc_{\bbP^1_k,\frp}$-module
$$\F_\frp=\begin{cases}
M'_{\frp'}&\text{if }\frp\neq (x_0),\\
M''_{\frp''}&\text{if }\frp\neq (x_1),
\end{cases}$$
where $M'_{\frp'}$ and $M''_{\frp''}$ are identified via $\mu$, if
$\frp\neq (x_i)$ for $i=0,1$.  The \index{support} \emph{support} of $\F$ is by
definition
$$\Supp\F=\{\frp\in\Proj k[x_0,x_1]\mid\F_\frp\neq 0\}.$$ Note that
a torsion sheaf $\F$ admits a unique decomposition
$$\F=\bigoplus_{0\neq\frp\in\Proj k[x_0,x_1]}\F_{\{\frp\}}$$ such that
each $\F_{\{\frp\}}$ is a sheaf supported at $\frp$.  This follows directly
from properties of torsion modules over a polynomial ring in one
variable.

The functor sending a sheaf $\F$ to the
family $(\F_\frp)_{\frp\in\Proj k[x_0,x_1]}$ provides an equivalence
\begin{equation}\label{eq:equiv_tor}
\coh_0\bbP^1_k\lto[\sim]\coprod_{0\neq\frp\in\Proj
k[x_0,x_1]}\mod_0\Oc_{\bbP^1_k,\frp}.
\end{equation}
Note that this yields a quasi-inverse for the equivalence
from Proposition~\ref{pr:coh_P} if the functor is composed with the family of equivalences
$\mod_0\Oc_{\bbP^1_k,\frp}\xto{\sim}\mod^\bbZ k[x_0,x_1]_\frp$ from
Lemma~\ref{le:dehom_mod}.

Let $\frp$ be a closed point and choose a homogeneous irreducible
polynomial $P$ of degree $d$ that generates $\frp$. The bijection
\eqref{eq:bij_pol} identifies for each integer $r>0$ the polynomial
$P^r$ with a morphism $\p\colon \Oc(-rd)\to\Oc(0)=\Oc$.  Denote by
$\Oc_{\frp,r}$ the cokernel of this morphism. Thus there is an exact
sequence
\begin{equation}\label{eq:tor}
0\lto \Oc(-rd)\lto[\p]\Oc\lto \Oc_{\frp,r}\lto 0.
\end{equation}

\begin{prop}
Let $\frp$ be a closed point of $\bbP^1_k$ and $r>0$ an integer.  An
indecomposable sheaf in $\coh\bbP^1_k$ has support $\{\frp\}$ and
length $r$ if and only if it is isomorphic to $\Oc_{\frp,r}$.
\end{prop}
\begin{proof}
We compute the support and the length of $\Oc_{\frp,r}$. We may assume
that $\frp\neq (x_0)$.  The morphism $\p$ is given by multiplication
with $P^r(1,y)$ and $P^r(y^{-1},1)$ respectively. Thus for each point
$\frq\neq\frp$, the stalk morphism $\p_\frq$ is an isomorphism, and
therefore $(\Oc_{\frp,r})_\frq=0$. On the other hand, the
$\Oc_{\bbP^1_k,\frp}$-module $(\Oc_{\frp,r})_\frp$ is isomorphic to
$\Oc_{\bbP^1_k,\frp}/\frm^r$, where $\frm$ denotes the maximal ideal
of $\Oc_{\bbP^1_k,\frp}$. Note that this module is
indecomposable. Thus $\Oc_{\frp,r}$ has length $r$ and is indecomposable.

The equivalence \eqref{eq:equiv_tor} shows that an indecomposable
torsion sheaf is uniquely determined by its support and its length.
\end{proof}

\begin{rem}\label{re:hom-simple}
Let $S_\frp=\Oc_{\frp,1}$ be the simple sheaf concentrated at
$\frp$. The sequence \eqref{eq:tor} induces an isomorphism
$\End(S_\frp)\xto{\sim}\Hom(\Oc,S_\frp)$. Moreover, we have an
isomorphism $\End(S_\frp)\cong k(\frp)$ of algebras.
\end{rem}

\subsection{Automorphisms}

Let $\PGL(2,k)$ denote the \index{projective linear
group}\emph{projective linear group}, that is, the group of invertible
$2\times 2$ matrices over $k$ modulo the subgroup of matrices of the
form $\smatrix{a&0\\0&a}$. Any element $\s=\smatrix{\s_{00}&\s_{01}\\
\s_{10}&\s_{11}}$ in $\PGL(2,k)$ induces an automorphism
$k[x_0,x_1]\xto{\sim}k[x_0,x_1]$ by sending  $x_i$ to
$\s_{i0}x_0+\s_{i1}x_1$ ($i=0, 1$). This yields a map
$\PGL(2,k)\to\Aut\bbP^1_k$ into the automorphism group of the
projective line. Recall that a rational point $[\la_0:\la_1]$ is
identified with the prime ideal $(\la_1x_0-\la_0x_1)$.  Then the
automorphism corresponding to $\s$ sends a rational point
$[\la_0:\la_1]$ to
$[\s_{00}\la_0+\s_{01}\la_1:\s_{10}\la_0+\s_{11}\la_1]$.

\begin{prop}\label{pr:aut}
The map $\PGL(2,k)\to\Aut\bbP^1_k$ is an isomorphism.
\end{prop}
\begin{proof}
It is clear that the map is injective.  We provide an inverse map
as follows. Let $\p\colon\bbP^1_k\xto{\sim}\bbP^1_k$ be an
automorphism. This morphism sends rational points to rational
points. In particular for $i=0, 1$ the inverse $\phi^{-1}$ sends the
prime ideal $(x_i)$ to an ideal of the form $(P_i)$  for some
homogenous irreducible polynomial $P_i=\s_{i0}x_0+\s_{i1}x_1$ in
$k[x_0, x_1]$. Let $U_i=\bbP^1_k\smallsetminus\{(x_i)\}$ and denote
by $U'_i$ its image under $\p$. Then $\p$ induces isomorphisms of
affine lines
$$\p_0\colon \Spec k[P_1/P_0]=U'_0 \lto[\sim] U_0=\Spec k[x_1/x_0]$$ and
$$\p_1\colon  \Spec k[P_0/P_1]=U'_1\lto[\sim] U_1=\Spec k[x_0/x_1]$$
which preserve the origins. Thus there are non-zero scalars $a,b$
such that
$$\p_0^*(x_1/x_0)=a (P_1/P_0) \quad\text{and}\quad
\p_1^*(x_0/x_1)=b(P_0/P_1).$$ Here, $\phi_i^*$ denotes the induced
morphism between the rings of regular functions ($i=0,1$). The
morphisms $\p_0$ and $\p_1$ agree on $U_0\cap U_1$, and therefore
$b=a^{-1}$. It follows that $\p$ is
given by the linear transformation $\smatrix{\s_{00}&\s_{01}\\
a\s_{10}&a\s_{11}}$ in $\PGL(2,k)$.
\end{proof}

\subsection{Tilting objects}
The category $\coh\bbP^1_k$ admits a tilting object which is actually
unique up to a shift and up to multiplicities of its indecomposable
direct summands.

\begin{prop}\label{pr:coh_tilt}
An object $T$ in $\coh\bbP^1_k$ is a tilting object if and only if
\begin{equation}\label{eq:addT}
\add T=\add(\Oc(n)\oplus\Oc(n+1))\quad\text{for some}\quad n\in\bbZ.
\end{equation}
For each $n\in\bbZ$, the endomorphism algebra of the tilting object
$\Oc(n)\oplus\Oc(n+1)$ is isomorphic to the Kronecker algebra $\La$
(i.e.\ the path algebra of the quiver $\xymatrix@=15pt{\cdot
\ar@<.5ex>[r]\ar@<-.5ex>[r]&\cdot}$), and this yields a derived
equivalence $\bfD^b(\coh\bbP^1_k)\xto{\sim}\bfD^b(\mod\La)$.
\end{prop}
\begin{proof}
Consider $T=\Oc\oplus\Oc(1)$. We apply the bijection
\eqref{eq:bij_pol} and Serre duality. This yields $\Ext^1(T,T)=0$.
Let $\F$ be an indecomposable sheaf. If $\F$ is torsion, then
$\Hom(\Oc,\F)\neq 0$; see \S\ref{se:support}. If $\F$ is locally
free, say $\F\cong\Oc(n)$, then $\Hom(\Oc,\Oc(n))\neq 0$ if $n\geq
0$, and $\Ext^1(\Oc(1),\Oc(n))\neq 0$ if $n<0$. Thus
$T=\Oc\oplus\Oc(1)$ is a tilting object, and its endomorphism
algebra equals the Kronecker algebra. From this it follows that any
object $T$ in $\coh\bbP^1_k$ satisfying \eqref{eq:addT} is a tilting
object.

Now let $T$ be any tilting object in $\coh\bbP^1_k$. Then $T$ is
locally free since any non-zero torsion sheaf $\F$ has
$\Ext^1(\F,\F)\neq 0$. Another application of the bijection
\eqref{eq:bij_pol} and Serre duality yields the condition
\eqref{eq:addT}.

The derived equivalence is a consequence of  Theorem~\ref{th:tilt}.
\end{proof}

\begin{cor}
\pushQED{\qed} The Grothendieck group of $\coh\bbP^1_k$ is free of
rank two and the corresponding Euler form is non-degenerate. \qedhere
\end{cor}

\section{Coherent sheaves on weighted projective lines}\label{se:cohX}

Following work of Lenzing \cite{L}, we describe the abelian categories
that arise as categories of coherent sheaves on weighted projective
lines. We provide two different approaches: a list of axioms and a
description in terms of expansions of abelian categories.

The axioms basically say that these abelian categories are hereditary
and noetherian, admit a tilting object, and have no non-zero
projective objects. We collect some direct consequences of these
axioms. In particular, we investigate the quotient category modulo the
Serre subcategory of finite length objects; it is a length category
with a unique simple object.

An abelian category satisfying these axioms has a Grothendieck group
that is free of finite rank.  We show that the rank is minimal if and
only if the category is equivalent to the category of coherent sheaves
on the projective line.

The axioms are invariant under forming expansions.  Moreover,
an expansion increases the rank of the Grothendieck group by
one. Thus the formation of expansions reflects the insertion
of weights for specific points of the projective line. These
observations provide the basis for the final description of categories of
coherent sheaves on weighted projective lines.

Throughout this section we fix an arbitrary field $k$.

\subsection{Weighted projective lines}

A \index{weighted projective line} \emph{weighted projective line}
over a field $k$ is by definition a triple $\bbX =
(\bbP^1_k,\bflambda,\bfp)$, where $\bflambda=(\la_1,\dots,\la_n)$ is
a (possibly empty) collection of distinct closed points of the
projective line $\bbP^1_k$, and $\bfp=(p_1,\dots,p_n)$ is a
\index{weight sequence} \emph{weight sequence}, that is, a sequence
of positive integers. In this work we make the additional assumption
that the closed points of $\bflambda$ are supposed to be
\emph{rational}. This assumption simplifies our exposition. In fact,
there is no substantial difference between this case and the case
where the field $k$ is algebraically closed.

We refer to the introduction for the definition of the category
$\coh\bbX$ of coherent sheaves on a weighted projective line $\bbX$.

Let us remark that since the field $k$ is not necessarily
algebraically closed, the $\bfL(\bfp)$-graded algebra
$S(\bfp,\bflambda)$, even up to isomorphism, might depend on the
choice of the homogeneous coordinates $\la_{i0}, \la_{i1}$ for each
$\lambda_i=[\la_{i0}: \la_{i1}]$. However, up to equivalence the
associated category $\coh\bbX$ of coherent sheaves is independent of
this choice; see Corollary \ref{co:wpl_wf}.

\subsection{Hereditary noetherian categories with tilting object}

Let $\A$ be a $k$-linear abelian category. We consider the following
set of axioms:

\begin{enumerate}
\item[(H1)] The category $\A$ is skeletally small, connected, and Ext-finite.
\item[(H2)] The category $\A$ is noetherian, that is, each object of $\A$ is noetherian.
\item[(H3)] The category $\A$ is hereditary and has no non-zero projective object.
\item[(H4)] The category $\A$ has a tilting object.
\item[(H5)] The Euler form associated to $\A$ is non-degenerate and has discriminant $\pm 1$.
\end{enumerate}

Let us collect the basic properties of a category satisfying
(H1)--(H4) so that we can use them from now on freely without any
further reference.

Recall that $\A_0$ denotes the full subcategory consisting of all
finite length objects in $\A$ and that $\A_+$ is the full subcategory
consisting of all objects $A$ in $\A$ satisfying $\Hom_\A(A_0,A)=0$ for
all $A_0$ in $\A_0$.

One should think of objects in $\A_0$ as \index{object!torsion}
\emph{torsion objects}, whereas the objects in $\A_+$ are
\index{object!torsion-free} \emph{torsion-free} or \index{vector
bundle} \emph{vector bundles}.

\begin{prop}
A $k$-linear abelian category $\A$ satisfying \emph{(H1)--(H4)}
has the following properties:
\begin{enumerate}
\item The category $\A$ admits a Serre functor $\tau\colon\A\to\A$.
\item The Grothendieck group $K_0(\A)$ is free of finite rank and the Euler
form associated to $\A$ is non-degenerate.
\item Every object in $\A$ is a direct sum of an object in $\A_0$ and
an object in $\A_+$.
\item The category of finite length objects admits a decomposition
$\A_0=\coprod_{x\in\bfX}\A_x$ into connected uniserial categories.
\end{enumerate}
\end{prop}
\begin{proof}
(1) follows from Proposition~\ref{pr:serre}, (2) from
Proposition~\ref{pr:euler}, (3) from
Proposition~\ref{pr:Serreduality}, and (4) from
Proposition~\ref{pr:serre_finlen}.
\end{proof}

The noetherianness of $\A$ implies that $\A_0$ is non-trivial. On the
other hand, $\A\neq\A_0$ because a length category with Serre duality
and a Grothendieck group of finite rank has a degenerate Euler form;
see Example~\ref{ex:degen}.

Denote by $\Si_0$ a set of representatives of the isomorphism classes
of simple objects of $\A$. We identify the set $\Si_0/\tau$ of
$\tau$-orbits with the index set $\bfX$ of the decomposition
$\A_0=\coprod_{x\in\bfX}\A_x$.  For each $x\in\bfX$, let $p(x)$ denote
the number of isomorphism classes of simple objects of $\A_x$.

\begin{lem}
Each $p(x)$ is finite. More precisely, $\sum_{x\in\bfX}(p(x)-1)$ is
bounded by the rank of $K_0(\A)$.
\end{lem}
\begin{proof}
First observe that for two simple objects $S,T\in\A$, we have
$\Ext^1_\A(S,T)\neq 0$ if and only if $T\cong\tau S$; see
Theorem~\ref{th:Gabriel} and Proposition~\ref{pr:Serreduality}. Now
choose for each $x\in\bfX$ a simple object $S_x\in\A_x\cap\Si_0$ and
let $\Si'_0=\Si_0\smallsetminus\{S_x\mid x\in\bfX\}$.  The set
$\Si'_0$ admits a linear ordering such that $S>T$ implies $\langle
[S],[T]\rangle =0$. It follows that the corresponding classes $[S]$,
$S\in\Si'_0$, are linearly independent in $K_0(\A)$; see
Lemma~\ref{le:lin} below. Thus $\card \Si'_0=\sum_{x\in\bfX}(p(x)-1)$
is bounded by the rank of $K_0(\A)$.
\end{proof}

\begin{lem}\label{le:lin}
Let $G$ be an abelian group and $X\subseteq G$ a subset. Suppose there
is a non-degenerate bilinear form $\p$ on $G$ and a linear ordering on
$X$ such that $\p(x,x)\neq 0$ for all $x\in X$ and $\p(x,y)=0$ for all
$x>y$ in $X$. Then $X$ is linearly independent.
\end{lem}
\begin{proof}
Straightforward.
\end{proof}

\begin{lem}\label{le:linear-rank}
There exists a linear map $\rho\colon K_0(\A)\to\bbZ$ such that
\begin{enumerate}
\item $\rho([A])\geq 0$ for all $A\in\A$;
\item $\rho([A])=0$ if and only if $A\in\A_0$;
\item $\rho([\tau A])=\rho ([A])$ for all  $A\in\A$.
\end{enumerate}
\end{lem}
\begin{proof}
For each $x\in\bfX$ choose a simple object $S_x\in\A_x$ and let
$w_x=\sum_{i=1}^{p(x)}[\tau^i S_x]$. Choose elements $x_1,\dots,x_r$ from
$\bfX$ such that the subgroup of $K_0(\A)$ generated by
$w_{x_1},\dots,w_{x_r}$ contains each $w_x$ and set $w=w_{x_1}+\dots
+w_{x_r}$.

Serre duality implies that $\langle [A],w\rangle=0$ for all $A$ in
$\A_0$. If $A$ is a non-zero object in $\A_+$, then $A$ has a simple
quotient, say $S_x$, by noetherianness. Using that
$\Ext_\A^1(A,A_0)=0$ for all $A_0$ in $\A_0$ by Serre duality, this
yields $\langle [A],w_x\rangle>0$ and therefore $\langle
[A],w\rangle>0$. Thus the map $\rho=\langle-,w\rangle$ has the desired
properties.
\end{proof}

\begin{prop}
The abelian category $\A/\A_0$ is a length category.
\end{prop}
\begin{proof}
Let $A$ be an object in $\A$. We prove by induction on $\rho([A])$
that $A$ has finite length in $\A/\A_0$. If $\rho([A])=0$, then $A=0$
in $\A/\A_0$. Now suppose $\rho([A])>0$.  The category $\A/\A_0$ is
noetherian and therefore each non-zero object has a simple
quotient. Thus there exists a subobject $A'\subseteq A$ such that
$A/A'$ is simple in $\A/\A_0$. Then $\rho([A'])< \rho([A])$ and
therefore $A'$ has finite length in $\A/\A_0$. It follows that $A$ has
finite length in $\A/\A_0$.
\end{proof}

For each object $A$ in $\A$, we denote by $\rank A$ the length of
$A$ in $\A/\A_0$ and call it the \index{rank} \emph{rank}. This
function extends to a linear map $K_0(\A)\to \bbZ$. This linear map
is surjective and satisfies the conditions in Lemma
\ref{le:linear-rank}. Indeed, such a map is unique by Proposition
\ref{pr:K-split}.

\subsection{Line bundles}
Let $\A$ be a $k$-linear abelian category satisfying (H1)--(H4).  An
indecomposable object in $\A$ of rank one is called \index{line
bundle} \emph{line bundle}. Thus line bundles are precisely the
objects in $\A_+$ of rank one; they form the building blocks of the
category $\A_+$. Let us collect their basic properties.

\begin{prop}\label{pr:lb_filtr}
Every object $A$ in $\A_+$ admits
a filtration $0=A_0\subseteq A_1\subseteq \dots\subseteq A_n=A$ of
length $n=\rank A$ such that each factor $A_i/A_{i-1}$ is a line
bundle.
\end{prop}
\begin{proof}
We proceed by induction on $n$. The case $n\leq 1$ is clear. If $n>1$,
choose a monomorphism $U\to A$ in $\A/\A_0$ with simple cokernel. This
morphism is represented by a morphism $\p\colon U'\to A/A'$ in $\A$
such that $U'\subseteq U$ and $A'\subseteq A$ are subobjects with
$U/U'$ and $A'$ in $\A_0$; see Lemma~\ref{le:mor}. It follows that
$A'=0$ since $A\in\A_+$. Passing from $U'$ to the image of $\p$, we
may assume that $\Ker\p=0$. The cokernel $C=\Coker\p$ is simple in
$\A/\A_0$. Thus there is a decomposition $C=C_0\oplus C_1$ with $C_0$
of finite length and $C_1$ a line bundle. Forming the pullback of the
exact sequence $0\to U'\to A\to C\to 0$ along the inclusion $C_0\to C$
yields a monomorphism $A_{n-1}\to A$ with cokernel $C_1$. Clearly,
$A_{n-1}$ belongs to $\A_+$ and has rank $n-1$.
\end{proof}

\begin{lem}\label{le:filtr}
Let $A$ be a non-zero object in $\A_+$ and suppose that
$\Ext^1_\A(S,A)\neq 0$ for some simple object $S$.
\begin{enumerate}
\item There exists a chain of monomorphisms $A=A_0\xto{\p_1}
A_1\xto{\p_2}\cdots$ in $\A$ such that $A_i\in\A_+$ and
$A_{i}/\Im\p_{i}\cong\tau^{-i+1}S$ for all $i>0$.
\item For each $n\geq 0$, there exists an exact sequence $0\to A\to B\to
C\to 0$ in $\A$ such that $B\in\A_+$ and $[C]=\sum_{i=0}^n[\tau^{-i} S]$.
\end{enumerate}
\end{lem}
\begin{proof}
(1) Choose a non-split exact sequence $0\to A\xto{\p_1} A_1\to S\to
0$. For each simple object $T$, the induced morphism
$\Hom_\A(T,A)\to\Hom_\A(T,A_1)$ is an isomorphism. Thus $A_1$ belongs
to $\A_+$. Serre duality implies that $\Ext_\A^1(\tau^{-1}S,A_1)\neq
0$. Thus we can iterate the construction and obtain a sequence of
morphisms $\p_i\colon A_{i-1}\to A_i$

(2) Apply (1) by taking the composite $\p_{n+1}\dots\p_1$ for the
    morphism $A\to B$.
\end{proof}

\begin{lem}\label{le:lineb}
Let $L,L'$ be line bundles and $0\to L\to L'\to S\to 0$ an exact
sequence such that $S$ is simple. Then for each $x\in\bfX$,
$\Hom_\A(L,\A_x)= 0$ if and only if $\Hom_\A(L',\A_x)= 0$.
\end{lem}
\begin{proof}
Note that $\Hom_\A(L,\A_x)\neq 0$ if and only if $\Hom_\A(L,S_x)\neq
0$ for some simple object $S_x\in\A_x$. The assumptions imply
$\Hom_\A(L',S)\neq 0$ and $\Hom_\A(L,\tau S)\neq 0$. If $T$ is a
simple object not lying in the $\tau$-orbit of $S$, then $L\to L'$
induces an isomorphism $\Hom_\A(L',T)\xto{\sim}\Hom_\A(L,T)$.
\end{proof}

\begin{lem}\label{le:lineb1}
Let $L,L'$ be line bundles. Then the following are
equivalent:
\begin{enumerate}
\item For each $x\in\bfX$, $\Hom_\A(L,\A_x)= 0$ if and only if
$\Hom_\A(L',\A_x)= 0$.
\item There exists $x\in\bfX$ such that $\Hom_\A(L,\A_x)\neq 0$ and
$\Hom_\A(L',\A_x)\neq 0$.
\item There exists $n\in\bbZ$ such that $L'\cong\tau^n L$ in $\A/\A_0$.
\end{enumerate}
\end{lem}
\begin{proof}
(1) $\Rightarrow$ (2): Observe that $\Hom_\A(L,\A_x)\neq 0$ if and
only if $\Hom_\A(L,S)\neq 0$ for some simple object $S\in\A_x$. Thus
noetherianness of $L$ implies that $\Hom_\A(L,\A_x)\neq 0$ for some
$x\in\bfX$.

(2) $\Rightarrow$ (3): Choose a simple object $S$ and a non-zero
morphism $\p\colon L\to S$. Modulo some power of $\tau$, there is a
non-zero morphism $\p'\colon L'\to S$. Forming the pullback of $\p$
and $\p'$, one obtains the following commutative diagram with exact
rows.
$$\xymatrix{
0\ar[r]&K\ar@{=}[d]\ar[r]&P\ar[d]\ar[r]&L'\ar[d]^{\p'}\ar[r]&0\\
0\ar[r]&K\ar[r]&L\ar[r]^\p&S\ar[r]&0 }$$ If the top row splits, then
$\Hom_\A(L',L)\neq 0$ and therefore $L'\cong L$ in
$\A/\A_0$. Otherwise, $\Hom_\A(K,\tau L')\cong D\Ext_\A^1(L',K)\neq 0$.
Thus $L\cong\tau L'$ in $\A/\A_0$.

(3) $\Rightarrow$ (1): Suppose that $L'\cong\tau^n L$ in $\A/\A_0$.
Then there is a subobject $L''\subseteq L'$ with $L'/L''$ of finite
length and there is a monomorphism $L''\to \tau^n L$ with cokernel of
finite length.  For each $x\in\bfX$, it follows then by iterating
Lemma~\ref{le:lineb} that $\Hom_\A(L,\A_x)\neq 0$ if and only if
$\Hom_\A(L',\A_x)\neq 0$.
\end{proof}

\begin{prop}[Lenzing]\label{pr:Lmorph} 
Let  $L$ be a line bundle and $x\in\bfX$. Then we have
$\Hom_\A(L,\A_x)\neq 0$.
\end{prop}
\begin{proof}
Let $\bfX_1$ be the set of all $y\in\bfX$ such that
$\Hom_\A(L,\A_y)\neq 0$, and let $\bfX_2=\bfX\smallsetminus\bfX_1$.
We show that $\bfX_1=\bfX$.  Let $\L_1$ be the class of line bundles
$L'$ such that $L'\cong\tau^n L$ in $\A/\A_0$ for some $n\in\bbZ$,
and let $\L_2$ be the class of all remaining line bundles.  Denote
by $\A_i$ ($i=1,2$) the full subcategory consisting of objects
$A\in\A$ having a filtration with factors in $\L_i$ or
$\bigcup_{x\in\bfX_i}\A_x$.  There are no non-zero morphisms between
objects from different $\A_i$'s, and therefore also no extensions by
Serre duality. This follows from Lemma~\ref{le:lineb1} and the fact
that each non-zero morphism between line bundles in $\A$ induces an
isomorphism in $\A/\A_0$.  On the other hand, each indecomposable
object belongs to one of the $\A_i$'s. This is clear for objects of
finite length. An object $A$ from $\A_+$ has a finite filtration
$0=A_0\subseteq \dots \subseteq A_r=A$ such that each factor
$A_i/A_{i-1}$ is a line bundle; see Proposition~\ref{pr:lb_filtr}.
The factors belong to a single $\L_i$ since there are no non-split
extensions between different $\L_i$'s. Thus $\A=\A_1\amalg\A_2$, but
this implies $\A_2=0$ since $\A$ is connected.
\end{proof}

\begin{lem}\label{le:lb}
Let $L$ be a line bundle and $0\neq A\in\A_+$. Then there are
monomorphisms $L\to A'$ and $A\to A'$ such that $A'$ belongs to
$\A_+$ and the cokernel  of $A\to A'$ belongs to  $\A_0$. Moreover,
each non-zero morphism $L\to A$ is a monomorphism.
\end{lem}
\begin{proof}
The object $A$ has a simple quotient since it is noetherian. Thus
$\Ext_\A^1(S,A)\neq 0$ for some simple object $S$ by Serre duality.
Depending on the value of $\langle [L],[A]\rangle$ and using that
$L$ admits a non-zero morphism to the $\tau$-orbit of $S$ by
Proposition~\ref{pr:Lmorph}, we choose in Lemma~\ref{le:filtr} the
number $n$ sufficiently big so that there exists an exact sequence
$0\to A\to B\to C\to 0$ in $\A$ with $B\in\A_+$,
$[C]=\sum_{i=0}^n[\tau^{-i} S]$, and
$$\langle [L],[B]\rangle= \langle [L],[A]\rangle+\langle
[L],[C]\rangle>0.$$ Now set $A'=B$.

If $\p\colon L\to A$ is a non-zero morphism, then
$\rank\Ker\p=0$. Thus $\Ker\p=0$ since $L$ belongs to $\A_+$.
\end{proof}

The discussion of line bundles yields further properties of $\A/\A_0$
and $K_0(\A)$.

\begin{prop}\label{pr:K-split}
Let $\A$ be a $k$-linear abelian category satisfying \emph{(H1)--(H4)}.
Then the following holds:
\begin{enumerate}
\item The abelian category $\A/\A_0$ has, up to isomorphism, a unique simple object.
\item Each non-zero object $A$ in $\A$ satisfies $[A]\neq 0$ in $K_0(\A)$.
\item Let $L$ be a line bundle in $\A$. Then $K_0(\A)=\bbZ [L]\oplus K'_0(\A)$,
where $K'_0(\A)$ denotes the image of the canonical map $K_0(\A_0)\to K_0(\A)$.
\end{enumerate}
\end{prop}
\begin{proof}
(1) Let $L,L'$ be line bundles in $\A$. There are monomorphisms
$L\to L''$ and $L'\to L''$ with both cokernels in $\A_0$, by
Lemma~\ref{le:lb}. Thus $L\cong L''\cong L'$ in $\A/\A_0$.

(2) Let $A$ be a non-zero object. If $A$ is not of finite length, then
    $\rank{[A]}\neq 0$ and therefore $[A]\neq 0$. Now suppose that $A$
    is of finite length. It follows from Proposition~\ref{pr:Lmorph} that
    $\Hom_\A(L,A)\neq 0$ for some line bundle $L$. On the other hand,
    $\Ext^1_\A(L,A)=0$ by Serre duality. Thus $\langle
    [L],[A]\rangle\neq 0$, and it follows that $[A]\neq 0$.

(3) We have $\bbZ [L]\cap K'_0(\A)=0$ since $\rank L>0$ and $\rank x=0$
for all $x\in K'_0(\A)$. We show by induction on the rank that each
class $[A]$ belongs to $\bbZ [L]+ K'_0(\A)$. This is clear if $\rank
A=0$. If $\rank A>0$, then there is an exact sequence $0\to L\to A'\to
A''\to 0$ such that $A'\cong A$ in $\A/\A_0$; see
Lemma~\ref{le:lb}. It follows that $\rank A''=\rank A' -1$, and
therefore $[A']=[L]+[A'']$ belongs to $\bbZ [L]+ K'_0(\A)$. Finally
observe that $[A]-[A']$ belongs to $K'_0(\A)$.
\end{proof}

An immediate consequence is the fact that the rank of $K_0(\A)$ is at
least two.

\subsection{Exceptional objects}

Let $\A$ be a $k$-linear abelian category that is Hom-finite and
hereditary.  An object $A$ is called \index{object!exceptional}
\emph{exceptional} if $\Ext_\A^1(A,A)=0$ and $\End_\A(A)$ is a
division ring.

\begin{lem}[Happel-Ringel]\label{le:HR}
Let $A,B$ be indecomposable objects in $\A$. If $\Ext^1_\A(B,A)=0$, then
each non-zero morphism $A\to B$ is a monomorphism or an epimorphism.
\end{lem}
\begin{proof}
Let $\p\colon A\to B$ be a non-zero morphism and
$A\xto{\p'}\Im\p\xto{\p''}B$ its canonical factorization. We obtain
the following commutative diagram with exact rows
$$\xymatrix{0\ar[r]&A\ar[r]\ar[d]_-{\p'}&E\ar[r]\ar[d]&B/\Im\p\ar[r]\ar@{=}[d]&0\\
0\ar[r]&\Im\p\ar[r]^-{\p''}&B\ar[r]&B/\Im\p\ar[r]&0}$$ since
$\Ext^1_\A(B/\Im\p,-)$ is right exact. The induced exact sequence
$0\to A\to \Im\p\oplus E\to B\to 0$ splits. Thus $\p'$ or $\p''$ is
a split monomorphism since $\End_\A(A)$ is local. In the first case
$\p$ is a monomorphism, and in the second case $\p$ is an
epimorphism.
\end{proof}

Let us collect some immediate consequences.

\begin{prop}\label{pr:HR1}
Let $\A$ be a $k$-linear abelian category that is Hom-finite and
hereditary. Then the following holds:
\begin{enumerate}
\item An indecomposable object $A$ satisfying $\Ext_\A^1(A,A)=0$ is exceptional.
\item Let $A,B$ be non-isomorphic exceptional objects.  Suppose that
$\Ext^1_\A(A,B)=0$ and $\Ext^1_\A(B,A)=0$. Then $\Hom_\A(A,B)=0$ or
$\Hom_\A(B,A)=0$.
\item Assume further that $\A$ is Ext-finite. Let  $A,B$ be exceptional objects and $[A]=[B]$ in $K_0(\A)$. Then $A\cong B$.
\end{enumerate}
\end{prop}
\begin{proof}
We apply Lemma~\ref{le:HR} and use the fact that for each
indecomposable object, an endomorphism is either nilpotent or
invertible. In particular, an endomorphism that is a monomorphism or
an epimorphism is invertible. From this, (1) and (2) are clear.

(3) Observe that $\Hom_\A(A,B)\neq 0$ since $\langle
[A],[B]\rangle>0$. Let $\p\colon A\to B$ be a non-zero morphism and
$B'=\Im\p$. Applying the right exact functor $\Ext^1_\A(-,B)$ to the
inclusion $B'\to B$ shows that $\Ext^1_\A(B',B)=0$. Thus $\langle
[B'],[A]\rangle=\langle [B'],[B]\rangle>0$. Composing a non-zero
morphism $B'\to A$ with the epimorphism $A\to B'$ induced by $\p$
yields an isomorphism since $\End_\A(A)$ is a division ring. Thus $\p$
is a monomorphism.  The dual argument shows that $\p$ is an
epimorphism. Thus $A\cong B$.
\end{proof}

\subsection{Expansions of abelian categories}

We consider expansions of abelian categories that satisfy
the axioms (H1)--(H4).

\begin{lem}\label{le:reduction}
Let $\A$ be a $k$-linear abelian category satisfying
\emph{(H1)--(H4)}. For a full subcategory $\B$ of $\A$, the following
are equivalent:
\begin{enumerate}
\item The inclusion $\B\to\A$ is a non-split expansion of
abelian categories.
\item There exists a simple object $S$ such that $\tau S\not\cong S$
and $S^\perp=\B$.
\end{enumerate}
\end{lem}
\begin{proof}
Apply Lemma~\ref{le:localizable}. If $\B\to\A$ is an expansion, then
$\B=S^\perp$ for some localizable object $S$, and Serre duality
implies $\tau S\not\cong S$. Conversely, if $S$ is simple and $\tau
S\not\cong S$, then $S$ is a (co)localizable object with
$S^\perp={^\perp\tau S}$. Thus the inclusion $S^\perp\to\A$ is an
expansion, and this is non-split since $\A$ is connected.
\end{proof}

Let $i\colon\B\to\A$ be a non-split expansion of
$k$-linear abelian categories satisfying (H1)--(H4). Then $i$
restricts to an expansion $\B_0\to \A_0$ by
Proposition~\ref{pr:onepoint_length}. Let
$\A_0=\coprod_{x\in\bfX_\A}\A_x$ and $\B_0=\coprod_{x\in\bfX_\B}\B_x$
be the decompositions into connected uniserial categories; see
Proposition~\ref{pr:serre_finlen}.

\begin{prop}\label{pr:onepoint_comp}
Let $i\colon\B\to\A$ be a non-split expansion of
$k$-linear abelian categories satisfying \emph{(H1)--(H4)}. There
exists a bijection $\p\colon\bfX_\B\to\bfX_\A$ and $x_0\in\bfX_\B$
such that $i$ restricts to an expansion
$\B_{x_0}\to\A_{\p(x_0)}$ and to equivalences
$\B_x\xto{\sim}\A_{\p(x)}$ for all $x\neq x_0$ in $\bfX_\B$.
\end{prop}
\begin{proof}
Apply Lemma~\ref{le:decomp}.
\end{proof}

Next we investigate the existence of tilting objects for expansions of
abelian categories.

\begin{lem}\label{le:onepoint_tilt}
Let $i\colon\B\to\A$ be a non-split expansion of $k$-linear
abelian categories that are Ext-finite.
\begin{enumerate}
\item If $\B$ admits a tilting object, then $\A$ admits a tilting
object.
\item Suppose in addition that $\A$ is hereditary and has no non-zero
projective object. If $\A$ admits a tilting object, then $\B$ admits a
tilting object.
\end{enumerate}
\end{lem}
\begin{proof}
(1) Let $T$ be a tilting object in $\B$. Choose an exact
sequence $0\to S\to T'\to iT\to 0$ in $\A$ with $S$ in $\add S_\la$ such that
the induced map $\Hom_\A(S,S_\la)\to\Ext^1_\A(iT,S_\la)$ is an
epimorphism. We claim that $U=T'\oplus S_\la$ is a tilting object for $\A$.

The formula $\Ext_\A^n(iT,A)\cong\Ext_\B^n(T,i_\rho A)$ for all
$A\in\A$ and $n\ge 0$ implies $\pdim iT\le 1$. Therefore $\pdim U\le
1$. The construction of $T'$ implies $\Ext_\A^1(T',S_\la)=0$, and
$\Ext_\A^1(S_\la,T')=0$ is clear since $\Ext_\A^1(S_\la,S)=0$ and
$\Ext_\A^1(S_\la,iT)=0$. We have
$$\Ext^1_\A(T',T')\cong\Ext^1_\A(T',iT)\cong\Ext^1_\B(i_\la
T',T)\cong\Ext^1_\B(T,T)=0,$$ and therefore
$\Ext_\A^1(U,U)=0$. Finally, assume that $\Ext_\A^n(U,A)=0$ for some
$A$ in $\A$ ($n=0,1$). The condition $\Ext_\A^n(S_\la,A)=0$ implies
that $A$ belongs to the image of $i$, say $A=iB$, and that
$\Ext_\A^n(iT,A)=0$.  Then $0=\Ext_\A^n(iT,iB)\cong\Ext_\B^n(T,B)$
implies $B=0$. Thus $U$ is a tilting object.

(2) Let $T$ be a tilting object in $\A$. We intend to show that $i_\la
T$ is a tilting object for $\B$. The formula $\Ext^n_\B(i_\la
T,B)\cong\Ext^n_\A(T,iB)$ for all $B\in\B$ and $n\ge 0$ shows that
$\pdim i_\la T\le 1$ and that $\Ext^n_\B(i_\la T,B)=0$ ($n=0,1$) implies
$B=0$. It remains to show that $\Ext^1_\B(i_\la T,i_\la T)=0$. In
fact, it is equivalent to show that $\Ext^1_\A(T,ii_\la T)=0$.

We proceed by cases. First assume that $\Ext_\A^1(S_\la,T)=0$. Thus
the adjunction morphism $\eta_T\colon T\to ii_\la T$ is an
epimorphism, and therefore $\Ext^1_\A(T,ii_\la T)=0$ since
$\Ext_\A^1(T,-)$ is right exact.

Next assume that $\Ext_\A^1(T,S_\la)=0$. Then $\Ext_\A^1(T,ii_\la T)=0$
follows since $\Ker\eta_T$ and $\Coker\eta_T$ belong to $\add S_\la$.

Finally, assume that $\Ext_\A^1(S_\la,T)\neq 0$. We apply the
Auslander-Reiten formula from Proposition~\ref{pr:ARformula} and have
$\Hom_\A(T,S_\rho)\neq 0$. Thus there is an epimorphism $T\to S_\rho$,
and this implies $\Ext_\A^1(T,S_\rho)=0$. The category $\A$ admits a
Serre functor $\tau\colon\A\xto{\sim}\A$ by
Proposition~\ref{pr:serre}, and the Auslander-Reiten formula implies
$\tau S_\la=S_\rho$. Thus $\Ext_\A^1(\tau^{-1}T,S_\la)=0$.  The object
$U=\tau^{-1}T$ is a tilting object for $\A$, and the above argument
shows that $i_\la U$ is a tilting object for $\B$.
\end{proof}

The following result says that the concept of an expansion of
abelian categories is compatible with the list of axioms (H1)--(H5).

\begin{thm}\label{th:reduction}
Let $k$ be a field and $\B\to\A$ a non-split expansion of
$k$-linear abelian categories with associated division ring $k$. Then
$\A$ satisfies \emph{(H1)--(H5)} if and only if $\B$ satisfies
\emph{(H1)--(H5)}. In that case the rank of $K_0(\B)$ is one less
than that of $K_0(\A)$.
\end{thm}

\begin{proof}
We provide the references for each axiom.  The final assertion about
the rank of $K_0(\B)$ follows from Lemma~\ref{le:euler_onepoint}.

(H1) Lemmas~\ref{le:small}, \ref{le:decomp}, and \ref{le:extfin}.

(H2) Lemma~\ref{le:noetherian}.

(H3) Lemma~\ref{le:cohdim}.

(H4) Lemma~\ref{le:onepoint_tilt}.

(H5) Proposition~\ref{pr:euler} and Lemma~\ref{le:euler_onepoint}.
\end{proof}

\subsection{An equivalence via tilting}
We give a criterion so that an equivalence of derived categories
$\bfD^b(\A)\xto{\sim}\bfD^b(\A')$ restricts to an equivalence
$\A\xto{\sim}\A'$.

Let $\A$ be a $k$-linear abelian category satisfying (H1)--(H4) and
consider its bounded derived category $\bfD^b(\A)$. Recall that
$$\bfD^b(\A)=\bigsqcup_{n\in\bbZ}\A[n]$$ with non-zero morphisms
$\A[i]\to\A[j]$ only if $j-i\in\{0,1\}$ since $\A$ is hereditary; see
Corollary~\ref{co:der_hereditary}.

The isomorphism $K_0(\A)\xto{\sim}K_0(\bfD^b(\A))$ yields a rank
function $K_0(\bfD^b(\A))\to\bbZ$.  Note that for each complex $X$
concentrated in degree $n$, we have $(-1)^n\cdot\rank{[X]}\geq 0$.

Let $T$ be an indecomposable object in $\A_+$ and view it as a complex
concentrated in degree zero. Define $\L(T)$ to be the class of
indecomposable objects $L\in\bfD^b(\A)$ of rank one such that
$\Hom_{\bfD^b(\A)}(L',L)\neq 0$ for some indecomposable object
$L'\in\bfD^b(\A)$ of rank one satisfying $\Hom_{\bfD^b(\A)}(L',T)\neq
0$.

\begin{lem}\label{le:L(T)}
Let $T$ be an indecomposable object in $\A_+$. Then the objects in
$\L(T)$ are precisely those that are isomorphic to a line bundle in
$\A$, viewed as a complex concentrated in degree zero.
\end{lem}
\begin{proof}
Let $L'$ be a complex in $\bfD^b(\A)$ that is concentrated in one
degree, say $n$. If $\Hom_{\bfD^b(\A)}(L',T)\neq 0$, then $n=0$ or
$n=1$. If the rank of $L'$ is positive, then $n=0$. The same argument
shows that every indecomposable object $L$ in $\bfD^b(\A)$ of rank one
and satisfying $\Hom_{\bfD^b(\A)}(L',L)\neq 0$ is isomorphic to a line
bundle, viewed as a complex concentrated in degree zero.

Conversely, let $L$ be a line bundle. Then there exists a non-zero
morphism $L\to T'$ for some object $T'$ that admits an exact sequence
$0\to T\to T'\to C\to 0$ such that $C$ has finite length; see
Lemma~\ref{le:lb}. The pullback of $T\to T'$ and $L\to T'$ yields a
line bundle $L'$ with non-zero morphisms to $T$ and $L$. Thus $L$
belongs to $\L(T)$.
\end{proof}

\begin{lem}\label{le:tilt_zero}
Let $T$ be an indecomposable object in $\A_+$. Then the following are
equivalent for an indecomposable object $X$ in $\bfD^b(\A)$:
\begin{enumerate}
\item  $H^iX=0$ for all $i\neq 0$.
\item $\rank{[X]} \geq 0$ and $\Hom_{\bfD^b(\A)}(L,X)\neq 0$ for some $L\in\L(T)$.
\end{enumerate}
\end{lem}
\begin{proof}
We apply Lemma~\ref{le:L(T)}. Because $X$ is indecomposable, there
exists an integer $n$ such that $H^iX=0$ for all $i\neq n$.

(1) $\Rightarrow$ (2): If $n=0$, then $\rank{[X]} \geq 0$ and
$\Hom_\A(L,H^0X)\neq 0$ for some line bundle $L$ in $\A$.  Thus
$\Hom_{\bfD^b(\A)}(L,X)\neq 0$ for some $L\in\L(T)$.

(2) $\Rightarrow$ (1): If $\Hom_{\bfD^b(\A)}(L,X)\neq 0$ for some
    $L\in\L(T)$, then $n=0$ or $n=1$.  It follows that $n=0$ if $X$
    has positive rank. If the rank of $X$ is zero, then $H^nX$ has
    finite length and therefore $\Ext_\A^1(H^0L,H^nX)=0$. Thus $n=0$.
\end{proof}

\begin{prop}\label{pr:der_equiv}
Let $\A,\A'$ be abelian categories satisfying \emph{(H1)--(H4)} with
tilting objects $T\in\A$ and $T'\in\A'$.  Suppose that
$\End_\A(T)\cong\End_{\A'}(T')$ and that the induced equivalence
$\add T\xto{\sim}\add T'$ preserves the rank.  Then $\A$ and $\A'$ are
equivalent categories.
\end{prop}
\begin{proof}
We identify $\La=\End_\A(T)=\End_{\A'}(T')$ and obtain equivalences
$$\bfD^b(\A)\xto{\RHom_\A(T,-)}\bfD^b(\mod\La)\xleftarrow{\RHom(T',-)}\bfD^b(\A').$$
This yields an equivalence $F\colon \bfD^b(\A)\xto{\sim}\bfD^b(\A')$
taking $T$ to $T'$.  The functor $F$ preserves the rank since the
indecomposable direct summands of $T$ form a basis of $K_0(\bfD^b(\A))$. It
follows from Lemma~\ref{le:tilt_zero} that $F$ identifies $\A$ with
$\A'$.
\end{proof}

\subsection{The homogeneous case}

Let $\A$ be a $k$-linear abelian category satisfying (H1)--(H4).  Call
$\A$ \index{category!homogeneous} \emph{homogeneous} if $\tau A\cong
A$ for each object $A$ of finite length, or equivalently, $\tau S\cong
S$ for each simple object $S$.  This property together with the
conditions (H1)--(H5) characterizes the category $\coh\bbP^1_k$.  The
following characterization of the property of $\A$ to be homogeneous will be useful.

\begin{prop}\label{pr:homog}
Let $\A$ be a $k$-linear abelian category satisfying
\emph{(H1)--(H4)}. Then the following are equivalent:
\begin{enumerate}
\item If $S$ is a simple object in $\A$, then $\tau S\cong S$.
\item If $A,B$ are  finite length objects in $\A$, then $\langle [A],[B]\rangle=0$.
\item The rank of $K_0(\A)$ equals two.
\item If $A$ has infinite length and $S$ is a simple object in $\A$,
then $\Hom_\A(A,S)\neq 0$.
\end{enumerate}
\end{prop}
\begin{proof}
(1) $\Rightarrow$ (2): It suffices to show that $\langle
[S],[T]\rangle=0$ for each pair of simple objects $S,T$. The equality
$\langle [S],[T]\rangle=0$ is an immediate consequence of Serre
duality if $\tau T\cong T$.

(2) $\Rightarrow$ (3): Let $K'_0(\A)$ be the image of the canonical
    map $K_0(\A_0)\to K_0(\A)$. We have $K_0(\A)\cong \bbZ\oplus
    K'_0(\A)$ by Proposition~\ref{pr:K-split}. Now apply Lemma~\ref{le:Z2}
    below.

(3) $\Rightarrow$ (1): Suppose there exists a simple object $S$ such
that $\tau S\not\cong S$. Then $\B=S^\perp$ satisfies (H1)--(H4) and
$K_0(\A)\cong \bbZ [S]\oplus K_0(\B)$, by Theorem~\ref{th:reduction}
and Lemma~\ref{le:reduction}.  If $\B$ is homogeneous, then
$K_0(\A)\cong \bbZ^3$ by the first part of the proof. Otherwise, we
proceed as before and reduce to the homogeneous case. In any case, the
rank of $K_0(\A)$ is at least $3$.

(1) $\Rightarrow$ (4): This follows from Proposition~\ref{pr:Lmorph}
since each infinite length object $A$ admits a subobject $A'$ such
that $A/A'$ is a line bundle.

(4) $\Rightarrow$ (1): Suppose there exists a simple object $S$ such
that $\tau S\not\cong S$. Then $\B=S^\perp={^\perp\tau S}$ yields an
expansion $\B\to\A$ by Lemma~\ref{le:reduction}, and this induces an
equivalence $\B/\B_0\to\A/\A_0$ by
Proposition~\ref{pr:onepoint_length}. Thus any infinite length object
in $\A$ yields one in $\B$, say $A$, satisfying $\Hom_\A(A,\tau S)=0$
by construction.
\end{proof}

\begin{lem}\label{le:Z2}
Let $G$ be a free abelian group of finite rank with a non-degenerate
bilinear form $\p$. Suppose there is a subgroup $0\neq H\subseteq G$ such
that $G/H\cong\bbZ$ and $\p(x,y)=0$ for all $x,y\in H$. Then $G\cong\bbZ^2$.
\end{lem}
\begin{proof}
The assumption on $H$ implies that for each pair of non-zero elements
$x,y\in H$, there are non-zero integers $\a_x,\a_y$ with $\a_x x=\a_y
y$. This implies $H\cong\bbZ$ since $H$ is free.
\end{proof}

\begin{lem}\label{le:homog_except}
Let $\A$ be a $k$-linear abelian category satisfying \emph{(H1)--(H4)}
and suppose that $\A$ is homogeneous. Then the following holds:
\begin{enumerate}
\item Each line bundle $L$ is exceptional.
\item There exists a simple object $S$ such that in $K_0(\A)$ the
class of each simple object $S'$ is of the form $[S']=n\cdot [S]$ for
some $n> 0$.
\end{enumerate}
\end{lem}
\begin{proof}
(1) Set $e_1=[L]$ and choose a second basis vector $e_2$ of $K_0(\A)\cong
\bbZ^2$ lying in the image of the map $K_0(\A_0)\to K_0(\A)$; see
Proposition~\ref{pr:K-split}. For $a=\a_1e_1+\a_2e_2$ in $K_0(\A)$, we
have $\langle a,a\rangle=\a_1^2\langle e_1,e_1\rangle$ since $\langle
e_1,e_2\rangle=-\langle e_2,e_1\rangle$ and $\langle
e_2,e_2\rangle=0$. From the existence of a tilting object $T$ in $\A$,
it follows that $\langle e_1,e_1\rangle>0$ since $\langle
[T],[T]\rangle>0$.  Now observe that $\langle
e_1,e_1\rangle=\dim_k\Hom_\A(L,L)-\dim_k\Ext_\A^1(L,L)$ is divisible by
$\dim_k\End_\A(L)$ since the endomorphism ring of any line bundle is a
division ring by Lemma~\ref{le:lb}. Thus $\Ext_\A^1(L,L)=0$.

(2) Choose a simple object $S$ such that $d=\langle[L],[S]\rangle$
is minimal. Given any simple object $S'$, there are integers
$q,r\geq 0$ with $\langle[L],[S']\rangle=q\cdot d+r$ and $r<d$.
Applying Proposition~\ref{pr:Lmorph} and Lemma \ref{le:filtr}, we
obtain extensions $0\to L\to E\to C\to 0$ and $0\to L\to E'\to C'\to
0$ such that $E$, $E'$ are line bundles, $[C]=q\cdot [S]$, and
$[C']=[S']$. Hence
$$\langle [E],[E']\rangle= \langle [L],[L]\rangle+\langle
[L],[S']\rangle-q\cdot\langle [L],[S]\rangle>0.$$ This gives an exact
sequence $0\to E\to E'\to F\to 0$, where $F$ is an object of finite
length and $\langle[L],[F]\rangle=r<d$. The minimality of
$\langle[L],[S]\rangle$ implies $F=0$, and therefore  $[S']=q\cdot [S]$.
\end{proof}

The class $[S]$ in Lemma~\ref{le:homog_except} yields a generator for
the image of the map $K_0(\A_0)\to K_0(\A)$. Thus for any finite
length object $A$ in $\A$, there exists some $n\ge 0$ with $[A]=n\cdot
[S]$.  We call this number the \index{degree}\emph{degree} of $A$ and
observe that it is independent of the choice of $S$.

Next we describe tilting objects for an abelian category that
satisfies (H1)--(H5) and is homogeneous.

\begin{prop}\label{pr:homog_tilt}
Let $\A$ be a $k$-linear abelian category satisfying \emph{(H1)--(H5)}
and suppose that $\A$ is homogeneous. Let $L$ be a line bundle and $S$
a simple object of degree one. Then
$$\Hom_\A(L,S)=k,\quad\Ext^1_\A(S,L)=k,\quad
\End_\A(L)=k, \quad\text{and}\quad \End_\A(S)=k.$$ Let $0\to L\to
L'\to S\to 0$ be a non-split extension.  Then $L\oplus L'$ is a
tilting object and its endomorphism algebra is isomorphic to the
Kronecker algebra (i.e.\ the path algebra of the quiver
$\xymatrix@=15pt{\cdot \ar@<.5ex>[r]\ar@<-.5ex>[r]&\cdot}$).
Moreover, the simple objects of degree one are precisely the objects
that arise as the cokernel of a non-zero morphism $L\to L'$.
\end{prop}
\begin{proof}
It follows from Proposition~\ref{pr:K-split} and
Lemma~\ref{le:homog_except} that $[L]$ and $[S]$ form a basis of
$K_0(\A)$. The corresponding matrix $\smatrix{\langle [L],[L]
\rangle&\langle [L],[S] \rangle\\ \langle [S],[L] \rangle&\langle
[S],[S] \rangle}$ has determinant $\pm 1$. Thus $1=\langle [L],[S]
\rangle=-\langle [S],[L] \rangle$. This implies $\Hom_\A(L,S)=k$ and
$\Ext^1_\A(S,L)=k$. The space $\Hom_\A(L,S)$ is a module over
$\End_\A(L)$ and over $\End_\A(S)$. It follows that $\End_\A(L)=k$ and
$\End_\A(S)=k$.

Next we show that $T=L\oplus L'$ is a tilting object. An application
of $\Hom_\A(L,-)$ to $0\to L\to L'\to S\to 0$ yields
$\Ext^1_\A(L,L')=0$, while application of $\Hom_\A(-,L)$ implies
$\Ext^1_\A(L',L)=0$. Thus $\Ext_\A^1(T,T)=0$. For any non-zero object
$A$ in $\A$, we have $\langle [L],[A] \rangle\neq 0$ or $\langle
[L'],[A]\rangle\neq 0$ since $[L]$ and $[L']$ form a basis of
$K_0(\A)$ and $[A]\neq 0$. Thus $T$ is a tilting object.

A simple computation shows that $\dim_k\Hom_\A(L,L')=2$, while
$\Hom_\A(L',L)=0$. Thus $\End_\A(T)$ is isomorphic to the Kronecker
algebra.

Let $\p\colon L\to L'$ be a non-zero morphism. This is a monomorphism
since $L$ is a line bundle. The cokernel $C=\Coker \p$ is of finite
length since $L$ and $L'$ have the same rank. The degree of $C$ is one
since $[C]=[L']-[L]=[S]$. In particular, $C$ is simple.

Now let $S'$ be a simple object of degree one. Choose a non-split
extension $0\to L\to E\to S'\to 0$. Then $E$ is a line bundle and
therefore exceptional by Lemma~\ref{le:homog_except}. We have
$[E]=[L']$ in $K_0(\A)$ since $[S]=[S']$, and it follows from
Proposition~\ref{pr:HR1} that $E\cong L'$. Thus $S'$ arises as the
cokernel of a morphism $L\to L'$.
\end{proof}

The next theorem provides an axiomatic description of the
category $\coh\bbP^1_k$.

\begin{thm}[Lenzing]\label{th:homog}
Let $\A$ be a $k$-linear abelian category satisfying \emph{(H1)--(H5)}
and suppose that $\A$ is homogeneous. Then $\A$ is equivalent to
$\coh\bbP^1_k$.
\end{thm}
\begin{proof}
The categories $\coh\bbP^1_k$ and $\A$ admit each a tilting object
such that its endomorphism algebra is isomorphic to the Kronecker
algebra, see Propositions~\ref{pr:coh_tilt} and \ref{pr:homog_tilt}.
Note that in both cases the indecomposable direct summands of a
tilting object have rank one.  Now apply
Proposition~\ref{pr:der_equiv}.
\end{proof}

\subsection{Coherent sheaves on weighted projective lines}

The following theorem characterizes the abelian categories that arise
as categories of coherent sheaves on weighted projective lines in the
sense of Geigle and Lenzing \cite{GL1987}.

\begin{thm}[Lenzing]\label{th:axioms}
Let $k$ be a field and $\A$ a $k$-linear abelian category. Then the
following are equivalent:
\begin{enumerate}
\item The category $\A$ satisfies \emph{(H1)--(H5)}.
\item There is a finite sequence $\A^0\subseteq\A^1\subseteq
\dots\subseteq\A^r=\A$ of full subcategories such that $\A^0$ is
equivalent to $\coh\bbP^1_k$ and $\A^{i+1}$ is a non-split
expansion of $\A^{i}$ with associated division ring $k$ for
$0\le i\le r-1$.
\item The category $\A$ is equivalent to $\coh\bbX$ for some weighted projective line $\bbX =
(\bbP^1_k,\bflambda,\bfp)$.
\end{enumerate}
\end{thm}
\begin{proof}
(1) $\Rightarrow$ (2): Suppose that $\A$ satisfies (H1)--(H5).  The rank of $K_0(\A)$
is finite, say $n$. So one constructs a filtration
$\A^0\subseteq\A^1\subseteq \dots\subseteq\A^r=\A$ of length $r=n-2$
by reducing the rank of the Grothendieck group as follows.  If $\A$ is
homogeneous, then $\A$ is equivalent to $\coh\bbP^1_k$ by
Theorem~\ref{th:homog}. Otherwise, there is a simple object $S$ such
that $S\not\cong\tau S$ by Proposition~\ref{pr:homog}. Then put
$\A^{r-1}=S^\perp$ and the inclusion $\A^{r-1}\to\A$ is a non-split
expansion by Lemma~\ref{le:reduction}. Moreover, $\A^{r-1}$
satisfies (H1)--(H5) by Theorem~\ref{th:reduction}, and the
associated division ring is $k$ by Lemma~\ref{le:euler_onepoint}. Note
that the rank of $K_0(\A^{r-1})$ is one less than that of
$K_0(\A)$. So one proceeds and constructs a sequence of subcategories
$\A^i$. The process stops after $r$ steps when $\A^0$ is homogeneous.

(2) $\Rightarrow$ (1): Suppose that $\A$ admits a filtration
$\A^0\subseteq\A^1\subseteq \dots\subseteq\A^r=\A$ such that $\A^0$
is equivalent to $\coh\bbP^1_k$ and $\A^{i+1}$ is a non-split
expansion of $\A^{i}$ with associated division ring $k$ for $0\le i\le
r-1$. The discussion in \S\ref{se:coh} shows that $\coh\bbP^1_k$
satisfies (H1)--(H5). An iterated application of
Theorem~\ref{th:reduction} yields that $\A$ satisfies (H1)--(H5).

(2) $\Rightarrow$ (3): Suppose again that $\A$ admits a sequence
$\A^0\subseteq\A^1\subseteq \dots\subseteq\A^r=\A$ of expansions such
that $\A^0$ is equivalent to $\coh\bbP^1_k$.  This yields a fully
faithful exact functor $\coh\bbP^1_k\to\A$, and it follows from
Proposition~\ref{pr:onepoint_comp} that this functor identifies the
index set of the decomposition \eqref{eq:equiv_tor}
$$ \coh_0\bbP^1_k\lto[\sim]\coprod_{0\neq\frp\in\Proj
k[x_0,x_1]}\mod_0\Oc_{\bbP^1_k,\frp}$$ into connected components with
the index set of the decomposition $\A_0=\coprod_{x\in\bfX}\A_x$.
Thus there is a canonical bijection between the set of closed points
of $\bbP^1_k$ and the set $\bfX$.  Moreover, if $x\in\bfX$ is a point
with $p(x)>1$, then the corresponding closed point $\frp$ of
$\bbP^1_k$ is rational since the residue field of the corresponding
local ring $\Oc_{\bbP^1_k,\frp}$ equals $k$. This follows from the
fact that in the filtration $\A^0\subseteq\A^1\subseteq
\dots\subseteq\A^r=\A$ the associated division ring of each expansion
equals $k$.

Let $\bflambda$ be the finite collection of points $\{x\in\bfX\mid
p(x)>1\}$, viewed as points of $\bbP^1_k$, and denote by $\bfp$ the
corresponding sequence of positive integers $p(x)$. Then there exists
a tilting object $T$ such that $\End_\A(T)\cong \Sq(\bfp,\bflambda)$;
see Proposition~\ref{pr:tiltobj} below.  On the other hand, let $\bbX
= (\bbP^1_k,\bflambda,\bfp)$ be the weighted projective line that is
determined by the parameters $\bflambda$ and $\bfp$.  The category
$\coh\bbX$ of coherent sheaves on $\bbX$ admits the following tilting
object
$$\Oc\oplus \Oc(\oc)\oplus(S_1^{[1]}\oplus\dots\oplus
S_1^{[p_1-1]})\oplus \dots \oplus(S_n^{[1]}\oplus\dots\oplus
S_n^{[p_n-1]})$$ with endomorphism algebra $\Sq(\bfp,\bflambda)$,
where the notation is taken from the introduction with $S_i=S_{i1}$;
see \cite[Example~4.4]{LM}.  This yields a derived equivalence
$\bfD^b(\A)\xto{\sim}\bfD^b(\coh\bbX)$ which restricts to an
equivalence $\A\xto{\sim}\coh\bbX$ by Proposition~\ref{pr:der_equiv}.

(3) $\Rightarrow$ (1):  See \cite{GL1987}.
\end{proof}

\begin{rem}
Let $\A$ be a $k$-linear abelian category satisfying (H1)--(H5).

(1) The reduction to the homogeneous case in the proof of
Theorem~\ref{th:axioms} shows that the rank of the Grothendieck group
of $\A$ is $2+\sum_{x\in\bfX}(p(x)-1)$.

(2) Let $x\in\bfX$ and $p(x)>1$. Then $\A_x$ is equivalent to the
category of finite dimensional nilpotent representations of a
quiver of extended Dynkin type $\tilde\bbA_{p(x)-1}$ with cyclic
orientation; see Example~\ref{ex:cyclic}.
\end{rem}

\subsection{A tilting object}

Let $k$ be a field and $\A$ a $k$-linear abelian category satisfying
(H1)--(H5). We construct a tilting object and compute its endomorphism
algebra, which is a \index{squid algebra}\emph{squid algebra} in the
sense of Brenner and Butler \cite{BB}.

Given a collection $\bflambda = (\la_1,\dots,\la_n)$ of distinct
rational points $\la_i = [\lambda_{i0}:\lambda_{i1}]$ of $\bbP^1_k$,
and a sequence $\bfp=(p_1,\dots,p_n)$ of positive integers, we define
$\Sq(\bfp,\bflambda)$ to be the finite dimensional associative algebra
given by the quiver \setlength{\unitlength}{1.5pt}
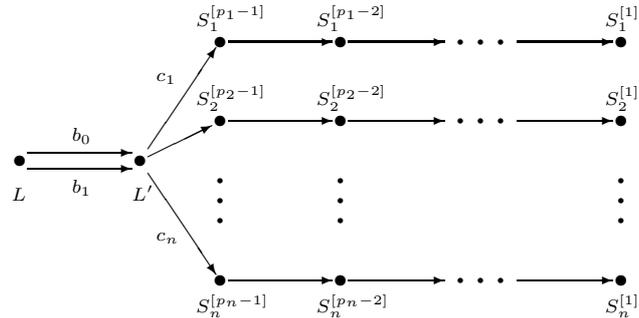
\[
\begin{picture}(110,80)
\put(-20,40){\circle*{2.5}}
\put(10,40){\circle*{2.5}}
\put(30,10){\circle*{2.5}}
\put(30,50){\circle*{2.5}}
\put(30,70){\circle*{2.5}}
\put(60,10){\circle*{2.5}}
\put(60,50){\circle*{2.5}}
\put(60,70){\circle*{2.5}}
\put(130,10){\circle*{2.5}}
\put(130,50){\circle*{2.5}}
\put(130,70){\circle*{2.5}}
\put(-18,38){\vector(1,0){26}}
\put(-18,42){\vector(1,0){26}}
\put(12,43){\vector(2,3){17}}
\put(12,41){\vector(2,1){16}}
\put(12,37){\vector(2,-3){17}}
\put(32,10){\vector(1,0){26}}
\put(32,50){\vector(1,0){26}}
\put(32,70){\vector(1,0){26}}
\put(62,10){\vector(1,0){24}}
\put(62,50){\vector(1,0){24}}
\put(62,70){\vector(1,0){24}}
\put(104,10){\vector(1,0){24}}
\put(104,50){\vector(1,0){24}}
\put(104,70){\vector(1,0){24}}
\put(90,10){\circle*{1}}
\put(95,10){\circle*{1}}
\put(100,10){\circle*{1}}
\put(90,50){\circle*{1}}
\put(95,50){\circle*{1}}
\put(100,50){\circle*{1}}
\put(90,70){\circle*{1}}
\put(95,70){\circle*{1}}
\put(100,70){\circle*{1}}
\put(30,25){\circle*{1}}
\put(30,30){\circle*{1}}
\put(30,35){\circle*{1}}
\put(60,25){\circle*{1}}
\put(60,30){\circle*{1}}
\put(60,35){\circle*{1}}
\put(130,25){\circle*{1}}
\put(130,30){\circle*{1}}
\put(130,35){\circle*{1}}
\put(8,30){$\scriptstyle{L'}$}
\put(-22,30){$\scriptstyle{L}$}
\put(-7,45){$\scriptstyle{b_0}$}
\put(-7,32){$\scriptstyle{b_1}$}
\put(14,60){$\scriptstyle{c_{1}}$}
\put(14,20){$\scriptstyle{c_{n}}$}
\put(24,2){$\scriptstyle{S_n^{[p_n-1]}}$}
\put(24,54){$\scriptstyle{S_2^{[p_2-1]}}$}
\put(24,74){$\scriptstyle{S_1^{[p_1-1]}}$}
\put(54,2){$\scriptstyle{S_n^{[p_n-2]}}$}
\put(54,54){$\scriptstyle{S_2^{[p_2-2]}}$}
\put(54,74){$\scriptstyle{S_1^{[p_1-2]}}$}
\put(126,2){$\scriptstyle{S_n^{[1]}}$}
\put(126,54){$\scriptstyle{S_2^{[1]}}$}
\put(126,74){$\scriptstyle{S_1^{[1]}}$}
\put(42,2){}
\put(42,54){}
\put(42,74){}
\end{picture}
\]
modulo the relations
$$c_i(\lambda_{i0} b_1 - \lambda_{i1} b_0)=0 \quad (i=1,\dots,n).$$

\begin{prop}[Lenzing-Meltzer]\label{pr:tiltobj}
A $k$-linear abelian category satisfying \emph{(H1)--(H5)}
admits a tilting object with endomorphism algebra isomorphic to
$\Sq(\bfp,\bflambda)$ for some pair $\bfp,\bflambda$.
\end{prop}
\begin{proof}
Fix a $k$-linear abelian category $\A$ satisfying (H1)--(H5). We apply
Theorem~\ref{th:axioms}(2) and follow its proof.  Thus there exists a
sequence $\B=\A^0\subseteq\A^1\subseteq \dots\subseteq\A^r=\A$ of
expansions such that $\B$ is equivalent to $\coh\bbP^1_k$.  Let
$(x_1,\dots, x_n)$ be the collection of distinct points $x\in\bfX$
with $p(x)>1$ and set $\bfp=(p_1,\dots,p_n)$ with $p_i=p(x_i)$ for
each $i$.  Choose line bundles $L$ and $L'$ in $\B$ forming a tilting
object $L\oplus L'$ for $\B$, and choose a basis $b_0,b_1$ of
$\Hom_\A(L,L')$; see Proposition~\ref{pr:homog_tilt}.  The inclusion
$F\colon \B\to\A$ restricts to a family of inclusions
$\B_{x_i}\to\A_{x_i}$; see Proposition~\ref{pr:onepoint_comp}. Note
that each inclusion $\B_{x_i}\to\A_{x_i}$ is a composite of $p_i$
non-split expansions.  Thus there are simple objects $S'_i\in\B_{x_i}$
and $S_i\in\A_{x_i}$ such that $FS'_i=S_i^{[p_i]}$. Here,
$S_i^{[p_i]}$ denotes the uniserial object with top $S_i$ and length
$p_i$, and we use that an expansion sends a specific simple object to
an object of length two; see Lemma~\ref{le:simple}.  In particular,
$S_i^{[j]}$ belongs to $^\perp\B$ for $1\leq
j<p_i$. Note that
$\End_\A(S'_i)=k$ since the division ring of each expansion is $k$.
In $\coh\bbP^1_k$, a simple object with trivial endomorphism ring has
degree one. Thus each simple object $S'_i$ fits into an exact sequence
\begin{equation}\label{eq:squid}
0\lto L\xto{\lambda_{i0} b_1 - \lambda_{i1} b_0} L'\lto S'_i\lto 0
\end{equation}
by Proposition~\ref{pr:homog_tilt}, and this yields a collection
$\bflambda=(\la_1,\dots,\la_n)$ of rational points
$\la_i=[\la_{i0}:\la_{i1}]$ in $\bbP^1_k$. Moreover, there are
canonical morphisms $c_i\colon L'\twoheadrightarrow
S_i^{[p_i]}\twoheadrightarrow S_i^{[p_i-1]}$ in $\A$ satisfying the
relations $c_i(\lambda_{i0} b_1 - \lambda_{i1} b_0)=0$. It is
straightforward to verify that the object
$$T=L\oplus L'\oplus(S_1^{[1]}\oplus\dots\oplus S_1^{[p_1-1]})\oplus
\dots \oplus(S_n^{[1]}\oplus\dots\oplus S_n^{[p_n-1]})$$ is a
tilting object for $\A$, using the criterion of
Lemma~\ref{le:euler_tilt}. Indeed, $\Ext_\A^1(T,T)=0$, the
indecomposable direct summands of $T$ yield a basis of $K_0(\A)$, and
$[A]\neq 0$ for each object $A\neq 0$ by
Proposition~\ref{pr:K-split}. Finally, one checks that $\End_\A(T)$ is
isomorphic to $\Sq(\bfp,\bflambda)$. Here, one uses that each
epimorphism $S_i^{[p_i]}\to S_i^{[j]}$ induces an isomorphism
$\Hom_\A(L\oplus L',S_i^{[p_i]})\xto{\sim} \Hom_\A(L\oplus
L',S_i^{[j]})$.
\end{proof}

\section{Canonical algebras}
The canonical algebras in the sense of Ringel \cite{R1984,R} provide a
link between weighted projective lines and the representation theory
of finite dimensional algebras. In fact, Geigle and Lenzing
constructed in \cite{GL1987} for each weighted projective line $\bbX =
(\bbP^1_k,\bflambda,\bfp)$ a tilting object in $\coh\bbX$ such that
its endomorphism algebra is isomorphic to the canonical algebra with
the same parameters; see Example~\ref{ex:can}.  It turns out that this
tilting object is somehow canonical. From this it follows that the
parameters $\bflambda$ and $\bfp$ can be reconstructed from the
category $\coh\bbX$. To be more precise, we consider an abelian
category $\A$ that is equivalent to $\coh\bbX$ for some weighted
projective line $\bbX$. For each line bundle $L$ in $\A$ one
constructs a canonical tilting object $T_L$ such that its endomorphism
algebra is a canonical algebra. Then one shows that for each pair of
line bundles $L,L'$ there is a sequence of tubular mutations in the
sense of Lenzing and Meltzer \cite{LM1993,M,L2007} that yields an
equivalence $\A\xto{\sim}\A$ taking $L$ to $L'$ and therefore $T_L$ to
$T_{L'}$. In particular, the endomorphism algebras of $T_L$ and
$T_{L'}$ are isomorphic and therefore an invariant of $\A$. Finally
one observes that the parameters $\bflambda$ and $\bfp$ form an
invariant of the canonical algebra $\End_\A(T_L)$.

Throughout this section we fix an arbitrary field $k$.

\subsection{Canonical algebras from weighted projective lines}
Consider a collection $\bflambda = (\la_1,\dots,\la_n)$ of distinct
rational points $\la_i = [\lambda_{i0}:\lambda_{i1}]$ of $\bbP^1_k$,
and a sequence $\bfp=(p_1,\dots,p_n)$ of positive integers. We define
the \index{canonical algebra}\emph{canonical algebra}
$C(\bfp,\bflambda)$ to be the finite dimensional associative algebra
given by the quiver \setlength{\unitlength}{1.5pt}
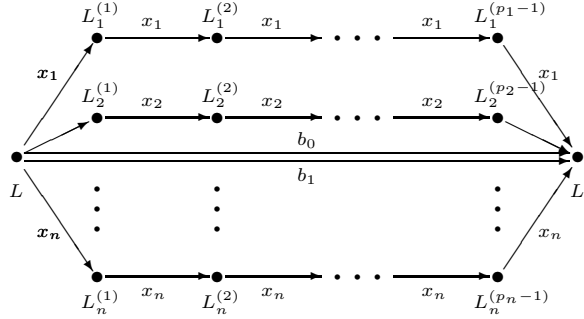
\[
\begin{picture}(160,80)
\put(10,40){\circle*{2.5}}
\put(30,10){\circle*{2.5}}
\put(30,50){\circle*{2.5}}
\put(30,70){\circle*{2.5}}
\put(60,10){\circle*{2.5}}
\put(60,50){\circle*{2.5}}
\put(60,70){\circle*{2.5}}
\put(130,10){\circle*{2.5}}
\put(130,50){\circle*{2.5}}
\put(130,70){\circle*{2.5}}
\put(150,40){\circle*{2.5}}
\put(12,39){\vector(1,0){136}}
\put(12,41){\vector(1,0){136}}
\put(12,43){\vector(2,3){17}}
\put(12,41){\vector(2,1){16}}
\put(12,37){\vector(2,-3){17}}
\put(32,10){\vector(1,0){26}}
\put(32,50){\vector(1,0){26}}
\put(32,70){\vector(1,0){26}}
\put(62,10){\vector(1,0){24}}
\put(62,50){\vector(1,0){24}}
\put(62,70){\vector(1,0){24}}
\put(104,10){\vector(1,0){24}}
\put(104,50){\vector(1,0){24}}
\put(104,70){\vector(1,0){24}}
\put(90,10){\circle*{1}}
\put(95,10){\circle*{1}}
\put(100,10){\circle*{1}}
\put(90,50){\circle*{1}}
\put(95,50){\circle*{1}}
\put(100,50){\circle*{1}}
\put(90,70){\circle*{1}}
\put(95,70){\circle*{1}}
\put(100,70){\circle*{1}}
\put(30,22){\circle*{1}}
\put(30,27){\circle*{1}}
\put(30,32){\circle*{1}}
\put(60,22){\circle*{1}}
\put(60,27){\circle*{1}}
\put(60,32){\circle*{1}}
\put(130,22){\circle*{1}}
\put(130,27){\circle*{1}}
\put(130,32){\circle*{1}}
\put(131.5,12){\vector(2,3){17}}
\put(132,49){\vector(2,-1){16}}
\put(131.5,68){\vector(2,-3){17}}
\put(8,30){$\scriptstyle{L}$}
\put(148,30){$\scriptstyle{L'}$}
\put(80,43){$\scriptstyle{b_0}$}
\put(80,34){$\scriptstyle{b_1}$}
\put(15,60){$\scriptstyle{x_{1}}$}
\put(15,20){$\scriptstyle{x_{n}}$}
\put(140,60){$\scriptstyle{x_{1}}$}
\put(140,20){$\scriptstyle{x_{n}}$}
\put(15,60){$\scriptstyle{x_{1}}$}
\put(15,20){$\scriptstyle{x_{n}}$}
\put(26,2){$\scriptstyle{L_n^{(1)}}$}
\put(26,54){$\scriptstyle{L_2^{(1)}}$}
\put(26,74){$\scriptstyle{L_1^{(1)}}$}
\put(56,2){$\scriptstyle{L_n^{(2)}}$}
\put(56,54){$\scriptstyle{L_2^{(2)}}$}
\put(56,74){$\scriptstyle{L_1^{(2)}}$}
\put(124,2){$\scriptstyle{L_n^{(p_n-1)}}$}
\put(124,54){$\scriptstyle{L_2^{(p_2-1)}}$}
\put(124,74){$\scriptstyle{L_1^{(p_1-1)}}$}
\put(41,5){$\scriptstyle{x_n}$}
\put(41,53){$\scriptstyle{x_2}$}
\put(41,73){$\scriptstyle{x_1}$}
\put(71,5){$\scriptstyle{x_n}$}
\put(71,53){$\scriptstyle{x_2}$}
\put(71,73){$\scriptstyle{x_1}$}
\put(111,5){$\scriptstyle{x_n}$}
\put(111,53){$\scriptstyle{x_2}$}
\put(111,73){$\scriptstyle{x_1}$}

\end{picture}
\]
modulo the relations\footnote{Note that the relations do not generate
an admissible ideal of the path algebra, except when the collection
$\bflambda$ is empty. In that case $C(\bfp,\bflambda)$ equals the
Kronecker algebra.}
$$x_i^{p_i}=\lambda_{i0} b_1 - \lambda_{i1} b_0 \quad (i=1,\dots,n).$$

\begin{thm}[Geigle-Lenzing]\label{th:tiltobj}
A $k$-linear abelian category satisfying \emph{(H1)--(H5)} admits a
tilting object $T$ such that $T$ is a direct sum of line bundles and
the endomorphism algebra of $T$ is isomorphic to $C(\bfp,\bflambda)$
for some pair $\bfp,\bflambda$.
\end{thm}
\begin{proof}
Fix a $k$-linear abelian category $\A$ satisfying (H1)--(H5). We adapt
the proof of Proposition~\ref{pr:tiltobj} and modify the tilting
object constructed in that proof as follows. Consider for each point
$x_i\in\bfX$ with $p(x_i)>1$ the exact sequence \eqref{eq:squid}
$$0\lto L\lto[\p_i] L'\lto S_i^{[p_i]}\lto 0$$ in
$\A$, where $\p_i=\lambda_{i0} b_1 - \lambda_{i1} b_0$. We form
successively the pullback along the chain of monomorphisms
$$(\tau^{-1}S_i)_{[1]}\rightarrowtail(\tau^{-1}S_i)_{[2]}\rightarrowtail\cdots
\rightarrowtail (\tau^{-1} S_i)_{[p_i-1]}\rightarrowtail (\tau^{-1}S_i)_{[p_i]}=S_i^{[p_i]}$$ and
obtain the following commutative diagram with exact columns.
\begin{equation}\label{eq:tilt}
\xymatrixrowsep{1.5pc} \xymatrixcolsep{1.5pc}\xymatrix{
0\ar[d]&0\ar[d]&&0\ar[d]&0\ar[d]\\
L\ar[d]\ar@{=}[r]&L\ar[d]\ar@{=}[r]&\hdots\ar@{=}[r]&
L\ar[d]\ar@{=}[r]&L\ar[d]^{\p_i}\\
L_i^{(1)}\ar[d]\ar[r]&L_i^{(2)}\ar[d]\ar[r]&\hdots\ar[r]&L_i^{(p_i-1)}\ar[d]\ar[r]&L'\ar[d]\\
(\tau^{-1}S_i)_{[1]}\ar[d]\ar[r]&(\tau^{-1}S_i)_{[2]}\ar[d]\ar[r]&\hdots\ar[r]&(\tau^{-1}
S_i)_{[p_i-1]}\ar[d]\ar[r]&(\tau^{-1}S_i)_{[p_i]}\ar[d]\\ 0&0&&0&0 }
\end{equation}
Note that each object $L_i^{(j)}$ is a line bundle.  It is
straightforward to verify that the object
$$T=L\oplus L'\oplus(L_1^{(1)}\oplus\dots\oplus L_1^{(p_1-1)})\oplus
\dots \oplus(L_n^{(1)}\oplus\dots\oplus L_n^{(p_n-1)})$$ is a
tilting object for $\A$, following the line of arguments in the proof of
Proposition~\ref{pr:tiltobj}.
\end{proof}

\begin{cor}
\pushQED{\qed} Let $\A$ be a $k$-linear abelian category satisfying
\emph{(H1)--(H5)}. Then there exists for some pair $\bfp,\bflambda$ an
equivalence of derived categories
\[\bfD^b(\A)\lto[\sim]\bfD^b(\mod C(\bfp,\bflambda)). \qedhere\]
\end{cor}

\subsection{Vector bundle presentations}

Let $\A$ be a $k$-linear abelian category satisfying (H1)--(H4). The
category $\A_+$ consisting of the vector bundles in $\A$ determines
together with its exact structure the category $\A$. In fact, every
object $A$ in $\A$ admits a presentation $0\to A_1\to A_0\to A\to 0$
with $A_0,A_1$ in $\A_+$. Introducing appropriate morphisms between
complexes in $\A_+$, one can make these presentations functorial.

\begin{lem}\label{le:vb_pres}
Every object in $\A$ is a factor object of an object in $\A_+$.
\end{lem}

\begin{proof}
Every object of $\A$ decomposes into an object of $\A_+$ and an
object of finite length. Thus it suffices to show that for an
indecomposable object $A$ of finite length there is an epimorphism
$E\to A$ with $E$ in $\A_+$. We use induction on the length
$\ell(A)$ of $A$. Up to certain power of $\tau$, the case
$\ell(A)=1$ follows from Proposition~\ref{pr:Lmorph}. Assume that
$\ell(A)>1$. Take a maximal subobject $A'\subseteq A$.  By the
induction hypothesis there is an epimorphism $\phi\colon E'\to A'$.
Note that $\phi$ induces an epimorphism $\Ext^1_\A(A/A', E')\to
\Ext^1_\A(A/A', A')$. In particular, we obtain the following
commutative diagram with exact rows.
\[\xymatrix{0 \ar[r] & E' \ar[d]^{\phi} \ar[r] & E \ar[r] \ar[d] &
A/A' \ar@{=}[d]\ar[r] &0\\ 0\ar[r] & A' \ar[r] & A \ar[r] & A/A'
\ar[r] & 0.  }\] Note that the upper row does not split since $A$ is
indecomposable. Thus the morphism $E'\to E$ induces a bijection
$\Hom_\A(S,E')\to\Hom_\A(S,E)$ for each simple object $S$. It follows
that $E$ belongs to $\A_+$, and therefore $A$ is a quotient of an
object in $\A_+$.
\end{proof}

The following result says that the subcategory $\A_+$ determines
$\A$. However, it is important to notice that one uses the exact
structure on $\A_+$ that is inherited from $\A$. To be precise, a
morphism of complexes in $\A_+$ is by definition a
\emph{quasi-isomorphism} if it is a quasi-isomorphism  of complexes
in $\A$.

\begin{prop}\label{pr:vb_pres}
Let $\A$ be a $k$-linear abelian category satisfying
\emph{(H1)--(H4)}. The inclusion $\A_+\to\A$ induces an equivalence
\begin{equation*}\label{eq:vb_pres}
\bfK^b(\A_+)[\qis^{-1}]\lto[\sim]\bfD^b(\A).
\end{equation*}

\end{prop}
\begin{proof}
The subcategory $\A_+$ of $\A$ is closed under forming extensions and
taking kernels of epimorphisms. Moreover, each object $A$ in $\A$ fits
into an exact sequence $0\to A_n\to\cdots\to A_1\to A_0\to A\to 0$
with each $A_i$ in $\A_+$; see Lemma~\ref{le:vb_pres}. With these
properties, the assertion follows from \cite[Chap.~III,
Prop.~2.4.3]{V}.
\end{proof}

\begin{cor}\label{co:vb_pres}
Let $F\colon\A_+\xto{\sim}\A_+$ be an equivalence and suppose that a
sequence $\xi$ in $\A_+$ is exact if and only if $F\xi$ is exact. Then
$F$ extends uniquely to an equivalence $\A\xto{\sim}\A$.
\end{cor}
\begin{proof}
We apply Proposition~\ref{pr:vb_pres}. Thus the assumption on $F$
implies that the functor extends to an equivalence
$\bfD^b(\A)\xto{\sim}\bfD^b(\A)$. This equivalence restricts to an
equivalence $\A\xto{\sim}\A$ because $\A$ identifies with the full
subcategory consisting of complexes $$\cdots \to 0\to
A_1\xto{\delta}A_0\to 0\to\cdots$$ with $A_0,A_1$ in $\A_+$ and
$\delta$ a monomorphism; see Lemma \ref{le:vb_pres}.
\end{proof}

\subsection{Line bundles and tubular mutations}

Let $\A$ be a $k$-linear abelian category satisfying (H1)--(H5) and
$\A_0=\coprod_{x\in\bfX}\A_x$ be the decomposition of $\A_0$ into
connected uniserial components. For each $x\in\bfX$ denote by $T_x$ the
direct sum of a representative set of simple objects in $\A_x$ and set
$\T_x=\add T_x$.

Note that $\T_x$ is a Hom-finite semisimple abelian category with
finitely many simple objects. Thus each additive functor $\T_x\to \mod
k$ is representable. This observation yields two functors
$\bar\d_x,\bar\e_x\colon\A\to\T_x$ such that for each object $A$ in $\A$
$$\Hom_\A(A,-)|_{\T_x}\cong\Hom_{\T_x}(\bar\d_x A,-)\quad\text{and}\quad
\Ext^1_\A(-,A)|_{\T_x}\cong\Hom_{\T_x}(-,\bar\e_x A).$$

Fix an object $A$ in $\A_+$.  The identity morphism of $\bar\d_x A$
corresponds to a morphism $A\to \bar\d_x A$. This is an epimorphism and we
complete it to an exact sequence $0\to \d_x A\to A\to \bar\d_x A\to 0$. On
the other hand, the identity morphism of $\bar\e_x A$ corresponds to an
exact sequence $0\to A \to \e_x A\to \bar\e_x A\to 0$. It is easily checked
that this defines two functors $\d_x,\e_x\colon\A_+\to\A$.

The following lemma shows that $\d_x$ and $\e_x$ yield equivalences
$\A_+\xto{\sim}\A_+$.

\begin{lem}\label{le:mutat}
Let $\xi\colon 0\to A\to A'\to T\to 0$ be an exact sequence in $\A$
with $T$ in $\T_x$ for some $x\in\bfX$. Then the
following are equivalent:
\begin{enumerate}
\item $\xi$ induces an iso
$\Hom_\A(-,T)|_{\T_x}\xto{\sim}\Ext^1_\A(-,A)|_{\T_x}$ and $A$ belongs
to $\A_+$.
\item $\xi$ induces an iso
$\Hom_\A(A,-)|_{\T_x}\xto{\sim}\Ext^1_\A(T,-)|_{\T_x}$ and $A$ belongs
to $\A_+$.
\item $\xi$ induces an iso
$\Hom_\A(T,-)|_{\T_x}\xto{\sim}\Hom_\A(A',-)|_{\T_x}$ and $A'$ belongs to $\A_+$.
\end{enumerate}
\end{lem}
\begin{proof}
(1) $\Leftrightarrow$ (2): Apply Serre duality and observe that $\tau
    T_x\cong T_x$.

(1) \& (2) $\Rightarrow$ (3): Let $S$ be any simple object in $\A$ and
apply $\Hom_\A(S,-)$ to $\xi$. Using (1) and the fact that $A$ belongs
to $\A_+$, it follows that $\Hom_\A(S,A')=0$. Thus $A'$ is in $\A_+$.

Now let $S$ be any object in $\T_x$ and apply $\Hom_\A(-,S)$ to $\xi$. This yields the following exact sequence
$$0\to\Hom_\A(T,S)\xto{\a}\Hom_\A(A',S)\xto{\b}\Hom_\A(A,S)\xto{\g}\Ext^1_\A(T,S)\to
0$$ where $\a$ is an isomorphism if and only if $\g$ is an
isomorphism. Thus (3) holds.

(3) $\Rightarrow$ (2): The object $A$ is in $\A_+$ since $\A_+$
    is closed under taking subobjects. The rest follows as before by
choosing $S$ in $\T_x$ and    applying $\Hom_\A(-,S)$ to $\xi$.
\end{proof}

\begin{prop}
The functors $\d_x$ and $\e_x$ form a pair of mutually inverse
equivalences $\A_+\xto{\sim}\A_+$. Moreover, $\d_x$ and $\e_x$ take
exact sequences in $\A_+$ to exact sequences.
\end{prop}
\begin{proof}
The first assertion is an immediate consequence of
Lemma~\ref{le:mutat}. For the exactness observe that $\bar\d_x$ and
$\bar\e_x$ are exact when restricted to $\A_+$.  The exactness of
$\d_x$ and $\e_x$ then follows from the $3\times3$ lemma.
\end{proof}

Using Corollary~\ref{co:vb_pres}, the functors $\d_x,\e_x$ yield
equivalences $\A\xto{\sim}\A$. These functors are called
\index{tubular mutation}\emph{tubular mutations} and were introduced
by Lenzing and Meltzer \cite{LM1993,M,L2007}.

\begin{cor}\label{co:mutat}
\pushQED{\qed} The equivalences $\d_x,\e_x\colon\A_+\xto{\sim}\A_+$
extend to a pair of mutually inverse equivalences $\A\xto{\sim}\A$.
\qedhere
\end{cor}

Next we show that for each pair of line bundles $L,L'$ there exists a
sequence of tubular mutations taking $L$ to $L'$. We need the
following proposition which is of independent interest.

\begin{prop}\label{pr:hom_lb}
Let $L$ be a line bundle. For each $x\in\bfX$ there exists up to
isomorphism a unique simple object $S_x$ in $\A_x$ such that
$\Hom_\A(L,S_x)\neq 0$. Moreover, any non-zero morphism $L\to S_x$
induces an isomorphism $\End_\A(S_x)\xto{\sim}\Hom_\A(L,S_x)$.
\end{prop}
\begin{proof}
For the purpose of this proof, call a line bundle $L$
\emph{Hom-simple} if the assertion of the proposition holds for
$L$. We begin by showing that a specific line bundle is Hom-simple.

Fix a sequence of subcategories $\A^0\subseteq
\A^1\subseteq\dots\subseteq\A^r=\A$ as in Theorem~\ref{th:axioms}
and choose a line bundle $L$ in $\A^0$. Note that the inclusion
$\A^0\to\A$ sends $L$ to a line bundle of $\A$ by
Proposition~\ref{pr:onepoint_length}.  For each $x\in\bfX$, there is
at least one simple object $S_x\in\A_x$ with $\Hom_\A(L,S_x)\neq 0$
by Proposition~\ref{pr:Lmorph}. On the other hand, the intersection
of $(\A^0)^\perp$ with $\A_x$ is a Serre subcategory having $p(x)-1$
simple objects. In particular, $\Hom_\A(L,-)$ vanishes on them. Thus
there is a unique simple object $S_x$ in $\A_x$ with
$\Hom_\A(L,S_x)\neq 0$.

Let $i_\rho\colon\A\to\A^0$ denote the right adjoint of the inclusion
$\A^0\to\A$. Then $i_\rho S_x$ is a simple object by
Lemma~\ref{le:simple}.  Choose a non-zero morphism $L\to S_x$. This
yields the following commutative square, where $\eta\colon i_\rho S_x\to S_x$
denotes the adjunction morphism.
$$\xymatrix{\End_{\A^0}(i_\rho S_x)\ar[r]&\Hom_{\A^0}(L,i_\rho S_x)\ar[d]^{\Hom_\A(L,\eta)}\\
\End_{\A}(S_x)\ar[r]\ar[u]^{i_\rho}&\Hom_{\A}(L,S_x)}$$
The map $\Hom_\A(L,\eta)$ is an isomorphism since $L$ belongs to
$\A^0$, and $i_\rho$ induces an isomorphism
$\End_{\A}(S_x)\xto{\sim}\End_{\A^0}(i_\rho S_x)$ since $i_\rho$ is a
quotient functor and $S_x$ is simple; see
Lemma~\ref{le:quotient_simple}. Finally observe that $\A^0$ is
equivalent to $\coh\bbP^1_k$. Thus the induced map $\End_{\A^0}(i_\rho
S_x)\to\Hom_{\A^0}(L,i_\rho S_x)$ is an isomorphism because we may
assume that $L$ corresponds to the structure sheaf; see
Remark~\ref{re:hom-simple}. It follows that
$\End_\A(S_x)\cong\Hom_\A(L,S_x)$.

Having shown the assertion for a specific line bundle, we apply
Lemma~\ref{le:lb} to verify the assertion for an arbitrary line bundle.
Thus we need to show that for any pair $L,L'$ of line bundles and each
monomorphism $\p\colon L\to L'$ with cokernel in $\A_0$,
the object $L$ is Hom-simple if
and only if $L'$ is Hom-simple.

Using induction on the length of the cokernel of $\p$, we may assume
that the cokernel is simple.  The exact sequence $0\to L\xto{\p} L'\to
S\to 0$ induces for each simple object $T$ an exact sequence
$$0\to\Hom_\A(S,T)\xto{\a}\Hom_\A(L',T)\xto{\b}\Hom_\A(L,T)\xto{\g}D\Hom_\A(T,\tau
S)\to 0.$$ If $L$ is Hom-simple, then a simple calculation shows
that $\g$ is an isomorphism. Thus $\a$ is an isomorphism and it
follows that $L'$ is Hom-simple. The same argument shows that $L$ is
Hom-simple if $L'$ is Hom-simple.
\end{proof}

The following result is due to Kussin; see \cite[Proposition~4.2.3]{Ku}.

\begin{prop}\label{pr:shift}
Let $L,L'$ be two line bundles in $\A$. Then there exists an
equivalence $\A\xto{\sim}\A$ taking $L$ to $L'$. In particular, each
line bundle is exceptional.
\end{prop}
\begin{proof}
We apply Lemma~\ref{le:lb}. Thus we may assume that there is an
exact sequence $\xi\colon 0\to L\to L'\to C\to 0$ with $C$ of finite
length. Using induction on the length of $C$, we may even assume
that $C$ is simple. Suppose that $C$ belongs to $\A_x$ with
$x\in\bfX$. Then $\xi$ induces an isomorphism
$\Hom_\A(-,C)|_{\T_x}\xto{\sim}\Ext^1_\A(-,L)|_{\T_x}$ by
Proposition~\ref{pr:hom_lb} and Serre duality. Thus $L'=\e_x L$ and
the tubular mutation given by $\e_x$ sends $L$ to $L'$; see
Corollary~\ref{co:mutat}.
\end{proof}

\subsection{Weight functions}

Let $k$ be a field and $\A$ a $k$-linear abelian category satisfying
(H1)--(H5). We associate to $\A$ a weight function and show that it
is an invariant of $\A$ which determines $\A$ up to equivalence.

In the following we identify the projective linear group $\PGL(2,k)$
with $\Aut\bbP^1_k$; see Proposition~\ref{pr:aut}.

A \index{weight function}\emph{weight function} $w\colon \bbP^1_k\to
\bbZ$ is a map which assigns to each closed point $x$ of $\bbP^1_k$ a
positive integer $w(x)$ such that $w(x)=1$ for almost all $x$. Two
weight functions $w,w'$ are \index{weight
function!equivalent}\emph{equivalent} if there exists some linear
transformation $\sigma\in\PGL(2,k)$ such that $w'(x)=w(\sigma x)$ for
every closed point $x$.  Given a collection $\bflambda =
(\la_1,\dots,\la_n)$ of distinct closed points $\la_i\in\bbP^1_k$, and
a sequence $\bfp=(p_1,\dots,p_n)$ of positive integers, there is
associated a weight function $w_{\bfp,\bflambda}$, where
$w_{\bfp,\bflambda}(\la_i)=p_i$ for $1\leq i\leq n$ and
$w_{\bfp,\bflambda}(x)=1$ otherwise.

It follows from our convention that each weight function $w$ corresponding
to a weighted projective line satisfies $w(x)=1$ if the point $x$ is not rational.

\begin{thm}[Lenzing]\label{th:canon_tiltobj}
Let $\A$ be a $k$-linear abelian category satisfying
\emph{(H1)--(H5)}. For each line bundle $L$, there exists a
canonically\footnote{This canonical choice provides another
justification for the term `canonical algebra'.} defined tilting
object $T_L$ which is unique up to  isomorphism. Moreover:
\begin{enumerate}
\item The object $T_L$ determines parameters $\bfp,\bflambda$ such
that $\End_\A(T_L)\cong C(\bfp,\bflambda)$. The parameters
$\bfp,\bflambda$ depend on a choice and any other choice gives
parameters $\bfp',\bflambda'$ such that the weight functions
$w_{\bfp,\bflambda}$ and $w_{\bfp',\bflambda'}$ are equivalent.
\item Let $M$ be a second line bundle. Then there exists an
equivalence $\A\xto{\sim}\A$ taking $L$ to $M$. Thus $\End_\A(T_L)$
is isomorphic to $\End_\A(T_{M})$ and the associated weight functions
are equivalent.
\end{enumerate}
\end{thm}
\begin{proof}
Let $(x_1,\dots, x_n)$ be the collection of distinct points
$x\in\bfX$ with $p(x)>1$ and set $\bfp=(p_1,\dots,p_n)$ with
$p_i=p(x_i)$ for each $i$. We apply Proposition~\ref{pr:hom_lb} and
choose for each $i$ a simple object $S_i=S_{x_i}$ such that
$\Hom_\A(L,S_i)\neq 0$. Thus $\Ext_\A^1(\tau^{-1} S_i,L)\neq 0$ by
Serre duality, and Lemma~\ref{le:filtr} yields a chain of
monomorphism
\begin{equation}\label{eq:tilt_chain}
L\lto[\psi_1] L_i^{(1)} \lto[\psi_2] L_i^{(2)}\lto[\psi_3]\cdots
\lto[\psi_{p_i}] L_i^{(p_i)}
\end{equation}
with $\Coker\psi_j\cong\tau^{-j}S_i$ for all $j$.  Note that each
object $L_i^{(j)}$ is a line bundle and therefore exceptional.

The class $[L_i^{(p_i)}]=[L]-\sum_{j=1}^{p(x)}[\tau^{j}S_i]$ does not
depend on $i$ since the inclusion $K_0(\A^0)\to K_0(\A)$ sends the
class of the simple object $S_i'$ of $\A_{x_i}^0$ to
$\sum_{j=1}^{p(x)}[\tau^{j}S_i]$; see Lemma~\ref{le:simple}. Here, we
use that $S_i'$ is a simple object of degree one; so its class in
$K_0(\A^0)$ is independent of $i$ by Lemma~\ref{le:homog_except}. Thus
the object $L_i^{(p_i)}$ does not depend on $i$ by
Proposition~\ref{pr:HR1}, and we denote it by $L'$.

The construction of the $L_i^{(j)}$ is parallel to the construction in
the proof of Theorem~\ref{th:tiltobj}, and we refer to the commutative
diagram \eqref{eq:tilt} illustrating it.

Next observe that each object $L_i^{(j)}$ depends only on $L$ and $x_i$
because its class has this property and the object is exceptional.
Thus the object
$$T_L=L\oplus L'\oplus(L_1^{(1)}\oplus\dots\oplus
L_1^{(p_1-1)})\oplus \dots \oplus(L_n^{(1)}\oplus\dots\oplus
L_n^{(p_n-1)})$$ depends up to isomorphism only on $L$. In
particular, the object equals up to  equivalence the tilting object
constructed in the proof of Theorem~\ref{th:tiltobj}, because for
each pair of line bundles $M,N$ there exists an equivalence
$\A\xto{\sim}\A$ taking $M$ to $N$, by Proposition~\ref{pr:shift}.
It follows that $T_L$ is a tilting object for $\A$.

(1) Choosing a basis $b_0,b_1$ of $\Hom_\A(L,L')$, we obtain
rational points $\la_i=[\la_{i0}:\la_{i1}]$ in $\bbP^1_k$ such that
$\p_i=\lambda_{i0} b_1 - \lambda_{i1} b_0$, where $\p_i$ is the
composite of the morphisms in \eqref{eq:tilt_chain}. Note that each
$\p_i$ depends on the choice of the $\psi_j$, but it is unique up to
a non-zero scalar.  It follows that $\End_\A(T_L)$ is isomorphic to
$C(\bfp,\bflambda)$, where $\bflambda=(\la_1,\dots,\la_n)$.  Any
other choice of the basis $b_0,b_1$ gives another set of parameters
$\bflambda'$ and a linear transformation $\sigma\in\PGL(2,k)$ such
that $\la_i'=\sigma(\la_i)$ for each $i$. Thus the weight function
$w_{\bfp,\bflambda}$ is unique up to  equivalence.

(2) Apply Proposition~\ref{pr:shift}.
\end{proof}

\begin{rem}
There is an analogue of Theorem~\ref{th:canon_tiltobj} with the canonical algebra
$C(\bfp,\bflambda)$ replaced by the squid algebra $\Sq(\bfp,\bflambda)$.
\end{rem}

The following example gives an explicit description of the tilting
object for the category of coherent sheaves on a weighted projective
line; see \cite[Proposition~4.1]{GL1987}.

\begin{exm}\label{ex:can}
Let $\bbX =(\bbP^1_k,\bflambda,\bfp)$ be a weighted projective
line. Then $T_\Oc=\bigoplus_{0\le\ox\le\oc}\Oc(\ox)$ is the tilting
object for $\coh\bbX$ that is associated to the line bundle $\Oc$. The
endomorphism algebra is isomorphic to the canonical algebra
$C(\bfp,\bflambda)$.
\end{exm}

It is now a consequence of Theorem~\ref{th:canon_tiltobj} that a
$k$-linear abelian category satisfying (H1)--(H5) is determined up
to equivalence by its associated weight function.

\begin{cor}\label{co:wpl_wf}
For two weighted projective lines $\bbX =(\bbP^1_k,\bflambda,\bfp)$
and $\bbX' =(\bbP^1_k,\bflambda',\bfp')$ over a field $k$, the
following are equivalent:
\begin{enumerate}
\item The weight functions $w_{\bfp,\bflambda}$ and
$w_{\bfp',\bflambda'}$ are equivalent.
\item The algebras $C(\bfp,\bflambda)$ and $C(\bfp',\bflambda')$ are isomorphic.
\item The categories $\coh\bbX$ and $\coh\bbX'$ are equivalent.
\end{enumerate}
\end{cor}
\begin{proof}
(1) $\Rightarrow$ (2): Suppose that $w_{\bfp,\bflambda}$ and
$w_{\bfp',\bflambda'}$ are equivalent via some linear transformation
$\sigma=\smatrix{\s_{00}&\s_{01}\\ \s_{10}&\s_{11}}$ in $\PGL(2,k)$.
Thus we may assume that the points of $\bflambda$ and $\bflambda'$
are related via $\sigma (\la_i)=\la_i'$ for $1\le i\le n$. The
algebra $C(\bfp,\bflambda)$ is generated by a collection of arrows
$b_0,b_1,x_i$; analogously $C(\bfp',\bflambda')$ is generated by
arrows $b'_0,b'_1,x'_i$.  We obtain an isomorphism $f\colon
C(\bfp,\bflambda)\xto{\sim} C(\bfp',\bflambda')$ by defining
$f(b_0)=\s_{11}b'_0-\s_{01}b'_1$,  $f(b_1)=\s_{00}b'_1-\s_{10}b'_0$
and $f(x_i)=x'_i$ ($1\le i\le n$).

(2) $\Rightarrow$ (3): The category $\coh\bbX$ admits a tilting object $T$
    with endomorphism algebra $C(\bfp,\bflambda)$, which is obtained from
the tilting object
$$\Oc\oplus \Oc(\oc)\oplus(S_1^{[1]}\oplus\dots\oplus
S_1^{[p_1-1]})\oplus \dots \oplus(S_n^{[1]}\oplus\dots\oplus
S_n^{[p_n-1]})$$ by modifying it as in the proof of
Theorem~\ref{th:tiltobj}.  Analogously, $\coh\bbX'$ admits a tilting
object $T'$ with endomorphism algebra $C(\bfp',\bflambda')$. If both
algebras are isomorphic, then Proposition~\ref{pr:der_equiv} implies
that $\coh\bbX$ and $\coh\bbX'$ are equivalent.

(3) $\Rightarrow$ (1): Suppose that there is an equivalence
$F\colon\coh\bbX\xto{\sim}\coh\bbX'$. We apply
Theorem~\ref{th:canon_tiltobj}. Thus the functor $F$ takes for any
line bundle $L$ of $\coh\bbX$ the canonically defined tilting object
$T_L$ to $T_{FL}$. The associated weight function for $T_L$ is
equivalent to $w_{\bfp,\bflambda}$, whereas for $T_{FL}$ it is
equivalent to $w_{\bfp',\bflambda'}$. It follows that
$w_{\bfp,\bflambda}$ and $w_{\bfp',\bflambda'}$ are equivalent.
\end{proof}

\section{Further topics}

In this section we list a few topics which have attracted interest in
the past, and which are areas of present research. The lists of papers
is certainly not complete and we refer to the references in the listed
papers for more information.

\subsection*{1}
The classification of indecomposable vector bundles on weighted
projective lines: the trichotomy `domestic/tubular/wild'
based on the Euler characteristic.
\medskip

\noindent{\small W. Geigle\ and\ H. Lenzing, A class of weighted
projective curves arising in representation theory of
finite-dimensional algebras, in {\it Singularities, representation of
algebras, and vector bundles (Lambrecht, 1985)}, 265--297, Lecture
Notes in Math., 1273, Springer, Berlin, 1987.
\smallskip

\noindent H. Lenzing\ and\ J. A. de la Pe\~na, Wild canonical
algebras, Math. Z. {\bf 224} (1997), no.~3, 403--425.
\smallskip

\noindent H. Lenzing, Hereditary categories, in {\it Handbook of
tilting theory}, 105--146, Cambridge Univ. Press, Cambridge, 2007.}

\subsection*{2}
Noncommutative curves of genus zero: the study of weighted projective
lines over arbitrary base fields.
\medskip

\noindent{\small H. Lenzing, Representations of finite dimensional
algebras and singularity theory, in {\it Trends in ring theory
(Miskolc, 1996)}, 71--97, Amer. Math. Soc., Providence, RI, 1998.
\smallskip

\noindent H. Lenzing\ and\ J. A. de la Pe\~na,
Concealed-canonical algebras and separating tubular families,
Proc. London Math. Soc. (3) {\bf 78} (1999), no.~3, 513--540.
\smallskip

\noindent D. Kussin, Noncommutative curves of genus zero:
related to finite dimensional algebras, Mem. Amer. Math. Soc. {\bf
201} (2009), no.~942, x+128 pp.}

\subsection*{3}
Graded singularities: the study of weighted projective lines in terms
of graded singularities (maximal Cohen-Macaulay modules, vector
bundles, the triangulated category of singularities in the sense of
Buchweitz and Orlov).
\medskip

\noindent{\small H. Kajiura, K. Saito\ and\ A. Takahashi, Matrix
factorization and representations of quivers. II. Type $ADE$ case,
Adv. Math. {\bf 211} (2007), no.~1, 327--362.
\smallskip

\noindent H. Lenzing\ and\ J. A. de la Pe\~na, Extended canonical
algebras and Fuchsian singularities, arXiv:math/0611532, Math. Z.,
to appear.
\smallskip

\noindent H. Lenzing\ and\ J. A. de la Pe\~na, Spectral analysis
of finite dimensional algebras and singularities, in {\it Trends in
representation theory of algebras and related topics}, 541--588,
Eur. Math. Soc., Z\"urich, 2008.
\smallskip

\noindent A. Takahashi, Weighted projective lines associated to
regular systems of weights of dual type, arXiv:0711.3907, Adv. Stud.
Pure Math., to appear. }

\subsection*{4}
Kac's theorem, Hall algebras: the theorem characterizes the dimension types of
indecomposable coherent sheaves over weighted projective lines in
terms of loop algebras of Kac-Moody Lie algebras; the proof uses Hall algebras.
\medskip

\noindent{\small W. Crawley-Boevey, Kac's Theorem for weighted
projective lines, arXiv:math/0512078, J. Eur. Math. Soc., to appear.
\smallskip

\noindent W. Crawley-Boevey, Quiver algebras, weighted
projective lines, and the Deligne-Simpson problem, in {\it
International Congress of Mathematicians. Vol. II}, 117--129,
Eur. Math. Soc., Z\"urich, 2006.
\smallskip

\noindent O. Schiffmann, Noncommutative projective curves and quantum
loop algebras, Duke Math. J. {\bf 121} (2004), no.~1, 113--168.}

\printindex

\end{document}